\documentclass[10pt]{article}
\pdfoutput=1
\usepackage[utf8]{inputenc}
\usepackage[english]{babel}
\usepackage[T1]{fontenc}
\usepackage{framed}
\usepackage{hyperref}
\usepackage{lmodern}
\usepackage{amsmath}
\usepackage{amsfonts}
\usepackage{amssymb} 
\usepackage{amsthm}
\usepackage{mdframed}
\usepackage{mathtools}
\usepackage{tikz-cd}
\usepackage{tikzit}
\usepackage{cmll}
\usepackage{graphicx}
\usepackage{stmaryrd}
\usepackage{adjustbox}
\usepackage{xcolor}
\usepackage{epigraph}
\usepackage{ebproof}
\usepackage{geometry}
\usepackage{enumitem}
\hypersetup{
    colorlinks = true,
    linkcolor = blue,
    citecolor = blue,
    urlcolor = red
    }
    
\usepackage[
        noabbrev,
        capitalise,
        nameinlink,
    ]
    {cleveref}
   
\usepackage{etoolbox}
\apptocmd{\thebibliography}{\raggedright}{}{}
\usepackage{tikz}
\usepackage{lipsum}

\tikzstyle{blanc}=[fill=white, draw=black, shape=circle]
\tikzstyle{dot}=[draw=black, shape=circle, fill=black, minimum width=0.1cm]
\tikzstyle{box}=[fill=none, draw=black, shape=rectangle, minimum width=4cm, minimum height=2cm]
\tikzstyle{box-less-wide}=[fill=none, draw=black, shape=rectangle, minimum width=2.5cm, minimum height=1cm]
\tikzstyle{little box}=[fill=white, draw=black, shape=rectangle]
\tikzstyle{little box red}=[fill=white, draw=red, shape=rectangle, line width=0.35mm]
\tikzstyle{middle box}=[fill=white, draw=black, shape=rectangle, minimum height=1cm, minimum width=2cm]
\tikzstyle{middlebox2}=[fill=white, draw=black, shape=rectangle, minimum height=1cm, minimum width=4.5cm]
\tikzstyle{7}=[fill=white, draw=black, shape=rectangle, minimum height=1cm, minimum width=7cm]
\tikzstyle{box1}=[fill=none, draw=black, shape=rectangle, minimum width=4.2cm, minimum height=1.5cm]
\tikzstyle{box2}=[fill=none, draw=black, shape=rectangle, minimum width=6cm, minimum height=1.5cm]
\tikzstyle{box3}=[fill=white, draw=black, shape=rectangle, minimum width=3.2cm, minimum height=1.5cm]
\tikzstyle{box4}=[fill=white, draw=black, shape=rectangle, minimum width=7cm, minimum height=1.5cm]

\tikzstyle{dirgreen}=[->, fill=none, draw=green]
\tikzstyle{red}=[-, draw=red, line width=0.35mm]
\tikzstyle{green}=[-, draw=green]
\tikzstyle{redn}=[-, fill=none, draw=red]
\tikzstyle{magenta}=[-, draw=magenta]
\tikzstyle{blue}=[-, draw=blue]
\tikzstyle{dirmagenta}=[->, draw=magenta]
\tikzstyle{dirblue}=[->, draw=blue]
\tikzstyle{dirred}=[->, draw=red]

\newtheorem{theorem}{Theorem}
\newtheorem{corollary}[theorem]{Corollary}
\newtheorem{proposition}[theorem]{Proposition}
\newtheorem{lemma}[theorem]{Lemma}
\newtheorem{definition}[theorem]{Definition}
\newtheorem{remark}[theorem]{Remark}
\newtheorem{example}[theorem]{Example}

\crefname{subsection}{subsection}{subsections}
\Crefname{subsection}{Subsection}{Subsections}

\usepackage[nottoc]{tocbibind}
  
\begin{document}

\title{A bialgebraic characterization of symmetric powers in $\mathbb{Q}_{\ge 0}$-linear symmetric monoidal categories}
\date{}
\author{Jean-Baptiste Vienney}

\maketitle

\abstract{In any symmetric monoidal category, the $n$-th (co)equalizer symmetric power of an object $A$ is the (co)equalizer of all the permutations from $A^{\otimes n}$ to itself. If the symmetric monoidal category is $\mathbb{Q}_{\ge 0}$-linear, that is, enriched over $\mathbb{Q}_{\ge 0}$-modules, the notions of $n$-th equalizer symmetric power and $n$-th coequalizer symmetric power are equivalent. In this context, the $n$-th symmetric power of $A$ can be described as the intermediate object $A_n$ in a splitting of the idempotent $\frac{1}{n!}\underset{\sigma \in S_n}{\sum}\sigma\colon A^{\otimes n} \rightarrow A^{\otimes n}$. We define a permutation splitting as a countable family of such splittings.

The main goal of this paper is to prove two theorems. The first theorem exhibits in any $\mathbb{Q}_{\ge 0}$-linear symmetric monoidal category a bijection between operations making a graded object $(A_n)_{n \ge 0}$ into a permutation splitting and operations making this graded object into a bialgebraic structure that we call a binomial bimonoid. Binomial bimonoids can be defined in any additive symmetric monoidal category. The second theorem shows that, in any $\mathbb{Q}_{\ge 0}$-linear symmetric monoidal category, the biassociativity and bicommutativity axioms may be omitted from the definition of a binomial bimonoid. 

We then show that being a binomial bimonoid in a $\mathbb{Q}_{\ge 0}$-linear symmetric monoidal category is a property: two binomial bimonoids are isomorphic whenever their underlying graded objects are isomorphic. This result does not extend to arbitrary additive symmetric monoidal categories since both the one-variable polynomial algebra and the one-variable divided power polynomial algebra over a field $k$ of positive characteristic are non-isomorphic binomial $k$-bialgebras with isomorphic underlying $\mathbb{N}$-graded vector spaces.}
\tableofcontents
\section{Introduction}
\paragraph{Symmetric powers in $\mathbb{Q}_{\ge 0}$-linear symmetric monoidal categories.} In any symmetric monoidal category $(\mathsf{C},\otimes,I)$, we can define for every permutation $\sigma \in S_n$ a natural transformation
\begin{equation*}
\sigma\colon A^{\otimes n} \rightarrow A^{\otimes n}.
\end{equation*}
Given any object $A \in \mathsf{C}$, we can then look at either the equalizer or the coequalizer of the diagram
\begin{equation*}
% https://tikzcd.yichuanshen.de/#N4Igdg9gJgpgziAXAbVABwnAlgFyxMJZABgBpiBdUkANwEMAbAVxiRAEEA9YAHR4jwBbeAAIwAXxDjS6TLnyEUAJnJVajFmy69+Q0RKlqYUAObwioAGYAnCIKRkQOCEgCM1OAAsslnEgC0KurMrIggfNgmgnRSMiA2dg7Uzm4e3r5IQfQhbHxQAgjiFOJAA
\begin{tikzcd}
A^{\otimes n} \arrow[rr, "\sigma", shift left=2] \arrow[rr, "\dots", shift right=2] &  & A^{\otimes n}
\end{tikzcd}
\end{equation*}
where there is one arrow for each $\sigma \in S_n$. We call the equalizer of this diagram the $n$-th coequalizer symmetric power of $A$ and we call the coequalizer of this diagram the $n$-th equalizer symmetric power of $A$. In this paper, we will be interested in the case where the symmetric monoidal category $(\mathsf{C},\otimes,I)$ is a $\mathbb{Q}_{\ge 0}$-linear symmetric monoidal category. This means that the category $\mathsf{C}$ is enriched over $\mathbb{Q}_{\ge 0}$-modules and the tensor product is bilinear on morphisms. In such categories, an $n$-th coequalizer symmetric power of an object becomes the same as an $n$-th equalizer symmetric power of this object. Thus, we can just say an $n$-th symmetric power. Moreover, in a $\mathbb{Q}_{\ge 0}$-linear symmetric monoidal category, we can describe the $n$-th symmetric power of $A$ as an intermediate object $A_n$ in a splitting of the idempotent
\begin{equation} \label{the-idempotent}
\frac{1}{n!}\underset{\sigma \in S_n}{\sum}\sigma\colon A^{\otimes n} \rightarrow A^{\otimes n}.
\end{equation}
The definition of an $n$-th (co)equalizer symmetric power of an object is introduced in \Cref{THREE}. In the same subsection, we prove that the three above definitions of an $n$-th symmetric power are equivalent in any $\mathbb{Q}_{\ge 0}$-linear symmetric monoidal category.
\paragraph{The example of $(k_n[x])_{n \ge 0}$.} In this paper, we will be interested not in a single $n$-th symmetric power of an object, but in families $(A_n)_{n \ge 0}$ where each object $A_n$ is an $n$-th symmetric power of $A_1$. An important example is given by $(k_n[x])_{n \ge 0}$ where $k$ is a field of characteristic $0$ and $k_n[x]$ is the space of all polynomials of the form $a_n x^n$ where $a_n \in k$. The family of objects $(k_n[x])_{n \ge 0} \in (\mathsf{Vec}_k,\otimes,k)$ can be made into a special kind of $\mathbb{N}$-graded bialgebra that we call a binomial bialgebra. More generally, we define the notion of a binomial bimonoid in any additive symmetric monoidal category, that is, a symmetric monoidal category enriched over commutative monoids and such that the tensor product is biadditive on morphisms. We use the terminology ``binomial bialgebra'' when the additive symmetric monoidal category is the category of modules over a commutative rig.

In order to express the axioms in the definition of a binomial bimonoid, we will make use of string diagrams. String diagrams and bicommutative bimonoids are introduced in \cref{SEC-TWO-ONE}. We then focus on the special properties of the bicommutative bialgebra $k[x]$ in \cref{SEC-TWO-TWO}. In this subsection, the notion of a binomial bimonoid is introduced step by step and it is shown that $k[x]$ is an example.

\paragraph{Other examples.} The typical example of a family of symmetric powers in a $\mathbb{Q}_{\ge 0}$-linear symmetric monoidal category is slightly more general that $(k_n[x])_{n \ge 0}$ where $k$ is a field of characteristic $0$. First, we can replace $k$ with any commutative $\mathbb{Q}_{\ge 0}$-algebra, that is, a commutative rig $k$\footnote{We define a \emph{rig} $k$ as a set equipped with two binary operations $+$ and $\cdot$, and with two elements $0,1 \in k$, such that $(k,+,0)$ is a commutative monoid, $(k,\cdot,1)$ is a monoid, $\cdot$ distributes over $+$ and $0\lambda=\lambda0=0$ for every $\lambda \in k$. A \emph{commutative rig} is a rig $k$ such that the monoid $(k,\cdot,1)$ is commutative.} together with a rig homomorphism $\rho\colon\mathbb{Q}_{\ge 0} \rightarrow k$. If $k$ is any commutative $\mathbb{Q}_{\ge 0}$-algebra, then $(k_n[x])_{n \ge 0}$ is still a binomial bialgebra. More importantly, we can generalize this example by taking a coordinate-free approach. Let $k$ be a commutative $\mathbb{Q}_{\ge 0}$-algebra and let $A$ be a $k$-module. For every $n \ge 0$, we define $S^nA\colon=A^{\otimes n}/\sim$ where $\sim$ is the smallest congruence $\sim$ on $A^{\otimes n}$ such that $a_1 \otimes \dots \otimes a_n \sim a_{\sigma(1)} \otimes \dots \otimes a_{\sigma(n)}$ for all $a_1,\dots,a_n \in A$ and $\sigma \in S_n$. We obtain that $(S^nA)_{n \ge 0}$ is a binomial bimonoid. If $A \simeq k$, then the binomial bimonoids $(S^nA)_{n \ge 0}$ and $(k_n[x])_{n \ge 0}$ are isomorphic as $\mathbb{N}$-graded bimonoids. There are similar constructions in other categories that are covered in \cref{THREE,subsec:main-th}.

\paragraph{Main results.} \Cref{subsec:main-th} is devoted to the statement of the two main theorems in this paper. We first introduce in \cref{def-perm-splitting} a permutation splitting as any family of objects $(A_n)_{n \ge 0}$ in a $\mathbb{Q}_{\ge 0}$-linear symmetric monoidal category, together with morphisms $(r_n\colon A_1^{\otimes n} \rightarrow A_n)_{n \ge 0}$ and $(s_n\colon A_n \rightarrow A_1^{\otimes n})_{n \ge 0}$ such that $r_1=s_1=\mathsf{id}_{A_1}$ and the following equations hold for every $n \in \mathbb{N} \backslash \{1\}$:
\begin{equation*}
r_n;s_n=\frac{1}{n!}\underset{\sigma \in S_n}{\sum}\sigma
\end{equation*}
and
\begin{equation*} \label{perm-split-equation-2}
s_n;r_n=\mathsf{id}_{A_n}.
\end{equation*}
Note that these two equations then hold for every $n \in \mathbb{N}$. A permutation splitting is thus a family of objects constituted of an $n$-th symmetric power $A_n$ of $A_1$ for every $n \in \mathbb{N}$, together with morphisms exhibiting each $A_n$ as an intermediate object in a splitting of the idempotent (\ref{the-idempotent}), and such that the splitting of the idempotent (\ref{the-idempotent}) is trivial when $n=1$.

Our two main theorems will be concerned with how in any $\mathbb{Q}_{\ge 0}$-linear symmetric monoidal category, the notion of a permutation splitting is equivalent to the one of a binomial bimonoid and how in any $\mathbb{Q}_{\ge 0}$-linear symmetric monoidal category, the list of axioms in the definition of a binomial bimonoid can be shortened. The notion of a binomial bimonoid, first introduced in \cref{SEC-TWO-TWO}, is recalled in \cref{def:bin-bimonoid}. The $12$ axioms in the definition of a binomial bimonoid are listed in \cref{fig:bialg-axiom}. Starting from a permutation splitting, we can rearrange the maps $r_n$ and $s_n$ into a unit map
\begin{equation*}
\eta\colon I \rightarrow A_0,
\end{equation*}
a counit map
\begin{equation*}
\epsilon\colon A_0 \rightarrow I,
\end{equation*}
multiplication maps
\begin{equation*}
\nabla_{n,p}\colon A_n \otimes A_p \rightarrow A_{n+p}
\end{equation*}
(for all $n,p \ge 0$) and comultiplication maps
\begin{equation*}
\Delta_{n,p}\colon A_{n+p} \rightarrow A_n \otimes A_p
\end{equation*}
(for all $n,p \ge 0$) which satisfy all the equations required to make $(A_n)_{n \ge 0}$ into a binomial bimonoid. Conversely, starting from a binomial bimonoid, we can rearrange the maps $\eta,\epsilon,\nabla_{n,p}$ and $\Delta_{n,p}$ into maps $r_n$ and $s_n$ making $(A_n)_{n \ge 0}$ into a permutation splitting. \cref{main-theorem} states that these two transformations are inverse of each other so that we have a bijection between operations making $(A_n)_{n \ge 0}$ into a permutation splitting and operations making $(A_n)_{n \ge 0}$ into a binomial bimonoid. \cref{main-theorem-2} then states that the associativity, coassociativity, commutativity and cocommutativity axioms can be removed from the definition of a binomial bimonoid in a $\mathbb{Q}_{\ge 0}$-linear symmetric monoidal category. The resulting axioms are listed in \cref{fig:bialg-axiom-short}.

\paragraph{Outline of the rest of the paper.} The rest of the paper, that is, \cref{SIX} to \cref{SEC:LAST}, is devoted to proving the main theorems and presenting some additional material. In \cref{SIX}, we introduce permutation averages, which are natural transformations obtained as convex sums of permutations, of the form $\frac{1}{|M|}\underset{\sigma \in M}{\sum}\sigma\colon A^{\otimes n} \rightarrow A^{\otimes n}$ for some multiset $M \in \mathcal{M}(S_n)$. Permutation averages will be useful in later sections. \cref{SEC:FOUR} is the longest section of the paper and contains the proofs of \cref{main-theorem,main-theorem-2}.

In \cref{SEC:FIVE}, we are interested in whether being a binomial bimonoid is a property or a structure. For this, we first need to introduce morphisms of binomial bimonoids and of permutation splittings. This is done in \cref{SUBSEC:FIVE-ONE}. In this subsection, \cref{main-theorem} is somehow improved from a bijection at the level of objects to an isomorphism of categories. We show that starting from any $\mathbb{Q}_{\ge 0}$-linear symmetric monoidal category, the categories of binomial bimonoids and permutation splittings are isomorphic, and these two categories are isomorphic to two other categories whose objects are still either binomial bimonoids or permutation splittings, but whose morphisms are simpler. We then obtain in \cref{SUBSEC:FIVE-TWO} that in a $\mathbb{Q}_{\ge 0}$-linear symmetric monoidal category, being a binomial bimonoid is a property: if $(A_n)_{n \ge 0}$ and $(B_n)_{n \ge 0}$ are two binomial bimonoids whose underlying objects are isomorphic in $\mathsf{C}^\mathbb{N}$, then these two binomial bimonoids are isomorphic. We have in fact more: it suffices that $A_1$ and $B_1$ are isomorphic for the two binomial bimonoids to be isomorphic. In \cref{SUBSEC:FIVE-THREE}, we show that being a binomial bimonoid in an arbitrary additive symmetric monoidal category is no longer a property but a structure. If $k$ is a field of positive characteristic, then there exist two nonisomorphic binomial bimonoids $(A_n)_{n \ge 0}$ and $(B_n)_{n \ge 0}$ in $\mathsf{Vec}_k$ such that $(A_n)_{n \ge 0}$ and $(B_n)_{n \ge 0}$ are isomorphic in $\mathsf{Vec}_k^\mathbb{N}$. 

Finally, in \cref{SEC:LAST}, we compute binomial bimonoids from permutation splittings in three $\mathbb{Q}_{\ge 0}$-linear symmetric monoidal categories: the category of modules over a commutative $\mathbb{Q}_{\ge 0}$-algebra, the category of sets and relations and the category of suplattices.
\paragraph{Related works.} This paper originates from the work of Lemay and the author on graded differential linear logic \cite{GradedDiffCat}. A model of $\mathbb{N}$-graded differential linear logic (more precisely, the opposite category of such a model) is given by an additive symmetric monoidal category $(\mathsf{C},\otimes,I)$ together with a family of functors $S^n\colon\mathsf{C} \rightarrow \mathsf{C}$ and some natural transformations which in particular make every graded object $(S^nA)_{n \ge 0}$ into a bicommutative $\mathbb{N}$-graded bimonoid (we call graded object any family of objects $(X_n)_{n \ge 0} \in \mathsf{C}^\mathbb{N}$ for some category $\mathsf{C}$). We obtain a model of $\mathbb{N}$-graded differential linear logic by chosing $S^n$ to be the $n$-th symmetric power functor in any $\mathbb{Q}_{\ge 0}$-linear symmetric monoidal category with symmetric powers and adapting to these categories the usual structure of a $k$-bialgebra on a symmetric algebra $\underset{n \ge 0}{\bigoplus}S^nA$ where $A$ is a module over a commutative ring $k$. We wanted to know which conditions on a $\mathbb{Q}_{\ge 0}$-linear model of $\mathbb{N}$-graded differential linear logic can be imposed to ensure that the functors $S^n$ are the $n$-th symmetric power functors and that the graded bialgebra structure is the usual one adapted from symmetric algebras. The present paper answers a similar question: Which conditions can we impose on a bicommutative $\mathbb{N}$-graded bimonoid in a $\mathbb{Q}_{\ge 0}$-linear symmetric monoidal category to ensure that this is an $\mathbb{N}$-graded bimonoid defined on symmetric powers in the usual way? This question was previously answered in \cite{AxelMag} in the case where the $\mathbb{Q}_{\ge 0}$-linear symmetric monoidal category is $(\mathsf{Vec}_k,\otimes,k)$ for $k$ a field of characteristic $0$. In the present paper, we work in the more general context of an arbitrary $\mathbb{Q}_{\ge 0}$-linear symmetric monoidal category, and we also give more detailled results---even when applied to the $\mathbb{Q}_{\ge 0}$-linear symmetric monoidal category $(\mathsf{Vec}_k,\otimes,k)$ for $k$ a field of characteristic $0$---in the form of \cref{main-theorem}  and \cref{main-theorem-2}.

The notion of a divided power polynomial algebra and its structure of an $\mathbb{N}$-graded bialgebra, which appear in \cref{SUBSEC:FIVE-THREE}, are classical concepts which can be found for instance as Example 2.6 in \cite{ABE}. For background on symmetric monoidal categories, see \cite{MAC}. The notion of a symmetric algebra, which is the countable coproduct $\underset{n \ge 0}{\bigoplus}S^nA$ where the $S^nA$ are the $n$-th coequalizer symmetric powers of a module $A$ over a commutative rig $k$, is well-known in abstract algebra, particularly when $k$ is a commutative ring, see for example \cite{BOUR}.
\paragraph{Conventions and notations.}
\begin{itemize}
\item The composition will mainly be written in the diagrammatic order, that is, if $f\colon A \rightarrow B$ and $g\colon B \rightarrow C$ are morphisms in a category, then the composite is written $f;g\colon A \rightarrow C$. However, we will also sometimes use the classical order, so that the previous composite is written $g \circ f\colon A \rightarrow C$. 
\item We will ignore the coherence isomorphisms and work as if the symmetric monoidal categories were strict monoidal, that is, if $(\mathsf{C},\otimes,I)$ is a symmetric monoidal category, we will assume that the unitors and associators are equalities of functors:
\begin{equation*}
I \otimes - = \mathsf{id}_{\mathsf{C}} = - \otimes I\colon\mathsf{C} \rightarrow \mathsf{C},
\end{equation*}
\begin{equation*}
(- \otimes -) \otimes - = - \otimes (- \otimes -)\colon\mathsf{C}^3 \rightarrow \mathsf{C}.
\end{equation*}
In particular, if $k$ is a commutative rig, $M,N,P$ are $k$-modules, $\lambda \in k$ and $(m,n,p) \in M \times N \times P$, we will write
\begin{equation*}
\lambda \otimes m=m \otimes \lambda=\lambda m \in M
\end{equation*}
and
\begin{equation*}
m \otimes n \otimes p =(m \otimes n) \otimes p=m \otimes (n \otimes p).
\end{equation*}
\item The tensor product on morphisms will have higher precedence than the composition which means that an expression such as
\begin{equation*}
f \otimes g;h \otimes i
\end{equation*}
must be interpreted as
\begin{equation*}
(f \otimes g);(h \otimes i).
\end{equation*}
\item If $n \le p \in \mathbb{N}$, we define $[n,p]:=\{n,n+1,\dots,p\}$.
\item $S_n$ is the symmetric group on $n$ elements. We recall that both $S_1$ and $S_0$ are the trivial group.
\item We will write $\gamma_{A,B}\colon A \otimes B \rightarrow B \otimes A$ for the exchange in a symmetric monoidal category $(\mathsf{C},\otimes,I)$.
\item Let $(\mathsf{C},\otimes,I)$ be a symmetric monoidal category. By using the exchange, we can define for all $n \ge 0$ and $\sigma \in S_n$, a natural transformation
\begin{equation*}
\sigma\colon A^{\otimes n} \rightarrow A^{\otimes n}
\end{equation*}
that we will denote by the same symbol. These natural transformations are defined so that if $\sigma,\tau \in S_n$, then the natural transformation associated with $\sigma;\tau \in S_n$ coincides with the composite of the two natural transformations $\sigma,\tau\colon A^{\otimes n} \rightarrow A^{\otimes n}$.\footnote{Let $n \ge 0$. Recall that the adjacent transposition $(a,a+1) \in S_n$ where $1 \le a \le n-1$ is the permutation defined by 
\begin{equation*}
(a,a+1)(k)=
\left\{
\begin{aligned}
k &~\text{if } 1 \le k \le a-1 \text{ or } a+2\le k \le n, \\
a+1 &~\text{if } k=a, \\
a &~\text{if } k=a+1.
\end{aligned}
\right.
\end{equation*}
If $\sigma \in S_n$, we can decompose $\sigma$ as a product of adjacent transpositions: $\sigma=(a_1,a_1+1);\dots;(a_p,a_p+1)$ where $a_1,\dots,a_p \in [1,n]$. We then define the natural transformation $\sigma\colon A^{\otimes n}\rightarrow A^{\otimes n}$ as 
\begin{equation*}
\sigma=\mathsf{id}_{A^{\otimes (a_1-1)}} \otimes \gamma_{A,A} \otimes \mathsf{id}_{A^{\otimes(n-(a_1+1))}};\dots;\mathsf{id}_{A^{\otimes(a_p-1)}} \otimes \gamma_{A,A} \otimes \mathsf{id}_{A^{\otimes(n-(a_p+1))}}.
\end{equation*}
According to MacLane's coherence theorem for symmetric monoidal categories, the natural transformation $\sigma\colon A^{\otimes n} \rightarrow A^{\otimes n}$ does not depend on the choice of the decomposition of $\sigma$ as a product of adjacent transpositions.}
\item The symbol $*$ denotes a one-point set.
\item RHS means ``Right-Hand Side'' and LHS means ``Left-Hand Side''.
\item For all positive integers $n$ and $p$, we denote by $\mathbb{N}^{n \times p}$ the set of all matrices with nonnegative integer coefficients, $n$ rows and $p$ columns. We denote a matrix in $\mathbb{N}^{n \times p}$ by an expression of the form $(m_i^j)$ where $1 \le i \le n$ and $1 \le j \le p$. 
\item The commutative rig of nonnegative rational numbers is denoted by $\mathbb{Q}_{\ge 0}$. 
\end{itemize}
\paragraph{Acknowldegments.} I would like to thank John Baez, Richard Blute, Vikraman Choudhury, Sacha Ikonicoff, Shin-ya Katsumata, Jean-Simon Lemay, Thomas Vandeven, Geoff Vooys and the late Philip Scott for useful discussions and their support of this research. I would like to thank the anonymous referees for their very helpful reports, which significantly improved the exposition. This work was financially supported by Aix-Marseille University and NSERC under the grants awarded to Richard Blute and the late Philip Scott.
\section{String diagrams and the bialgebraic structure of $k[x]$} \label{SEC-TWO}
In this section, we will introduce string diagrams and explain using them what is a binomial bimonoid, one of the main notions in this paper. We will also introduce an important example of a binomial bimonoid which is $k[x]$ for $k$ a field of characteristic $0$ (in fact, as explained at the end of this section, $k[x]$ is a binomial bimonoid for any field $k$ but the case of a field of characteristic $0$ will be more important to us). The style of this section is somewhat informal. Later in the paper, we recall the notion of a binomial bimonoid in a concise and formal manner, prior to stating our main theorems.
\subsection{String diagrams and bicommutative bimonoids} \label{SEC-TWO-ONE}
\paragraph{i. String diagrams for symmetric monoidal categories.}
We will make heavy use of string diagrams to represent morphisms in a symmetric monoidal category $(\mathsf{C},\otimes,I)$. Strings represent objects and morphisms are represented by nodes. A morphism $f$ from $A$ to $B$ is represented as
\begin{equation*}
\resizebox{!}{1.2cm}{\input{PRES1.tikz}}\quad.
\end{equation*}
If $f\colon A \rightarrow B$ and $g\colon B \rightarrow C$ are two morphisms, the composite $f;g$ is then represented as follows:
\begin{equation*}
\resizebox{!}{1.2cm}{\input{PRES2.tikz}}\quad:=\quad\resizebox{!}{1.2cm}{\input{PRES2A.tikz}}\quad.
\end{equation*}
The identity on $A$ is represented as a simple string:
\begin{equation*}
\resizebox{!}{1.2cm}{\begin{tikzpicture}
	\begin{pgfonlayer}{nodelayer}
		\node [style=none] (1) at (0, 2) {};
		\node [style=none] (2) at (0, -2) {};
		\node [style=none] (3) at (0, 2.5) {$A$};
		\node [style=none] (4) at (0, -2.5) {$A$};
	\end{pgfonlayer}
	\begin{pgfonlayer}{edgelayer}
		\draw (1.center) to (2.center);
	\end{pgfonlayer}
\end{tikzpicture}
}\quad:=\quad\resizebox{!}{1.2cm}{\input{PRES3A.tikz}}\quad.
\end{equation*}
If $A_1,\dots,A_n$ are $n$ objects, the object $A_1 \otimes \dots \otimes A_n$ is represented as
\begin{equation*}
\resizebox{!}{0.17cm}{\begin{tikzpicture}
	\begin{pgfonlayer}{nodelayer}
		\node [style=none] (0) at (-1, 0) {$A_1$};
		\node [style=none] (1) at (1, 0) {$A_n$};
		\node [style=none] (2) at (0, 0) {$\dots$};
	\end{pgfonlayer}
\end{tikzpicture}
}\quad.
\end{equation*}
In particular, an empty tensor product, that is, the monoidal unit $I$, is represented as an empty node. The empty string diagram thus represents $\mathsf{id}_I$. 

We will represent a morphism $f\colon A_1 \otimes \dots \otimes A_n \rightarrow B_1 \otimes \dots \otimes B_p$ as follows:
\begin{equation*}
\resizebox{!}{1.2cm}{\input{PRES8A.tikz}}\quad.
\end{equation*}
If $f\colon A \rightarrow B$ and $g\colon C \rightarrow D$ are two morphisms, the tensor product $f \otimes g$ is represented as follows:
\begin{equation*}
\resizebox{!}{1.2cm}{\input{PRES5B.tikz}}\quad:=\quad\resizebox{!}{1.2cm}{\input{PRES5A.tikz}}\quad.
\end{equation*}
Finally, the exchange $\gamma\colon A \otimes B \rightarrow B \otimes A$ is represented as follows:
\begin{equation*}
\resizebox{!}{1.2cm}{\input{PRES6B.tikz}}\quad:=\quad\resizebox{!}{1.2cm}{\input{PRES6A.tikz}}\quad.
\end{equation*}
The naturality of the exchange is drawn as follows:
\begin{equation*}
\resizebox{!}{1.2cm}{\input{PRES7B.tikz}}\quad=\quad\resizebox{!}{1.2cm}{\input{PRES7A.tikz}}\quad.
\end{equation*}
where $f\colon A \rightarrow D$ and $g\colon B \rightarrow C$ are any two morphisms.
\paragraph{ii. String diagrams for additive symmetric monoidal categories.} \label{SDA} In the next sections of this paper, we will only be interested in symmetric monoidal categories which are additive, which means that each homset is a commutative monoid.\footnote{In fact, we will mainly be interested in symmetric monoidal categories which are $\mathbb{Q}_{\ge 0}$-linear, which means that each homset is a $\mathbb{Q}_{\ge 0}$-module, see \cref{def:k-linear}.} More precisely, an \emph{additive symmetric monoidal category} is a symmetric monoidal category $(\mathsf{C},\otimes,I)$ such that each homset $\mathsf{C}[A,B]$ is a commutative monoid, whose binary operation and unit we will denote additively, and such that the following identities are satisfied whenever they make sense:
\begin{equation} \label{eq:additive}
\begin{aligned}
(f+g);h &= (f;h) + (g;h), &\qquad f;(g+h) &= (f;g)+(f;h), \\
0;f &= f;0=0, &\qquad (f+g)\otimes h &= (f\otimes h) + (g\otimes h), \\
f\otimes(g+h) &= (f\otimes g) + (f\otimes h), &\qquad 0\otimes f &= f\otimes 0 = 0.
\end{aligned}
\end{equation}
The main example of an additive symmetric monoidal category is the category $(\mathsf{Vec}_k,\otimes,k)$ of vector spaces over a field $k$ or more generally the category $(\mathsf{Mod}_R,\otimes,R)$ of modules over a commutative rig $R$.\footnote{If $R$ is a rig, we define a \emph{module over $R$} as a commutative monoid $(A,+,0)$, together with an operation $R \times A \rightarrow A$ denoted by $(\lambda,a) \mapsto \lambda a$ such that $\lambda(a+b)=\lambda a+\lambda b$, $(\lambda+\mu)a=\lambda a+\mu a$, $(\lambda \mu)a=\lambda(\mu a)$, $1a=a$ and $0a=\lambda 0=0$.}

To obtain string diagrams for additive symmetric monoidal categories, we must add zero maps to string diagrams for symmetric monoidal categories, represented as
\begin{equation*}
\resizebox{!}{1.2cm}{\input{PRESADD1.tikz}}\quad,
\end{equation*}
and we must add sums of morphisms represented as
\begin{equation*}
\resizebox{!}{1.2cm}{\input{PRESADD2.tikz}}\quad:=\quad\resizebox{!}{1.2cm}{\input{PRES1.tikz}}\quad+\quad\resizebox{!}{1.2cm}{\input{PRESADD3.tikz}}\quad.
\end{equation*}
The equations \eqref{eq:additive} are represented as follows in string diagrams:
\begin{align*}
\resizebox{!}{1.2cm}{\input{PRESADD4.tikz}}\quad&=\quad\resizebox{!}{1.2cm}{\input{PRESADD5.tikz}}\quad+\quad\resizebox{!}{1.2cm}{\input{PRESADD6.tikz}}\quad, &\quad\resizebox{!}{1.2cm}{\input{PRESADD7.tikz}}\quad&=
\quad\resizebox{!}{1.2cm}{\input{PRESADD8.tikz}}\quad+\quad\resizebox{!}{1.2cm}{\input{PRESADD9.tikz}}\quad, \\[1em]
\resizebox{!}{1.2cm}{\input{PRESADD10.tikz}}\quad&=\quad\resizebox{!}{1.2cm}{\input{PRESADD11.tikz}}\quad=\quad\resizebox{!}{1.2cm}{\input{PRESADD12.tikz}}\quad, && 
\end{align*}
\begin{align*}
\resizebox{!}{1.2cm}{\input{PRESADD13.tikz}}\quad\resizebox{!}{1.2cm}{\input{PRESADD14.tikz}}\quad&=\quad\left(\quad\resizebox{!}{1.2cm}{\input{PRESADD15.tikz}}\quad\resizebox{!}{1.2cm}{\input{PRESADD14.tikz}}\quad\right)\quad+\quad\left(\quad\resizebox{!}{1.2cm}{\input{PRESADD16.tikz}}\quad\resizebox{!}{1.2cm}{\input{PRESADD14.tikz}}\quad\right)\quad, \\[1em]
\resizebox{!}{1.2cm}{\input{PRESADD15.tikz}}\quad\resizebox{!}{1.2cm}{\input{PRESADD17.tikz}}\quad&=\quad\left(\quad\resizebox{!}{1.2cm}{\input{PRESADD15.tikz}}\quad\resizebox{!}{1.2cm}{\input{PRESADD18.tikz}}\quad\right)\quad+\quad\left(\quad\resizebox{!}{1.2cm}{\input{PRESADD15.tikz}}\quad\resizebox{!}{1.2cm}{\input{PRESADD14.tikz}}\quad\right)\quad,
\end{align*}
\begin{align*}
\resizebox{!}{1.2cm}{\input{PRESADD19.tikz}}\quad\resizebox{!}{1.2cm}{\input{PRESADD20.tikz}}\quad&=\quad\resizebox{!}{1.2cm}{\input{PRESADD22.tikz}}\quad\resizebox{!}{1.2cm}{\input{PRESADD23.tikz}}\quad=\quad\resizebox{!}{1.2cm}{\input{PRESADD21.tikz}}\quad.&&
\end{align*}
\paragraph{iii. Commutative monoids.} String diagrams are very convenient to discuss commutative monoids in symmetric monoidal categories. A \emph{commutative monoid} in a symmetric monoidal category $(\mathsf{C},\otimes,I)$ is an object $A$, together with a multiplication map $\nabla\colon A \otimes A \rightarrow A$, in string diagrams
\begin{equation*}
\resizebox{!}{1.2cm}{\input{PRES9.tikz}}\quad,
\end{equation*}
and a unit map $\eta\colon I \rightarrow A$, in string diagrams
\begin{equation*}
\resizebox{!}{1cm}{\begin{tikzpicture}
	\begin{pgfonlayer}{nodelayer}
		\node [style=none] (13) at (0, -1.25) {$\scriptstyle A$};
		\node [style=none] (25) at (0, -0.75) {};
		\node [style=blanc] (26) at (0, 1.25) {};
	\end{pgfonlayer}
	\begin{pgfonlayer}{edgelayer}
		\draw (26) to (25.center);
	\end{pgfonlayer}
\end{tikzpicture}
}\quad,
\end{equation*}
such that the following equations are satisfied:
\begin{equation*}
\resizebox{!}{1.2cm}{\input{PRES11A.tikz}}~~=~~\resizebox{!}{1.2cm}{\input{PRES11B.tikz}}~~,\qquad
\resizebox{!}{1.2cm}{\input{PRES11C.tikz}}~~=~~\resizebox{!}{1.2cm}{\input{PRES11D.tikz}}~~,\qquad
\resizebox{!}{1.2cm}{\input{PRES11E.tikz}}.
\end{equation*}
The first equation is the unitality of the multiplication, the second one its associativity and the third one its commutativity.

In the symmetric monoidal category $(\mathsf{Set},\times,*)$ where $\times$ is the cartesian product and $*$ is a singleton set, a commutative monoid is a commutative monoid in the usual sense. If $k$ is a field, in the symmetric monoidal category $(\mathrm{Vec}_k,\otimes,k)$ where $\otimes$ is the tensor product of vector spaces, a commutative monoid is called a \emph{commutative $k$-algebra}.

\paragraph{iv. Cocommutative comonoids.} Dualizing, on string diagrams applying an up-down symmetry, we obtain the notion of a cocommutative comonoid. A \emph{cocommutative comonoid} in a symmetric monoidal category is an object $A$, together with a comultiplication map $\Delta\colon A \rightarrow A \otimes A$, in string diagrams
\begin{equation*}
\resizebox{!}{1.2cm}{\input{PRES12.tikz}}\quad,
\end{equation*}
and a counit map $\epsilon\colon A \rightarrow I$, in string diagrams
\begin{equation*}
\resizebox{!}{1cm}{\begin{tikzpicture}
	\begin{pgfonlayer}{nodelayer}
		\node [style=none] (13) at (0, 1.25) {$\scriptstyle A$};
		\node [style=none] (25) at (0, 0.75) {};
		\node [style=blanc] (26) at (0, -1.25) {};
	\end{pgfonlayer}
	\begin{pgfonlayer}{edgelayer}
		\draw (26) to (25.center);
	\end{pgfonlayer}
\end{tikzpicture}
}\quad,
\end{equation*}
such that the following equations are satisfied:
\begin{equation*}
\resizebox{!}{1.2cm}{\input{PRES14A.tikz}}~~=~~\resizebox{!}{1.2cm}{\input{PRES14B.tikz}}~~,\qquad
\resizebox{!}{1.2cm}{\input{PRES14C.tikz}}~~=~~\resizebox{!}{1.2cm}{\input{PRES14D.tikz}}~~,\qquad
\resizebox{!}{1.2cm}{\input{PRES14E.tikz}}.
\end{equation*}
The first equation is the counitality of the comultiplication, the second one its coassociativity and the third one its cocommutativity.

Every set becomes a cocommutative comonoid in $(\mathsf{Set},\times,*)$ by defining a comultiplication 
\begin{equation*}
\mathsf{copy}\colon X \rightarrow X \times X
\end{equation*}
as $\mathsf{copy}(x)=(x,x)$ and a counit
\begin{equation*}
\mathsf{erase}\colon X \rightarrow *
\end{equation*}
by $\mathsf{erase}(x)=\bullet$ where $\bullet$ is the unique element in the singleton $*$. More generally, if $\mathsf{C}$ is a category with finite products, and if we denote by $\times$ the binary product and by $\top$ the terminal object, then we obtain a symmetric monoidal category $(\mathsf{C},\times,\top)$ in which each object can be made into a cocommutative comonoid as in $\mathsf{Set}$. If $k$ is a field, in the symmetric monoidal category $(\mathsf{Vec}_k,\otimes,k)$, a cocommutative comonoid is called a \emph{cocommutative $k$-coalgebra}. Suppose that $k$ is a field of characteristic $0$, which means that $n\cdot 1 \neq 0$ for every $n \in \mathbb{N} \backslash \{0\}$, so that we can divide by $n$ for every positive integer $n$. An example of a cocommutative $k$-coalgebra is given by the differentiation of polynomials. The comultiplication
\begin{equation*}
\Delta\colon k[x] \rightarrow k[x] \otimes k[x]
\end{equation*}
is defined by 
\begin{equation*}
\Delta(f(x))=\sum_{n \ge 0}\frac{1}{n!}\frac{\mathrm{d}^nf}{\mathrm{d}x^n} \otimes x^n
\end{equation*}
and the counit 
\begin{equation*}
\epsilon\colon k[x] \rightarrow I
\end{equation*}
is defined by 
\begin{equation*}
\epsilon(a_nx^n+\dots+a_0)=a_0.
\end{equation*}
We will check counitality, coassociativity and cocommutativity. Since the monomials $x^p$ for $p \ge 0$ form a basis of $k[x]$, it suffices to check the identities on these.

First, observe that
\begin{equation*}
\frac{1}{n!}\frac{\mathrm{d}^n}{\mathrm{d}x^n}(x^p)=\frac{p(p-1)\dots(p-n+1)}{n!}x^{p-n}=\binom{p}{n}x^{p-n}
\end{equation*}
for every $0 \le n \le p$ and
\begin{equation*}
\frac{1}{n!}\frac{\mathrm{d}^n}{\mathrm{d}x^n}(x^p)=0
\end{equation*}
for every $n>p$, so that
\begin{align} \label{Delta-binom-intro}
\Delta(x^p)=\underset{0 \le n \le p}{\sum}\binom{p}{n}x^{p-n} \otimes x^n.
\end{align}
We thus have 
\begin{equation*}
(\epsilon \otimes \mathsf{id}_{k[x]})(\Delta(x^n))=1 \otimes x^n=x^n
\end{equation*}
which proves counitality. Next, we have
\begin{equation*}
\gamma(\Delta(x^p))=\underset{0 \le n \le p}{\sum}\binom{p}{n}x^n \otimes x^{p-n}.
\end{equation*}
We obtain $\gamma(\Delta(x^p))=\Delta(x^p)$ and thus cocommutativity by observing that
\begin{equation*}
\{(n,p-n)~|~0 \le n \le p\}=\{(a,b)~|~0 \le a,b \le p,~a+b=p\}=\{(p-n,n)~|~0 \le n \le p\}.
\end{equation*}
and that
\begin{equation*}
\binom{p}{n}=\binom{p}{p-n}
\end{equation*}
for every $0 \le n \le p$.

It remains to prove coassociativity. We have
\begin{equation} \label{coassoc-intro-1}
(\Delta \otimes \mathsf{id})(\Delta(x^p))=\underset{0 \le n \le p}{\sum}~\underset{0 \le k \le p-n}{\sum}\binom{p}{n}\binom{p-n}{k}x^{p-n-k} \otimes x^k \otimes x^n.
\end{equation}
We also have
\begin{equation} \label{coassoc-intro-2}
(\mathsf{id} \otimes \Delta)(\Delta(x^p))=(\mathsf{id} \otimes \Delta)(\gamma(\Delta(x^p)))=\underset{0 \le n \le p}{\sum}~\underset{0 \le k \le p-n}{\sum}\binom{p}{n}\binom{p-n}{k}x^n \otimes x^{p-n-k} \otimes x^k.
\end{equation}
By cocommutativity, we can apply a permutation to the RHS of \cref{coassoc-intro-2} to get that
\begin{equation*}
(\mathsf{id} \otimes \Delta)(\Delta(x^p))=\underset{0 \le n \le p}{\sum}~\underset{0 \le k \le p-n}{\sum}\binom{p}{n}\binom{p-n}{k}x^{p-n-k} \otimes x^k \otimes x^n.
\end{equation*}
Coassociativity follows from \cref{coassoc-intro-1,coassoc-intro-2}.
\paragraph{v. Bicommutative bimonoids.} The cocommutative comonoid $k[x]$ for $k$ a field of characteristic $0$ will play an important role in our paper, although in a more general form. Since $k[x]$ is a commutative $k$-algebra, that is, a commutative monoid in $(\mathsf{Vec}_k,\otimes,k)$, we can wonder how its structures of a commutative monoid and of a cocommutative comonoid in $(\mathsf{Vec}_k,\otimes,k)$ interact. This is where the notion of a bicommutative bimonoid appears. A \emph{bicommutative bimonoid} in a symmetric monoidal category $(\mathsf{C},\otimes,I)$ is given by a tuple $(A,\nabla,\eta,\Delta,\epsilon)$ such that $(A,\nabla,\eta)$ is a commutative monoid, $(A,\Delta,\epsilon)$ is a cocommutative comonoid, and the following identities are satisfied:
\begin{equation} \label{compatibility-intro-1}
\resizebox{!}{1.2cm}{\input{PRES15A.tikz}}~~=~~\resizebox{!}{1.2cm}{\input{PRES15B.tikz}}~~,\qquad
\resizebox{!}{1.2cm}{\input{PRES15C.tikz}}~~=~~\resizebox{!}{1.2cm}{\input{PRES15D.tikz}}~~,\qquad
\resizebox{!}{1.2cm}{\input{PRES15E.tikz}}~~=\qquad\qquad,
\end{equation}
\begin{equation} \label{compatibility-intro-2}
\resizebox{!}{1.2cm}{\input{PRES15F.tikz}}~~=~~\resizebox{!}{1.2cm}{\input{PRES15G.tikz}}\quad.
\end{equation}
We call a bicommutative bimonoid in $(\mathsf{Vec}_k,\otimes,k)$ a \emph{bicommutative $k$-bialgebra}. Another example is as follows: if $(X,\nabla,\eta)$ is a commutative monoid in the usual sense, that is, a commutative monoid in $(\mathsf{Set},\times,*)$, we obtain that $X$ is a bicommutative bimonoid with the comultiplication $\mathsf{copy}$ and the counit $\mathsf{erase}$ previously defined on any set $X$.

The identities in \cref{compatibility-intro-1} are easy to check for $k[x]$. We will check \cref{compatibility-intro-2} for $k[x]$. It suffices to check on the elements $x^n \otimes x^p$ since they form a basis of $k[x] \otimes k[x]$.

We obtain from \cref{Delta-binom-intro} that
\begin{equation} \label{intro-before-binomial-1}
\Delta(\nabla(x^n \otimes x^p))=\underset{0 \le k \le n+p}{\sum}\binom{n+p}{k}x^{n+p-k} \otimes x^k.
\end{equation}
Moreover, we have
\begin{equation*}
(\Delta \otimes \Delta)(x^n \otimes x^p)=\underset{0 \le a \le n}{\sum}~\underset{0 \le b \le p}{\sum}\binom{n}{a}\binom{p}{b}x^{n-a}\otimes x^a \otimes x^{p-b} \otimes x^b.
\end{equation*}
Applying the RHS in \cref{compatibility-intro-2} to $x^n \otimes x^p$ thus gives
\begin{equation} \label{intro-before-binomial-2}
\underset{0 \le a \le n}{\sum}~\underset{0 \le b \le p}{\sum}\binom{n}{a}\binom{p}{b}x^{n+p-a-b}\otimes x^{a+b}=\underset{0 \le k \le n+p}{\sum}~\underset{\substack{0 \le a \le n \\ 0 \le b \le p \\ a+b=k}}{\sum}\binom{n}{a}\binom{p}{b}x^{n+p-k} \otimes x^k.
\end{equation}
But the following formula, known as Vandermonde's identity, holds:\footnote{Vandermonde's identity can be proved as follows: Consider a set $A$ of cardinal $n$ and a set $B$ of cardinal $p$. Choosing a subset of $A \sqcup B$ with $k$ elements is equivalent to choosing a subset of $A$ with $a$ elements and a subset of $B$ with $b$ elements, for some decomposition $k=a+b$ where $0 \le a \le n$ and $0 \le b \le p$.}
\begin{equation} \label{intro-before-binomial-3}
\underset{\substack{0 \le a \le n \\ 0 \le b \le p \\ a+b=k}}{\sum}\binom{n}{a}\binom{p}{b}=\binom{n+p}{k}.
\end{equation}
Combining \cref{intro-before-binomial-1,intro-before-binomial-2,intro-before-binomial-3}, we obtain that $\Delta(\nabla(x^n \otimes x^p))$ is equal to the RHS in \cref{compatibility-intro-2} applied to $x^n \otimes x^p$, as desired.
\subsection{Bialgebraic properties of $k[x]$} \label{SEC-TWO-TWO}
In the previous subsection, we introduced two bicommutative bimonoids: the bicommutative bimonoid in $(\mathsf{Set},\times,*)$ obtained from any commutative monoid by copying and erasing, and the bicommutative $k$-bialgebra $k[x]$ for $k$ a field of characteristic $0$, obtained by adding a comultiplication defined from differentiation and a counit which extracts the constant term, to the usual commutative $k$-algebra structure on $k[x]$. In this paper, we will mainly be interested in the second bicommutative bimonoid. In this section, we will discuss some special properties of the bicommutative bimonoid $k[x]$ for $k$ a field of characteristic $0$.

\paragraph{i. $k[x]$ is an $\mathbb{N}$-graded $k$-vector space.} Denote by $k_n[x]$ the $k$-vector space of monomials of the form $a_nx^n$ for any nonnegative integer $n$ (note that we consider that the polynomial $0$ is of degree $n$ for any $n \ge 0$). We have the decomposition
\begin{equation} \label{graded-biprod}
k[x]=\underset{n \ge 0}{\bigoplus}k_n[x].
\end{equation}
We will say that $k[x]$ is an \emph{$\mathbb{N}$-graded $k$-vector space}. We have
\begin{align*}
k[x] \otimes k[x]=&~\Big(\underset{n \ge 0}{\bigoplus}k_n[x]\Big) \otimes \Big(\underset{p \ge 0}{\bigoplus}k_p[x]\Big) \\
\simeq&~\underset{n \ge 0}{\bigoplus}\Big(k_n[x] \otimes \Big(\underset{p \ge 0}{\bigoplus}k_p[x]\Big)\Big) \\
\simeq&~\underset{n \ge 0}{\bigoplus}\underset{p \ge 0}{\bigoplus}\Big(k_n[x] \otimes k_p[x]\Big) \\
\simeq&~\underset{n,p \ge 0}{\bigoplus}\Big(k_n[x] \otimes k_p[x]\Big). 
\end{align*}
\paragraph{ii. $k[x]$ is a commutative $\mathbb{N}$-graded $k$-algebra.} It follows that our multiplication map $\nabla\colon k[x] \otimes k[x] \rightarrow k[x]$ is completely determined by its restrictions to $k_n[x] \otimes k_p[x]$. But that's not everything. The product of an homogeneous polynomial of degree $n$ with an homogeneous polynomial of degree $p$ is an homogeneous polynomial of degree $n+p$ so that the restriction of $\nabla$ to $k_n[x] \otimes k_p[x]$ can in fact be corestricted to a map
\begin{equation} \label{graded-mult}
\nabla_{n,p}\colon k_n[x] \otimes k_p[x] \rightarrow k_{n+p}[x].
\end{equation}
Similarly, the image of our unit map $\eta\colon k \rightarrow k[x]$ is exactly $k_0[x]$, so that we can as well write our unit map under the form
\begin{equation} \label{graded-unit}
\eta\colon k \rightarrow k_0[x].
\end{equation}
In abstract algebra, a commutative $k$-algebra whose underlying $k$-vector space admits a decomposition as a countable coproduct as in \cref{graded-biprod}, and whose multiplication and unit maps can be decomposed as in \cref{graded-mult,graded-unit} is called a \emph{commutative $\mathbb{N}$-graded algebra}.

In a symmetric monoidal category, it would be harder to work with a countable coproduct as $k[x]$ rather than directly with the small pieces $k_n[x]$ for at least two reasons. Firstly, we do not have countable coproducts in many categories of interest because they do not admit large enough objects. Two examples are $(\mathsf{FVec}_k,\otimes,k)$ the symmetric monoidal category of finite dimensional $k$-vector spaces and $(\mathsf{FSet},\times,\{*\})$ the symmetric monoidal category of finite sets, so we cannot directly generalize the notion of an $\mathbb{N}$-graded vector space to these symmetric monoidal categories without any change in the definition. Secondly, even if we formulated a notion of $\mathbb{N}$-graded object as a countable product of objects, we would not be able to express that the multiplication sends $k_n[x] \otimes k_p[x]$ to $k_{n+p}[x]$ and that the unit sends $k$ to $k_0[x]$ without a notion of image which is not available in arbitrary symmetric monoidal categories.

We thus define a \emph{commutative $\mathbb{N}$-graded monoid} in a symmetric monoidal category $(\mathsf{C},\otimes,I)$ as follows: a graded object\footnote{If $\mathsf{C}$ is a category, we wil call \emph{graded object} in $\mathsf{C}$ any family of objects $(X_n)_{n \ge 0} \in \mathsf{C}^\mathbb{N}$. We will often just write $(X_n)$ for a graded object.} $(A_n)_{n \ge 0}$ together with a family of morphisms $(\nabla_{n,p}\colon A_n \otimes A_p \rightarrow A_{n+p})_{n,p \ge 0}$, whose components $\nabla_{n,p}\colon A_n \otimes A_p \rightarrow A_{n+p}$ we draw in string diagrams
\begin{equation*}
\resizebox{!}{1.2cm}{\input{PRES16.tikz}}\quad,
\end{equation*}
and a morphism $\eta\colon I \rightarrow A_0$, in string diagrams
\begin{equation*}
\resizebox{!}{1cm}{\begin{tikzpicture}
	\begin{pgfonlayer}{nodelayer}
		\node [style=none] (13) at (0, -1.25) {$\scriptstyle 0$};
		\node [style=none] (25) at (0, -0.75) {};
		\node [style=blanc] (26) at (0, 1.25) {};
	\end{pgfonlayer}
	\begin{pgfonlayer}{edgelayer}
		\draw (26) to (25.center);
	\end{pgfonlayer}
\end{tikzpicture}
}\quad,
\end{equation*}
such that the following identities are satisfied:
\begin{equation*}
\resizebox{!}{1.2cm}{\input{PRES18A.tikz}}~=~\resizebox{!}{1.2cm}{\begin{tikzpicture}
	\begin{pgfonlayer}{nodelayer}
		\node [style=none] (8) at (0, -3.5) {$n$};
		\node [style=none] (13) at (0, 3.5) {$n$};
		\node [style=none] (22) at (0, -3) {};
		\node [style=none] (25) at (0, 3) {};
		\node [style=none] (29) at (0, -3) {};
	\end{pgfonlayer}
	\begin{pgfonlayer}{edgelayer}
		\draw (22.center) to (25.center);
	\end{pgfonlayer}
\end{tikzpicture}
}~,\quad
\resizebox{!}{1.2cm}{\input{PRES18C.tikz}}~=~\resizebox{!}{1.2cm}{\input{PRES18D.tikz}}~,\quad
\resizebox{!}{1.2cm}{\input{PRES18E.tikz}}~=~\resizebox{!}{1.2cm}{\input{PRES18F.tikz}}.
\end{equation*}
\paragraph{iii. $k[x]$ is a cocommutative $\mathbb{N}$-graded $k$-coalgebra.} We will now explain the dual notion of a cocommutative $\mathbb{N}$-graded comonoid, starting with the example of $k[x]$ in $(\mathsf{Vec}_k,\otimes,k)$. Recall that
\begin{equation*}
k[x] \otimes k[x] \simeq \underset{n,p \ge 0}{\bigoplus}\Big(k_n[x] \otimes k_p[x]\Big) \subseteq \underset{n,p \ge 0}{\prod}\Big(k_n[x] \otimes k_p[x]\Big).
\end{equation*}
In abstract algebra, a cocommutative $k$-coalgebra whose underlying $k$-vector space is $\mathbb{N}$-graded is called a \emph{cocommutative $\mathbb{N}$-graded $k$-coalgebra} if the comultiplication satisfies
\begin{equation*}
\Delta(k_n[x]) \subseteq \underset{0 \le i \le n}{\bigoplus}k_{n-i}[x] \otimes k_i[x]
\end{equation*}
and the counit satisfies
\begin{equation*}
\epsilon(k_n[x])=\{0\}
\end{equation*}
for every $n \ge 1$. In these conditions, the comultiplication can equivalently be described via the maps
\begin{equation*}
\Delta_{n,p}=\pi_{n,p} \circ \Delta \circ i_{n+p}\colon k_{n+p}[x] \rightarrow k_n[x] \otimes k_p[x]
\end{equation*}
for all $n,p \ge 0$ and the counit can be described via the restriction
\begin{equation*}
\epsilon\colon k_0[x] \rightarrow k.
\end{equation*}
We define a \emph{cocommutative $\mathbb{N}$-graded comonoid} in a symmetric monoidal category $(\mathsf{C},\otimes,I)$ as the dual notion of a commutative $\mathbb{N}$-graded monoid. That is, a family of objects $(A_n)_{n \ge 0}$ together with a family of morphisms $(\Delta_{n,p}\colon A_{n+p} \rightarrow A_n \otimes A_p)_{n,p \ge 0}$, whose components $\Delta_{n,p}\colon A_{n+p} \rightarrow A_n \otimes A_p$ we draw in string diagrams
\begin{equation*}
\resizebox{!}{1.2cm}{\input{PRES19.tikz}}\quad,
\end{equation*}
and a morphism $\epsilon\colon A_0 \rightarrow I$, in string diagrams
\begin{equation*}
\resizebox{!}{0.8cm}{\begin{tikzpicture}
	\begin{pgfonlayer}{nodelayer}
		\node [style=none] (13) at (0, 1.25) {$\scriptstyle 0$};
		\node [style=none] (25) at (0, 0.75) {};
		\node [style=blanc] (26) at (0, -1.25) {};
	\end{pgfonlayer}
	\begin{pgfonlayer}{edgelayer}
		\draw (26) to (25.center);
	\end{pgfonlayer}
\end{tikzpicture}
}\quad,
\end{equation*}
such that the following identities are satisfied:
\begin{equation*}
\resizebox{!}{1.2cm}{\input{PRES21A.tikz}}~=~\resizebox{!}{1.2cm}{\begin{tikzpicture}
	\begin{pgfonlayer}{nodelayer}
		\node [style=none] (8) at (0, -3.5) {$n$};
		\node [style=none] (13) at (0, 3.5) {$n$};
		\node [style=none] (22) at (0, -3) {};
		\node [style=none] (25) at (0, 3) {};
		\node [style=none] (29) at (0, -3) {};
	\end{pgfonlayer}
	\begin{pgfonlayer}{edgelayer}
		\draw (22.center) to (25.center);
	\end{pgfonlayer}
\end{tikzpicture}
}~,\quad
\resizebox{!}{1.2cm}{\input{PRES21C.tikz}}~=~\resizebox{!}{1.2cm}{\input{PRES21D.tikz}}~,\quad
\resizebox{!}{1.2cm}{\input{PRES21E.tikz}}~=~\resizebox{!}{1.2cm}{\input{PRES21F.tikz}}~.
\end{equation*}
For every field $k$, the cocommutative coalgebra $k[x]$ is $\mathbb{N}$-graded since $\epsilon(f) \in k_0[x]$ for every polynomial $f$, and for every $f(x) \in k_{n}[x]$, we have
\begin{equation} \label{comult-with-division}
\Delta(f)=\underset{i \ge 0}{\sum}\frac{1}{i!}\frac{\mathrm{d}^if}{\mathrm{d}x^i} \otimes x^i=\underset{0 \le i \le n}{\sum}\frac{1}{i!}\frac{\mathrm{d}^if}{\mathrm{d}x^i} \otimes x^i \in \underset{0 \le i \le n}{\bigoplus}k_{n-i}[x] \otimes k_i[x]
\end{equation}
because $\frac{\mathrm{d}^if}{\mathrm{d}x^i}=0$ for every $i>n$ and $\frac{1}{i!}\frac{\mathrm{d}^if}{\mathrm{d}x^i} \otimes x^i \in k_{n-i}[x] \otimes k_i[x]$ for every $0 \le i \le n$.

Moreover, we obtain from \cref{Delta-binom-intro} that the map 
\begin{equation*}
\Delta_{n,p}\colon k_{n+p}[x] \rightarrow k_n[x] \otimes k_p[x]
\end{equation*}
is given for all $n,p \ge 0$ by
\begin{equation} \label{comult-without-division}
\Delta_{n,p}(x^{n+p})=\binom{n+p}{n}x^n \otimes x^p.
\end{equation}
\paragraph{iv. $k[x]$ is a bicommutative $\mathbb{N}$-graded $k$-bialgebra.}
We know that $k[x]$ is both a commutative $\mathbb{N}$-graded $k$-algebra and a cocommutative $\mathbb{N}$-graded $k$-coalgebra. We also know that it is a bicommutative $k$-bialgebra.

We call a bicommutative $k$-bialgebra which is both an $\mathbb{N}$-graded $k$-vector space, an $\mathbb{N}$-graded $k$-algebra and an $\mathbb{N}$-graded $k$-coalgebra for the $\mathbb{N}$-grading on the underlying $k$-vector space a \emph{bicommutative $\mathbb{N}$-graded $k$-bialgebra}.

In the case of an $\mathbb{N}$-graded $k$-vector space which is both a commutative $\mathbb{N}$-graded $k$-algebra and a cocommutative $\mathbb{N}$-graded $k$-coalgebra for the $\mathbb{N}$-grading on the underlying $k$-vector space, the \cref{compatibility-intro-1,compatibility-intro-2} which ensure that we have a bicommutative $k$-bialgebra can equivalently be written under the following form:
\begin{equation} \label{compatibility-intro-graded-1}
\resizebox{!}{1.2cm}{\input{PRES22A.tikz}}~~=~~\resizebox{!}{1.2cm}{\input{PRES22B.tikz}}~~,\qquad
\resizebox{!}{1.2cm}{\input{PRES22C.tikz}}~~=~~\resizebox{!}{1.2cm}{\input{PRES22D.tikz}}~~,\qquad
\resizebox{!}{1.2cm}{\begin{tikzpicture}
	\begin{pgfonlayer}{nodelayer}
		\node [style=none] (23) at (0, 0) {};
		\node [style=blanc] (33) at (0, -3.5) {};
		\node [style=none] (34) at (0, 0) {};
		\node [style=blanc] (36) at (0, 3.5) {};
	\end{pgfonlayer}
	\begin{pgfonlayer}{edgelayer}
		\draw (33) to (23.center);
		\draw (36) to (34.center);
	\end{pgfonlayer}
\end{tikzpicture}
}~~=\qquad\qquad,
\end{equation}
\begin{equation} \label{compatibility-intro-graded-2}
\resizebox{!}{1.2cm}{\input{nabladeltaA.tikz}}\quad=\quad 
\scalebox{1}{$\underset{\substack{a,b,c,d \ge 0 \\ a+b=n \\ c+d=p \\ a+c=q \\ b+d=r}}{\sum}$} 
\resizebox{!}{1.2cm}{\input{nabladeltaB.tikz}}\quad.
\end{equation}
Note that \cref{compatibility-intro-graded-2} must be satisfied for all $n,p,q,r \ge 0$ such that $n+p=q+r$. This equation uses a sum of $k$-linear maps. The notion of a bicommutative $\mathbb{N}$-graded bimonoid will thus only make sense in a symmetric monoidal category which is additive. Let $(\mathsf{C},\otimes,I)$ be an additive symmetric monoidal category. We call \emph{bicommutative $\mathbb{N}$-graded bimonoid} any family of objects $(A_n)_{n \ge 0}$ which is both a commutative $\mathbb{N}$-graded monoid and a cocommutative $\mathbb{N}$-graded comonoid, and such that \cref{compatibility-intro-graded-1,compatibility-intro-graded-2} are satisfied.

\paragraph{v. $k[x]$ is a connected bicommutative $\mathbb{N}$-graded $k$-bialgebra.}
The bicommutative $\mathbb{N}$-graded $k$-bialgebra $k[x]$ satisfies a specific property: we have an isomorphism of $k$-vector spaces $k_0[x] \simeq k$. More precisely, the equation
\begin{equation} \label{connected-intro}
\resizebox{!}{1.2cm}{\input{PRES23.tikz}}~~=~~\resizebox{!}{1.2cm}{\begin{tikzpicture}
	\begin{pgfonlayer}{nodelayer}
		\node [style=none] (23) at (0, 3) {};
		\node [style=none] (32) at (0, 3.5) {$0$};
		\node [style=none] (34) at (0, -3) {};
		\node [style=none] (37) at (0, -3.5) {$0$};
	\end{pgfonlayer}
	\begin{pgfonlayer}{edgelayer}
		\draw (23.center) to (34.center);
	\end{pgfonlayer}
\end{tikzpicture}
}
\end{equation}
is satisfied. We call a bicommutative $\mathbb{N}$-graded bimonoid in an additive symmetric monoidal category \emph{connected} if \cref{connected-intro} is satisfied.

\paragraph{vi. $k[x]$ is a special connected bicommutative $\mathbb{N}$-graded $k$-bialgebra.}
We're still concerned with $k[x]$ for $k$ a field of characteristic $0$. We compute
\begin{align*}
(\nabla_{n,p} \circ \Delta_{n,p})(x^{n+p})=&~\nabla_{n,p}(\Delta_{n,p}(x^{n+p})) \\
=&~\nabla_{n,p}\bigg(\binom{n+p}{n}x^n \otimes x^p\bigg) \\
=&~\binom{n+p}{n}x^{n+p}.
\end{align*}
We conclude that 
\begin{equation} \label{special-intro}
\nabla_{n,p} \circ \Delta_{n,p}=\binom{n+p}{n}\mathsf{id}_{k_{n+p}[x]}.
\end{equation}
We call a connected bicommutative $\mathbb{N}$-graded bimonoid in an additive symmetric monoidal category \emph{special} if \cref{special-intro} (where $k_{n+p}[x]$ is replaced with $A_{n+p}$) is satisfied for all $n,p \ge 0$. In string diagrams, \cref{special-intro} is written:
\begin{equation} \label{special-intro-diag}
\resizebox{!}{1.2cm}{\input{deltanabla.tikz}}~~=~~\binom{n+p}{n}\resizebox{!}{1.2cm}{\begin{tikzpicture}
	\begin{pgfonlayer}{nodelayer}
		\node [style=none] (16) at (0, -3) {};
		\node [style=none] (17) at (0, 3) {};
		\node [style=none] (19) at (0, 3.5) {$n+p$};
		\node [style=none] (20) at (0, -3.5) {$n+p$};
	\end{pgfonlayer}
	\begin{pgfonlayer}{edgelayer}
		\draw (17.center) to (16.center);
	\end{pgfonlayer}
\end{tikzpicture}
}\quad.
\end{equation}
For brevity, we will call \emph{binomial bimonoid} any special connected bicommutative $\mathbb{N}$-graded bimonoid in an additive symmetric monoidal category.

An important remark that will be used in \cref{SUBSEC:FIVE-THREE} is that $k[x]$ is still a binomial $k$-bialgebra when the field of characteristic $0$ is replaced with any field $k$ (in fact, even with any commutative rig $k$). In order for the comultiplication to still make sense, we must define it directly with \cref{comult-without-division} and not with \cref{comult-with-division}.
\section{Symmetric powers}
In this section, we introduce the other main notion in this paper after binomial bimonoids: symmetric powers in $\mathbb{Q}_{\ge 0}$-linear symmetric monoidal categories. We will state our first main theorem which relates these two notions and our second main theorem which shows that bicommutativity and biassociativity are not required in the definition of a binomial bimonoid in a $\mathbb{Q}_{\ge 0}$-linear symmetric monoidal category.
\subsection{$\mathbb{Q}_{\ge 0}$-linear symmetric monoidal categories}
\begin{definition} \label{def:k-linear}
Let $R$ be a commutative rig. A \emph{$R$-linear symmetric monoidal category} is a symmetric monoidal category $(\mathsf{C},\otimes,I)$ such that every hom-set $\mathsf{C}[A,B]$ is an $R$-module and, moreover, the following equations are satisfied whenever they make sense: 
\begin{align*}
(f+g);h &= (f;h) + (g;h), &\qquad f;(g+h) &= (f;g)+(f;h), \\
0;f &= f;0=0, &\qquad (\lambda\cdot f);g &= \lambda\cdot (f;g) = f;(\lambda\cdot g), \\
(f+g) \otimes h &= (f \otimes h) + (g \otimes h), &\qquad f \otimes (g+h) &= (f \otimes g) + (f \otimes h), \\
0 \otimes f &= f \otimes 0 = 0, &\qquad (\lambda\cdot f) \otimes g &= \lambda\cdot (f \otimes g) = f \otimes (\lambda g).
\end{align*}
\end{definition}
As we will be mainly interested in $\mathbb{Q}_{\ge 0}$-linear symmetric monoidal categories, we now introduce several examples of $\mathbb{Q}_{\ge 0}$-linear symmetric monoidal categories.
\begin{example} \label{ex:mod}
The category $\mathsf{Mod}_{R}$ of modules over a commutative $\mathbb{Q}_{\ge 0}$-algebra $R$ can be made into a $\mathbb{Q}_{\ge 0}$-linear symmetric monoidal category $(\mathsf{Mod}_{R},\otimes,R)$. The tensor product of modules over a commutative rig is the natural generalization of the tensor product of modules over a commutative ring. Let $R$ together with a rig homomorphism $\rho\colon\mathbb{Q}_{\ge 0} \rightarrow R$ be a commutative $\mathbb{Q}_{\ge 0}$-algebra. The action of $\mathbb{Q}_{\ge 0}$ on hom-sets is given by
\begin{equation*}
(q\cdot f)(x)=\rho(q)f(x)
\end{equation*}
for every $x \in A$ where $f \in \mathsf{Mod}_R(A,B)$.  
\end{example}
All the other examples will be of the following type.
\begin{definition}
An additive symmetric monoidal category\footnote{We recall that additive symmetric monoidal categories were defined in part ii.~of subsection 2.1.} $\mathsf{C}$ is \emph{additively idempotent} if for every morphism $f\colon A \rightarrow B$, we have $f+f=f$.
\end{definition}
\begin{proposition} \label{q-on-hom-sets}
Let $(\mathsf{C},\otimes,I)$ be an additively idempotent additive symmetric monoidal category. Then $(\mathsf{C},\otimes,I)$ is a $\mathbb{Q}_{\ge 0}$-linear symmetric monoidal category by defining
\begin{equation*}
q \cdot f=
\left\{
\begin{aligned}
f~&\text{if }q \neq 0, \\
0~&\text{if }q=0
\end{aligned}\right.
\end{equation*}
for all $q \in \mathbb{Q}_{\ge 0}$ and morphisms $f$.
\end{proposition}
\begin{proof}
For all $q,q' \in \mathbb{Q}_{\ge 0}$, we have $qq'=0$ iff $q=0$ or $q'=0$. We deduce from this that $(qf);(q'f')=(qq')\cdot (f;f')$ and $(qf) \otimes (q'f')=(qq')\cdot (f \otimes f')$. It suffices to check the two identities in these two cases: 1. at least one of $q$ or $q'$ is $0$; 2. both $q$ and $q'$ are nonzero.
\end{proof} 
We now introduce several examples of additively idempotent additive symmetric monoidal category.
\begin{example}
The symmetric monoidal category $(\mathsf{Rel},\times,*)$ of sets and relations is an additively idempotent additive symmetric monoidal category, thus a $\mathbb{Q}_{\ge 0}$-linear symmetric monoidal category.
\end{example}
\begin{example} \label{ex:r-idempotent}
Let $R$ be an \emph{additively idempotent commutative rig}, that is, a commutative rig such that $r+r=r$ for every $r \in R$. Then $(\mathsf{Mod}_R,\otimes,R)$ is an additively idempotent additive symmetric monoidal category since for any morphism $f$, we have $(f+f)(x)=f(x)+f(x)=f(x+x)=f(x)$ for any $x$ in the domain of $f$, and thus $f+f=f$. It follows from \cref{q-on-hom-sets} that $(\mathsf{Mod}_R,\otimes,R)$ is a $\mathbb{Q}_{\ge 0}$-linear symmetric monoidal category.
\end{example}
We have obtained \cref{ex:r-idempotent} by applying \cref{q-on-hom-sets}, but \cref{ex:r-idempotent} can also be seen as a special case of \cref{ex:mod}. Indeed, for any additively idempotent commutative rig $R$, we obtain a rig homomorphism $\rho\colon\mathbb{Q}_{\ge 0} \rightarrow R$ by defining
\begin{equation*}
\rho(q)=
\left\{
\begin{aligned}
1~&\text{if } q \neq 0, \\
0~&\text{if } q=0,
\end{aligned}\right.
\end{equation*}
so that $R$ is a commutative $\mathbb{Q}_{\ge 0}$-algebra. The pointwise action of $\mathbb{Q}_{\ge 0}$ on hom-sets of $\mathsf{Mod}_R$ obtained from $\rho$ as in \cref{ex:mod} coincides with the action of $\mathbb{Q}_{\ge 0}$ on hom-sets defined in \cref{q-on-hom-sets}.
\begin{example}
Let $\mathbb{B}$ be the unique commutative rig whose underlying set is $\{0,1\}$ and such that $1+1=1$. We can specialize \cref{ex:r-idempotent} to the case when $R=\mathbb{B}$. The category $\mathsf{Mod}_\mathbb{B}$ is isomorphic to the category of idempotent commutative monoids which is isomorphic to the category $\mathsf{JoinSLat}$ of join-semilattices.\footnote{We define a \emph{join-semilattice} as a poset $(X,\le)$ such that every finite subset $Y \subseteq X$ admits a join $\vee\,Y$. A \emph{homomorphism of join-semilattices} is a function that preserves finite joins.} We thus obtain that the category of join-semilattices is a $\mathbb{Q}_{\ge 0}$-linear symmetric monoidal category.
\end{example}
\begin{example}
A \emph{suplattice} is a poset $X$ such that every family $(x_i)_{i \in I}$ of elements of $X$ admits a join $\vee_{i \in I}x_i$. A \emph{homomorphism of suplattices} is a function which preserves joins. Let $X,Y$ be two suplattices. We will define a suplattice $X \otimes Y$. Start with the suplattice $\mathcal{P}(X \times Y)$ and consider the smallest congruence\footnote{A congruence on a suplattice $X$ is an equivalence relation $\sim$ on $X$ such that if $(x_i)_{i \in I}$ and $(y_i)_{i \in I}$ are two families of elements of $X$ such that $x_i \sim y_i$ for every $i \in I$, then $\vee_{i \in I}x_i \sim \vee_{i \in I}y_i$.} $\sim$ on $\mathcal{P}(X \times Y)$ such that
\begin{enumerate}
\item $\{(\vee_{i \in I} x_i,y)\} \sim \{\vee_{i \in I}(x_i,y)\}$ for all families $(x_i)_{i \in I}$ of elements of $X$ and $y \in Y$;
\item $\{(x,\vee_{i \in I} y_i)\} \sim \{\vee_{i \in I}(x,y_i)\}$ for all $x \in X$ and families $(y_i)_{i \in I}$ of elements of $Y$.
\end{enumerate}
We define $X \otimes Y:=\mathcal{P}(X \times Y)/\sim$ and we write $x \otimes y$ for the equivalence class of $\{(x,y)\} \in \mathcal{P}(X \times Y)$. Every element of $X \otimes Y$ can be written under the form $\vee_{i \in I}(x_i \otimes y_i)$ for some family $(x_i)_{i \in I}$ of elements of $X$ and some family $(y_i)_{i \in I}$ of elements of $Y$.\footnote{In more detail: the joins in $X \otimes Y$ are defined as $\vee_{i \in I}[A_i]=[\vee_{i \in I}A_i]$ where $A_i \in \mathcal{P}(X \times Y)$ and $[A_i]$ denotes the equivalence class of $A_i$. These operations are well-defined thanks to $\sim$ being a congruence. The order on $X \otimes Y$ is then defined as $A \le B$ iff $A \vee B=B$ for all $A,B \in X \otimes Y$.}

The universal property of $X \otimes Y$ is as follows: Let $X,Y,Z$ be three suplattices. For every bihomomorphism $\phi\colon X \times Y \rightarrow Z$ (that is, a map which preserves joins in each variable) there exists a unique homomorphism of suplattices $\psi\colon X \otimes Y \rightarrow Z$ such that $\phi=\psi \circ \pi$ where $\pi\colon X \times Y \rightarrow X \otimes Y$ is defined by $\pi(x,y)=x \otimes y$.

We obtain a symmetric monoidal category $(\mathsf{Sup},\otimes,\mathcal{P}(*))$. The category $\mathsf{Sup}$ is enriched over itself, that is, every homset is a suplattice, with joins defined pointwisely, and composition maps are bihomomorphisms i.e.~we have
\begin{equation*}
(\vee_{i \in I}f_i);g=\vee_{i \in I}(f_i;g)
\end{equation*}
and
\begin{equation*}
f;(\vee_{i \in I}g_i)=\vee_{i \in I}(f;g_i)
\end{equation*}
whenever it makes sense. We also have
\begin{equation*}
(\vee_{i \in I}f_i) \otimes g=\vee_{i \in I}(f_i \otimes g)
\end{equation*}
and
\begin{equation*}
f \otimes (\vee_{i \in I}g_i)=\vee_{i \in I}(f\otimes g_i).
\end{equation*}
It follows that $(\mathsf{Sup},\otimes,\mathcal{P}(*))$ is an additively idempotent additive symmetric monoidal category with operations on morphisms given by $f+g:= f \vee g$ and $0:=\bot\colon X \rightarrow Y$.

See \cite{joyextension,Semigroups} for more information on suplattices and their tensor product. We note that \cite{Semigroups} contains an essentially complete proof that $(\mathsf{Sup},\otimes,\mathcal{P}(1))$ is a symmetric monoidal category.
\end{example}
\subsection{The $n$-th symmetric power of an object} \label{THREE}
In this subsection, we show that symmetric powers can be defined in three equivalent ways in a $\mathbb{Q}_{\ge 0}$-linear symmetric monoidal category.
\begin{definition}
Let $(\mathsf{C}, \otimes, I)$ be a symmetric monoidal category, and let $A \in \mathsf{C}$. For every $n \ge 0$, an $n$\emph{-th} \emph{(co)equalizer symmetric power} $S_{n}A$ of $A$ is, when it exists, defined as a (co)equalizer of the diagram
\begin{equation} \label{DIAG}
% https://tikzcd.yichuanshen.de/#N4Igdg9gJgpgziAXAbVABwnAlgFyxMJZABgBpiBdUkANwEMAbAVxiRAEEA9YAHR4jwBbeAAIwAXxDjS6TLnyEUAJnJVajFmy69+Q0RKlqYUAObwioAGYAnCIKRkQOCEgCM1OAAsslnEgC0KurMrIggfNgmgnQg1Ax0AEYwDAAKcngEbAwwvlIyIDZ2DtTObh7euYhB8Ump6QpZOX7U9CFsfADGUAII4hTiQA
\begin{tikzcd}
A^{\otimes n} \arrow[rr, "\sigma", shift left=2] \arrow[rr, "\cdots", shift right=2] &  & A^{\otimes n}
\end{tikzcd}
\end{equation}
where there is one arrow for every $\sigma \in S_{n}$.
\end{definition}
\begin{remark}
By a (co)equalizer of the diagram \labelcref{DIAG}, we mean a (co)limit of \labelcref{DIAG}. If $n=0$ or $1$, then the diagram \labelcref{DIAG} contains a single arrow (recall that both $S_0$ and $S_1$ are the trivial group) which is either the identity on $I$ or the identity on $A$. Moreover, the (co)limit of a diagram consisting of a single arrow $f\colon A \rightarrow B$ is simply an isomorphism $B \simeq C$. Thus, a $0$-th symmetric power of $A$ is an isomorphism $I \simeq A_0$ and a $1$-th symmetric power of $A$ is an isomorphism $A \simeq A_1$.
\end{remark}
We will need the following lemma.
\begin{lemma} \label{coeq-is-epi}
Let $\mathsf{C}$ be any category and let $f_i\colon A \rightarrow B$ for $1 \le i \le n$ be morphisms in $\mathsf{C}$, where $n \ge 1$. Any coequalizer of
\begin{equation} \label{diag-for-epi}
% https://tikzcd.yichuanshen.de/#N4Igdg9gJgpgziAXAbVABwnAlgFyxMJZABgBpiBdUkANwEMAbAVxiRAEEQBfU9TXfIRQAmclVqMWbAELdxMKAHN4RUADMAThAC2SMiBwQkARmpwAFljU4kAWgAc1es1aIQagPrGQ1BnQBGMAwACvx4BGwMMNbcvO5auoj6hiZmljGItgDMvgFBodjhQiBRMU6Sru4ewrHqCXrUKYimIBZWNoiOEi5snoQ8dToNBkbNae1IOd1SbgA6swDGUBA4CFwUXEA
\begin{tikzcd}[ampersand replacement=\&]
A \arrow[rr, "f_1", shift left=8] \arrow[rr, "f_2", shift left=3] \arrow[rr, "f_n", shift right=8] \arrow[rr, "\cdots", shift right=3] \&  \& B
\end{tikzcd}
\end{equation}
is an epimorphism.
\end{lemma}
\begin{proof}
Let $u\colon B \rightarrow X$ be a coequalizer of the diagam \labelcref{diag-for-epi}. Let $a,b\colon X \rightarrow Y$ be two parallel morphisms in $\mathsf{C}$ such that $u;a=u;b$. We then have $f_i;u;a=f_i;u;b$ for every $1 \le i \le n$. 

By the universal property of $u$, there exist a unique morphism $\phi\colon X \rightarrow Y$ such that $u;a=u;\phi$ and a unique morphism $\psi\colon X \rightarrow Y$ such that $u;b=u;\psi$. It follows that $a=b=\phi=\psi$.

We conclude that $u$ is an epimorphism.
\end{proof}
\begin{remark} \label{eq-is-mono}
The dual version of \cref{coeq-is-epi} is that any equalizer of a diagram consisting of $n \ge 1$ parallel morphisms is a monomorphism.
\end{remark}
\begin{proposition} \label{PROPSYM}
Let $(\mathsf{C}, \otimes, I)$ be a $\mathbb{Q}_{\ge 0}$-linear symmetric monoidal category, and let $X \in \mathsf{C}$. For every $n \ge 0$, we define bijections between:
\begin{enumerate}
\item[$(a)$]  coequalizers $
% https://tikzcd.yichuanshen.de/#N4Igdg9gJgpgziAXAbVABwnAlgFyxMJZABgBpiBdUkANwEMAbAVxiRAEEA9YAHR4jwBbeAAIwAXxDjS6TLnyEUARnJVajFmwAaUtTCgBzeEVAAzAE4RBSMiBwQkKkHAAWWUzhvV6zVohBMuuJAA
\begin{tikzcd}
A^{\otimes n} \arrow[r, "u"] & A_n
\end{tikzcd}
$
 of the diagram \labelcref{DIAG},
\item[$(b)$] couples
$
% https://tikzcd.yichuanshen.de/#N4Igdg9gJgpgziAXAbVABwnAlgFyxMJZABgBpiBdUkANwEMAbAVxiRAEEA9YAHR4jwBbeAAIwAXxDjS6TLnyEUARnJVajFmwAaUtTCgBzeEVAAzAE4RBSMiBwQkKkHAAWWUziQBaJ-WatEECYpGRALK0dqextqV3dPRB9qP01Aml1xIA
\begin{tikzcd}
A^{\otimes n} \arrow[r, "u", shift left] &A_n \arrow[l, "v", shift left]
\end{tikzcd}
$
such that $u;v = \frac{1}{n!}\underset{\sigma \in S_{n}}{\sum}\sigma$ and $v;u=\mathsf{id}_{A_n}$,
\item[$(c)$] equalizers $
% https://tikzcd.yichuanshen.de/#N4Igdg9gJgpgziAXAbVABwnAlgFyxMJZABgBpiBdUkANwEMAbAVxiRAA0QBfU9TXfIRQBGclVqMWbAIIA9YAB0FEPAFt4AAjBdu4mFADm8IqABmAJwiqkZEDghJRIOAAsspnDer1mrRLV0uIA
\begin{tikzcd}
A_n \arrow[r, "v"] & A^{\otimes n}
\end{tikzcd}
$
of the diagram \labelcref{DIAG}
\end{enumerate}
by associating:
\begin{enumerate}
\item[$(a \rightarrow b)$] to $u$, the pair $(u,v)$ where $v$ is the unique arrow that makes the triangle commute in the following diagram:
\begin{equation*}
% https://tikzcd.yichuanshen.de/#N4Igdg9gJgpgziAXAbVABwnAlgFyxMJZABgBpiBdUkANwEMAbAVxiRAEEA9YAHR4jwBbeAAIwAXxDjS6TLnyEUAJnJVajFmy69+Q0RKkyQGbHgJEALKur1mrRCAAah2aYWXSARjW3ND7XwCWMJwYpLiajBQAObwRKAAZgBOEIJIZCA4EEie1HAAFlgJOEgAtCrqdmx82NGCdC4gyanp1Fk5eYXFSBUMdABGMAwACnJmiiAMMN02GvYgfADGUAII0okpaYi5mdmIFb7zTI3NWzvtiADM1H2DI2PuDklY0fkls1UOfMl0i8Ce4mAYAAhOI+EwwLAknAYDgdLV6iI+FgwEiePUcPkfgBrYAAZXEAH0geJATUmIJxGiEQ11k1Nj02ntrpMUfMoHQClEQB8-LQpBRxEA
\begin{tikzcd}
A^{\otimes n} \arrow[rr, "\sigma", shift left=2] \arrow[rr, "\cdots", shift right=2] &  & A^{\otimes n} \arrow[rr, "u"] \arrow[rrd, "\frac{1}{n!}\underset{\sigma \in S_{n}}{\sum} \sigma"'] &  & A_n \arrow[d, "v", dashed] \\
                                                                                     &  &                                                                                                               &  & A^{\otimes n}           
\end{tikzcd}\quad,
\end{equation*}
\item[$(b \rightarrow a)$] to the pair $(u,v)$, the morphism $u$,
\item[$(b \rightarrow c)$] to the pair $(u,v)$, the morphism $v$,
\item[$(c \rightarrow b)$] to $v$, the pair $(u,v)$ where $u$ is the unique arrow that makes the triangle commute in the following diagram:
\begin{equation*}
% https://tikzcd.yichuanshen.de/#N4Igdg9gJgpgziAXAbVABwnAlgFyxMJZAJgBoAGAXVJADcBDAGwFcYkQBBAPWAB1eIeALbwABGAC+ICaXSZc+QigAsFanSat23PgOFjJ02SAzY8BIuTU0GLNohAANI3LOLLpAIzrbWhzv5BLBE4cSkJdRgoAHN4IlAAMwAnCCEkKxAcCCRPGjgACywEnCQAWjINO3Z+bGihehcQZNT0miycvMLipArGegAjGEYABXlzJRBGGG6bTXsQfgBjKEEEGUSUtMQK9sQM33naRuatgGY27L3Zqod+ZPpF4E8JYDAAQgl+ZjBYJLgYHC6Wr1UT8LBgUG8eo4fL3ADWwAAyhIAPqvCQvGrMIQSSHAho0PqDEZjdwOJJYaL5ErrJqbJDnTKXXoDIajNwWBxTGaTcHzKD0ApRECdIolK6VPwgZjSSgSIA
\begin{tikzcd}
A_n \arrow[rr, "v"]                                                                                                    &  & A^{\otimes n} \arrow[rr, "\sigma", shift left=2] \arrow[rr, "\cdots", shift right=2] &  & A^{\otimes n} \\
A^{\otimes n} \arrow[rru, "\frac{1}{n!}\underset{\sigma \in S_{n}}{\sum} \sigma"'] \arrow[u, "u", dashed] &  &                                                                                      &  &              
\end{tikzcd}\quad.
\end{equation*}
\end{enumerate}
\end{proposition}
\begin{proof}
We first check that the images satisfy the required property.
\begin{enumerate}
\item[$(a \rightarrow b)$] 
To start, $v$ is well-defined because for every $\rho \in S_n$ we have
\begin{equation*}
\rho;\bigg(\frac{1}{n!}\underset{\sigma \in S_n}{\sum}\sigma\bigg)
=\frac{1}{n!}\underset{\sigma \in S_n}{\sum}\rho;\sigma
=\frac{1}{n!}\underset{\tau \in S_n}{\sum}\tau
\end{equation*}
 by a change of variable. By definition of $v$, we have \[u;v = \frac{1}{n!}\underset{\sigma \in S_{n}}{\sum}\sigma.\] 
 The identity \[v;u = \mathsf{id}_{A_n}\] follows from $u$, as a coequalizer of the diagram \labelcref{DIAG}, being an epimorphism by \cref{coeq-is-epi}. Indeed, we have 
\begin{equation*}
u;v;u = \bigg(\frac{1}{n!}\underset{\sigma \in S_{n}}{\sum}\sigma\bigg);u = \frac{1}{n!}\underset{\sigma \in S_{n}}{\sum}\sigma;u = \frac{1}{n!}\underset{\sigma \in S_{n}}{\sum}u = \frac{1}{n!}(n!u) =u=u;\mathsf{id}_{A_n},\] thus \[v;u = \mathsf{id}_{A_n}.
\end{equation*}
\item[$(b \rightarrow a)$] Let $\sigma\colon A^{\otimes n} \rightarrow A^{\otimes n}$ be any permutation. We have
\begin{equation*}
\sigma;u=\sigma;u;\mathsf{id}_{A_n}=\sigma;u;(v;u)=\sigma;(u;v);u=\sigma;\bigg(\frac{1}{n!}\underset{\rho \in S_{n}}{\sum}\rho\bigg);u = \frac{1}{n!}\underset{\rho \in S_{n}}{\sum}\sigma;\rho;u = \frac{1}{n!}\underset{\gamma \in S_{n}}{\sum}\gamma;u.
\end{equation*}
Therefore,
\begin{equation*}
u = \sigma^{-1};\sigma;u = \sigma^{-1};\bigg(\frac{1}{n!}\underset{\gamma \in S_{n}}{\sum}\gamma;u\bigg ) = \frac{1}{n!}\underset{\gamma \in S_{n}}{\sum}\sigma^{-1};\gamma;u = \frac{1}{n!}\underset{\rho \in S_{n}}{\sum}\rho;u = \sigma;u.
\end{equation*} 
We proved that $u$ coequalizes the diagram \labelcref{DIAG}.

Now, suppose that $f\colon A^{\otimes n} \rightarrow Y$ is such that $\sigma;f=f$ for any permutation $\sigma\colon A^{\otimes} \rightarrow A^{\otimes n}$.
The triangle commutes in the following diagram:
\begin{equation*}
% https://tikzcd.yichuanshen.de/#N4Igdg9gJgpgziAXAbVABwnAlgFyxMJZABgBpiBdUkANwEMAbAVxiRAEEA9YAHR4jwBbeAAIwAXxDjS6TLnyEUAJnJVajFmy69+Q0RKkyQGbHgJEALKur1mrRCACah2aYWXSARjW3NDgBpSajBQAObwRKAAZgBOEIJIZCA4EEie1HAAFlhROEgAtCrqdmx82KGCdC4gsfGJ1ClpGdm5SEW+9iAAdD0g1Ax0AEYwDAAKcmaKIAwwrdLRcQmI6cmpiO0anVHVtUsrjYgAzDabbEx900Mj427mDjFYoZl58zWLSMerbf1XYxPu90ezwuHTYNAA3NtxBRxEA
\begin{tikzcd}
A^{\otimes n} \arrow[rr, "\sigma", shift left=2] \arrow[rr, "...", shift right=2] &  & A^{\otimes n} \arrow[rr, "f"] \arrow[rrd, "u"'] &  & Y                   \\
                                                                                  &  &                                                 &  & X \arrow[u, "v;f"']
\end{tikzcd}.
\end{equation*}
Indeed, we have
\begin{equation*}
u;v;f=\bigg(\frac{1}{n!}\underset{\sigma \in S_{n}}{\sum}\sigma\bigg);f = \frac{1}{n!}\underset{\sigma \in S_{n}}{\sum}\sigma;f = \frac{1}{n!}\underset{\sigma \in S_{n}}{\sum}f = f.
\end{equation*}
Finally, suppose that $\phi\colon A_n \rightarrow Y$ makes the triangle commute in the following diagram:
\begin{equation*}
% https://tikzcd.yichuanshen.de/#N4Igdg9gJgpgziAXAbVABwnAlgFyxMJZABgBpiBdUkANwEMAbAVxiRAEEA9YAHR4jwBbeAAIwAXxDjS6TLnyEUAJnJVajFmy69+Q0RKkyQGbHgJEALKur1mrRCACah2aYWXSARjW3NDgBpSajBQAObwRKAAZgBOEIJIZCA4EEie1HAAFlhROEgAtCrqdmx82KGCdC4gsfGJ1ClpGdm5SEW+9iAAdD0g1Ax0AEYwDAAKcmaKIAwwrdLRcQmI6cmpiO0anVHVtUsrjYgAzDabbEx900Mj427mDjFYoZl58zWLSMerbf1XYxPu90ezwuHVKPDQ2SC4iAA
\begin{tikzcd}
A^{\otimes n} \arrow[rr, "\sigma", shift left=2] \arrow[rr, "\cdots", shift right=2] &  & A^{\otimes n} \arrow[rr, "f"] \arrow[rrd, "u"'] &  & Y                    \\
                                                                                  &  &                                                 &  & A_n \arrow[u, "\phi"']
\end{tikzcd}\quad.
\end{equation*}
Then \[v;f = v;u;\phi=\mathsf{id}_{A_n};\phi=\phi.\]
\item[$(b \rightarrow c)$] Let $\sigma\colon A^{\otimes n} \rightarrow A^{\otimes n}$ be any permutation. We have
\begin{equation*}
v;\sigma = \mathsf{id}_{A_n};v;\sigma=(v;u);v;\sigma = v;(u;v);\sigma
=v;\bigg(\frac{1}{n!}\underset{\rho \in S_{n}}{\sum}\rho\bigg);\sigma
= \frac{1}{n!}\underset{\rho \in S_{n}}{\sum}v;\rho;\sigma = \frac{1}{n!}\underset{\gamma \in S_{n}}{\sum}v;\gamma.
\end{equation*} 
Therefore,
\begin{equation*}
v = v;\sigma;\sigma^{-1} = \frac{1}{n!}\underset{\gamma \in S_{n}}{\sum}v;\gamma;\sigma^{-1} = \frac{1}{n!}\underset{\rho \in S_{n}}{\sum}v;\rho = v;\sigma.
\end{equation*}
Now, suppose that $f\colon Y \rightarrow A^{\otimes n}$ is such that $\sigma;f=f$ for any permutation $\sigma\colon A^{\otimes} \rightarrow A^{\otimes n}$. The triangle commutes in the following diagram:
\begin{equation*}
% https://tikzcd.yichuanshen.de/#N4Igdg9gJgpgziAXAbVABwnAlgFyxMJZAJgBoAGAXVJADcBDAGwFcYkQBBAPWAB1eIeALbwABGAC+ICaXSZc+QigAsFanSat23PgOFjJ02SAzY8BIuTU0GLNohABNI3LOLLpAIzrbWhwA1pdRgoAHN4IlAAMwAnCCEkKxAcCCRPGjgACywonCQAWjINO3Z+bFChehcQWPjEmhS0jOzcpCLfexAAOh6QGkZ6ACMYRgAFeXMlEEYYVplouITEIsbEJI72KOrapYBmBtS1m01O2j7poZHxtwsHGKxQzLz5msW2g6R94r8agG5mc4DYZjCbuO4PJ5BCRAA
\begin{tikzcd}
Y \arrow[rr, "f"] \arrow[d, "f;u"'] &  & A^{\otimes n} \arrow[rr, "\sigma", shift left=2] \arrow[rr, "\cdots", shift right=2] &  & A^{\otimes n} \\
A_n \arrow[rru, "v"']                 &  &                                                                                   &  &              
\end{tikzcd}\quad.
\end{equation*}
Indeed, 
\begin{equation*}
f;u;v = f;\bigg(\frac{1}{n!}\underset{\sigma \in S_{n}}{\sum}\sigma\bigg) = \frac{1}{n!}\underset{\sigma \in S_{n}}{\sum}f;\sigma = \frac{1}{n!}\underset{\sigma \in S_{n}}{\sum}f = f.
\end{equation*}
Finally, suppose that $\phi\colon Y \rightarrow A_n$ makes the triangle commute in the following diagram:
\begin{equation*}
% https://tikzcd.yichuanshen.de/#N4Igdg9gJgpgziAXAbVABwnAlgFyxMJZAJgBoAGAXVJADcBDAGwFcYkQBBAPWAB1eIeALbwABGAC+ICaXSZc+QigAsFanSat23PgOFjJ02SAzY8BIuTU0GLNohABNI3LOLLpAIzrbWhwA1pdRgoAHN4IlAAMwAnCCEkKxAcCCRPGjgACywonCQAWjINO3Z+bFChehcQWPjEmhS0jOzcpCLfexAAOh6QGkZ6ACMYRgAFeXMlEEYYVplouITEIsbEJI72KOrapYBmBtS1m01O2j7poZHxtwsHGKxQzLz5msW2g6R94r8QfjRs84DYZjCbuO4PJ5BCRAA
\begin{tikzcd}
Y \arrow[rr, "f"] \arrow[d, "\phi"'] &  & A^{\otimes n} \arrow[rr, "\sigma", shift left=2] \arrow[rr, "\cdots", shift right=2] &  & A^{\otimes n} \\
A_n \arrow[rru, "v"']                  &  &                                                                                   &  &              
\end{tikzcd}\quad.
\end{equation*}
Then
\begin{equation*}
f;u=\phi;v;u = \phi;\mathsf{id}_{A_n} = \phi.
\end{equation*}
\item[$(c \rightarrow b)$] To start, $u$ is well-defined because for every $\rho \in S_n$, we have
\begin{equation*}
\bigg(\frac{1}{n!}\underset{\sigma \in S_n}{\sum}\sigma\bigg);\rho = \frac{1}{n!}\underset{\sigma \in S_n}{\sum}\sigma;\rho = \frac{1}{n!}\underset{\tau \in S_n}{\sum}\tau
\end{equation*}
by a change of variable. By definition of $u$, we have
\begin{equation*}
u;v = \frac{1}{n!} \underset{\sigma \in S_{n}}{\sum}\sigma.
\end{equation*}
The identity
\begin{equation*}
v;u = \mathsf{id}_{A_n}
\end{equation*}
follows from $v$, as a limit of the diagram \labelcref{DIAG}, being a monomorphism by \cref{eq-is-mono}. 
Indeed, we have 
\begin{equation*}
v;u;v = v;\bigg(\frac{1}{n!} \underset{\sigma \in S_{n}}{\sum}\sigma\bigg) = \frac{1}{n!}\underset{\sigma \in S_{n}}{\sum}v \sigma
= \frac{1}{n!}\underset{\sigma \in S_{n}}{\sum}v = v,
\end{equation*}
thus 
\begin{equation*}
(v;u);v = \mathsf{id}_{A_n};v
\end{equation*}
and therefore 
\begin{equation*}
v;u = \mathsf{id}_{A_n}.
\end{equation*}
\end{enumerate}
We now verify the bijectivities by checking that the appropriate composites are identities.
\begin{enumerate}
\item[$(a \rightarrow b \rightarrow a)$] We associate a couple $(u,v)$ to $u$ and then $u$ to $(u,v)$.
\item[$(b \rightarrow a \rightarrow b)$] We associate $u$ to $(u,v)$ and then $(u,v)$ to $u$ because we know that $v$ is the unique morphism $A_n \rightarrow A^{\otimes n}$ such that 
\[u;v = \frac{1}{n!} \underset{\sigma \in S_{n}}{\sum}\sigma.\]
\item[$(b \rightarrow c \rightarrow b)$] We associate $v$ to $(u,v)$ and then $(u,v)$ to $v$ because we know that $u$  is the unique morphism $A^{\otimes n} \rightarrow A_n$ such that 
\[u;v = \frac{1}{n!} \underset{\sigma \in S_{n}}{\sum}\sigma.\]
\item[$(c \rightarrow b \rightarrow c)$] We associate a couple $(u,v)$ to $v$ and then $v$ to $(u,v)$.\qedhere
\end{enumerate}
\end{proof}
\begin{example} \label{ex:1a}
Let $R$ be a commutative $\mathbb{Q}_{\ge 0}$-algebra. In $(\mathsf{Mod}_R,\otimes,R)$, we define for every module $M$, the modules $S^0M:=R$, $S^1M:=M$ and for every $n \ge 2$, the module $S^nM$ by the formula $S^nM:=M^{\otimes n}/\sim$ where $\sim$ is the smallest congruence\footnote{A congruence on an $R$-module $E$ is an equivalence relation $\sim$ on $E$ such that: 1.~if $e_1 \sim e_2$, then $\lambda e_1 \sim \lambda e_2$ for every $\lambda \in R$; 2.~if $e_1 \sim e_2$ and $e_3 \sim e_4$, then $e_1+e_3 \sim e_2+e_4$.} on $M^{\otimes n}$ such that $m_1 \otimes \dots \otimes m_n \sim m_{\sigma(1)} \otimes \dots \otimes m_{\sigma(n)}$ for every $\sigma \in S_n$. The equivalence class of $x_1 \otimes \dots \otimes x_n$ in $S^nM$ is denoted by $x_1 \otimes_s \dots \otimes_s x_n$. For every $n \ge 0$, the module $S^nM$ satisfies the following universal property: Let $N$ be a module and let $\phi\colon M^n \rightarrow N$ be a symmetric multilinear map, that is, a multilinear map which is invariant by permutation of its entries.\footnote{Recall that both $S_0$ and $S_1$ are the trivial group. A symmetric multilinear map $\phi\colon M^1=M \rightarrow N$ is just a linear map and a symmetric multilinear map $\phi\colon M^0=* \rightarrow N$ is just a function, which can be identified with the point $\phi(\bullet) \in N$ where $\bullet$ is the unique element in $*$.} Then, there exists a unique linear map $\psi\colon S^n X \rightarrow Y$ such that $\phi=\psi \circ u$ where 
$u\colon X^n \rightarrow S^nX$ is the symmetric multilinear map defined by\footnote{When $n=0$, $u\colon X^0=* \rightarrow S^0X=R$ is the function given by $u(\bullet)=1$.}
\begin{equation*}
u(x_1,\dots,x_n)=x_1 \otimes_s \dots \otimes_s x_n.
\end{equation*}
For every $n \ge 0$, the module $S^nM$ is an $n$-th symmetric power of $M$ as witnessed by the maps
\begin{equation*}
% https://tikzcd.yichuanshen.de/#N4Igdg9gJgpgziAXAbVABwnAlgFyxMJZABgBpiBdUkANwEMAbAVxiRAEEA9YAHR4jwBbeAAIwAXxDjS6TLnyEUARnJVajFmwAaUtTCgBzeEVAAzAE4RBSMiBwQkKkHAAWWUziQBaJ-WatEECYpGRALK0dqextqV3dPRB9qP01Aml1xIA
\begin{tikzcd}
M^{\otimes n} \arrow[r, "r_n", shift left] &S^nM \arrow[l, "s_n", shift left]
\end{tikzcd}
\end{equation*}
defined for every $n \ge 0$ by
\begin{equation*}
r_n(m_1 \otimes \dots \otimes m_n)=m_1 \otimes_s \dots \otimes_s m_n
\end{equation*}
and
\begin{equation*}
s_n(m_1 \otimes_s \dots \otimes_s m_n)=\frac{1}{n!}\underset{\sigma \in S_n}{\sum}m_{\sigma(1)} \otimes \dots \otimes m_{\sigma(n)}.
\end{equation*}
The notations $m_1 \otimes \dots \otimes m_n$ and $m_1 \otimes_s \dots \otimes_s m_n$ must be interpreted as denoting a same scalar $\lambda \in R$ when $n=0$. We immediately obtain that $r_n;s_n=\frac{1}{n!}\underset{\sigma \in S_n}{\sum}\sigma$ and $s_n;r_n=\mathsf{id}_{S^nM}$.
\end{example}
\begin{example} \label{ex:2a}
In $(\mathsf{Rel},\times,*)$, we define for every set $X$, the sets $\mathcal{M}_0X:=*$, $\mathcal{M}_1X:=X$ and for every $n \ge 2$, $\mathcal{M}_nX$ as the set of all the multisets of elements of $X$ of cardinality $n$. That is, $\mathcal{M}_nX:=\{[x_1,\dots,x_n],~x_1,\dots,x_n \in X\}$.

For every $n \ge 0$, the set $\mathcal{M}_nX$ is an $n$-th symmetric power of $X$ as witnessed by the maps
\begin{equation*}
% https://tikzcd.yichuanshen.de/#N4Igdg9gJgpgziAXAbVABwnAlgFyxMJZABgBpiBdUkANwEMAbAVxiRAEEA9YAHR4jwBbeAAIwAXxDjS6TLnyEUARnJVajFmwAaUtTCgBzeEVAAzAE4RBSMiBwQkKkHAAWWUziQBaJ-WatEECYpGRALK0dqextqV3dPRB9qP01Aml1xIA
\begin{tikzcd}
M^n \arrow[r, "r_n", shift left] &\mathcal{M}_nX \arrow[l, "s_n", shift left]
\end{tikzcd}
\end{equation*}
\begin{equation*}
r_n=\{((x_1,\dots,x_n),[x_1,\dots,x_n])~|~(x_1,\dots,x_n) \in X^n\}
\end{equation*}
and
\begin{equation*}
s_n=\{([x_1,\dots,x_n],(x_{\sigma(1)},\dots,x_{\sigma(n)}))~|~(x_1,\dots,x_n) \in X^n,~\sigma \in S_n\}.
\end{equation*}
Indeed, we have $r_n;s_n=\underset{\sigma \in S_n}{\sum}\sigma=\frac{1}{n!}\underset{\sigma \in S_n}{\sum}\sigma$ and $s_n;r_n=\mathsf{id}_{\mathcal{M}_nX}$.

Since we defined $\mathcal{M}_0X:=*$ and $\mathcal{M}_1X:=X$, in the formulas for $r_0,s_0,r_1,s_1$, we must read $()=[]:=\bullet \in *$ and $(x_1)=[x_1]:=x_1 \in X$.
\end{example}
\begin{example} \label{ex:3a}
In $(\mathsf{Sup},\otimes,\mathcal{P}(*))$, we define for every suplattice $X$, the suplattices $S^0X:=\mathcal{P}(*)$, $S^1X:=X$ and for every $n \ge 2$ the suplattice $S^nX$ by the formula $S^nX:=X^{\otimes n}/\sim$ where $\sim$ is the smallest congruence on $X^{\otimes n}$ such that $x_1 \otimes \dots \otimes x_n \sim x_{\sigma(1)} \otimes \dots \otimes x_{\sigma(n)}$ for all $x_1,\dots,x_n \in X$ and  $\sigma \in S_n$. As for modules, the equivalence class of $x_1 \otimes \dots \otimes x_n$ in $S^nM$ is denoted by $x_1 \otimes_s \dots \otimes_s x_n$.

The universal property of $S^nX$ is a follows: Let $X,Y$ be suplattices and let $\phi\colon X^n \rightarrow Y$ be a symmetric multimorphism of suplattices, that is, a map which preserves joins in each variable and is invariant by permutation of its entries.\footnote{As in the case of modules, a symmetric multimorphism $\phi\colon X^1=X \rightarrow Y$ is just a homomorphism of suplattices and a symmetric multimorphism $\phi\colon X^0=* \rightarrow Y$ is just a function, which can be identified with the point $\phi(\bullet) \in Y$.} Then, there exists a unique homomorphism of suplattices $\psi\colon S^n X \rightarrow Y$ such that $\phi=\psi \circ u$ where 
$u\colon X^n \rightarrow S^nX$ is the symmetric multimorphism of suplattices defined by\footnote{When $n=0$, $u\colon X^0=* \rightarrow S^0X=\mathcal{P}(*)$ is the function given by $u(\bullet)=*$.}
\begin{equation*}
u(x_1,\dots,x_n)=x_1 \otimes_s \dots \otimes_s x_n.
\end{equation*}
For every $n \ge 0$, the suplattice $S^nX$ is an $n$-th symmetric power of $X$ as witnessed by the maps
\begin{equation*}
% https://tikzcd.yichuanshen.de/#N4Igdg9gJgpgziAXAbVABwnAlgFyxMJZABgBpiBdUkANwEMAbAVxiRAEEA9YAHR4jwBbeAAIwAXxDjS6TLnyEUARnJVajFmwAaUtTCgBzeEVAAzAE4RBSMiBwQkKkHAAWWUziQBaJ-WatEECYpGRALK0dqextqV3dPRB9qP01Aml1xIA
\begin{tikzcd}
X^{\otimes n} \arrow[r, "r_n", shift left] &S^nX \arrow[l, "s_n", shift left]
\end{tikzcd}
\end{equation*}
defined by
\begin{equation*}
r_n(x_1 \otimes \dots \otimes x_n)=x_1 \otimes_s \dots \otimes_s x_n
\end{equation*}
and
\begin{equation*}
s_n(x_1 \otimes_s \dots \otimes_s x_n)=\underset{\sigma \in S_n}{\bigvee}x_{\sigma(1)} \otimes \dots \otimes x_{\sigma(n)}.
\end{equation*}
The notations $x_1 \otimes \dots \otimes x_n$ and $x_1 \otimes_s \dots \otimes_s x_n$ must be interpreted as denoting a same element in $X^{\otimes 0}=S^0X=\mathcal{P}(*)$ when $n=0$. We immediately obtain that $r_n;s_n=\underset{\sigma \in S_n}{\vee}\sigma=\underset{\sigma \in S_n}{\sum}\sigma=\frac{1}{n!}\underset{\sigma \in S_n}{\sum}\sigma$ and $s_n;r_n=\mathsf{id}_{S^nX}$.
\end{example}
\subsection{Main theorems} \label{subsec:main-th}
We first introduce permutation splittings which are our formalization of a family of all the symmetric powers of an object in a $\mathbb{Q}_{\ge 0}$-linear symmetric monoidal category.
\begin{definition} \label{def-perm-splitting}
A \emph{permutation splitting} in a $\mathbb{Q}_{\ge 0}$-linear symmetric monoidal category $(\mathsf{C},\otimes,I)$ is a graded object $(A_n)_{n \ge 0}$ together with two families of morphisms
\begin{equation*}
(r_n\colon A_1^{\otimes n} \rightarrow A_n)_{n \ge 0}
\end{equation*}
and
\begin{equation*}
(s_n\colon A_n \rightarrow A_1^{\otimes n})_{n \ge 0}
\end{equation*}
such that 
\begin{equation} \label{perm-split-equation-1}
r_n;s_n=\frac{1}{n!}\underset{\sigma \in S_n}{\sum}\sigma
\end{equation}
and
\begin{equation} \label{perm-split-equation-2}
s_n;r_n=\mathsf{id}_{A_n}
\end{equation}
for every $n \in \mathbb{N}$, and
\begin{equation} \label{eq-trivial}
r_1=s_1=\mathsf{id}_{A_1}.
\end{equation}
\end{definition}
A permutation splitting is thus a graded object $(A_n)_{n \ge 0}$ where $A_n$ is an $n$-th symmetric power of $A_1$, together with morphisms exhibiting $A_n$ as an $n$-th symmetric power of $A_1$. We require by \cref{eq-trivial} that the morphisms exhibiting $A_1$ as a first symmetric power of itself are identities. 
\begin{example} \label{ex:1b}
Let $R$ be a commutative $\mathbb{Q}_{\ge 0}$-algebra. From \cref{ex:1a}, we obtain for every module $M \in \mathsf{Mod}_R$ a permutation splitting in $(\mathsf{Mod}_R,\otimes,R)$ given by the graded object $(S^nM)_{n \ge 0}$ together with the maps $(r_n\colon M^{\otimes n} \rightarrow S^nM)_{n \ge 0}$ and $(s_n\colon S^nM \rightarrow M^{\otimes n})_{n \ge 0}$ defined by
\begin{equation*}
r_n(m_1 \otimes \dots \otimes m_n)=m_1 \otimes_s \dots \otimes_s m_n
\end{equation*}
and
\begin{equation*}
s_n(m_1 \otimes_s \dots \otimes_s m_n)=\frac{1}{n!}\underset{\sigma \in S_n}{\sum}m_{\sigma(1)} \otimes \dots \otimes m_{\sigma(n)}.
\end{equation*}
We have $S^1M=M=M^{\otimes 1}$ and $r_1=s_1=\mathsf{id}_M$.
\end{example}
\begin{example} \label{ex:2b}
Let $X$ be any set. From \cref{ex:2a}, we obtain a permutation splitting in $(\mathsf{Rel},\times,*)$ given by the graded object $(\mathcal{M}_nX)_{n \ge 0}$ together with the maps $(r_n\colon X^n \rightarrow \mathcal{M}_nX)_{n \ge 0}$ and $(s_n\colon\mathcal{M}_nX \rightarrow X^n)_{n \ge 0}$ defined by
\begin{equation*}
r_n=\{((x_1,\dots,x_n),[x_1,\dots,x_n])~|~(x_1,\dots,x_n) \in X^n\}
\end{equation*}
and
\begin{equation*}
s_n=\{([x_1,\dots,x_n],(x_{\sigma(1)},\dots,x_{\sigma(n)}))~|~(x_1,\dots,x_n) \in X^n,~\sigma \in S_n\}.
\end{equation*}
We have $\mathcal{M}_1X=X=X^1$ and $r_1=s_1=\mathsf{id}_X$.
\end{example}
\begin{example} \label{ex:3b}
Let $X$ be a suplattice. From \cref{ex:3a}, we obtain a permutation splitting in $(\mathsf{Sup},\otimes,\mathcal{P}(*))$ given by the graded object $(S^nX)_{n \ge 0}$ together with the maps $(r_n\colon X^{\otimes n} \rightarrow S^nX)_{n \ge 0}$ and $(s_n\colon S^nX \rightarrow X^{\otimes n})_{n \ge 0}$ defined by
\begin{equation*}
r_n(m_1 \otimes \dots \otimes m_n)=m_1 \otimes_s \dots \otimes_s m_n
\end{equation*}
and
\begin{equation*}
s_n(m_1 \otimes_s \dots \otimes_s m_n)=\underset{\sigma \in S_n}{\bigvee}m_{\sigma(1)} \otimes \dots \otimes m_{\sigma(n)}.
\end{equation*}
We have $S^1X=X=X^{\otimes 1}$ and $r_1=s_1=\mathsf{id}_X$.
\end{example}
We now introduce the notion of a binomial bimonoid which will be the bialgebraic structure characterizing families of symmetric powers in a $\mathbb{Q}_{\ge 0}$-linear symmetric monoidal category.
\begin{definition} \label{def:bin-bimonoid}
Let $(\mathsf{C},\otimes,I)$ be an additive symmetric monoidal category. A \emph{binomial bimonoid} is a special connected bicommutative $\mathbb{N}$-graded bimonoid, as defined in \cref{SEC-TWO}. That is, a binomial bimonoid is a graded object $(A_n) \in \mathsf{C}^\mathbb{N}$ together with maps $(\nabla_{n,p}\colon A_n \otimes A_p \rightarrow A_{n+p})_{n,p \ge 0}$, $(\Delta_{n,p}\colon A_{n+p} \rightarrow A_n \otimes A_p)_{n,p \ge 0}$, $\eta\colon I \rightarrow A_0$ and $\epsilon\colon A_0 \rightarrow I$ subject to the twelve axioms in \cref{fig:bialg-axiom}.\footnote{In more detail, we ask that the axioms hold for all appropriate $n,p,q,r \ge 0$. That is, the axioms involving only $n$ must hold for every $n \ge 0$, the axioms involving $n$ and $p$ must hold for all $n,p \ge 0$, the axioms involvoving $n,p,q$ must hold for all $n,p,q \ge 0$ and the axiom involving $n,p,q,r$ must hold for all $n,p,q,r \ge 0$ such that $n+p=q+r$.}
\end{definition}
\begin{figure}
\begin{framed}
\begin{align*}
\resizebox{!}{1.2cm}{\input{PRES18A.tikz}} & ~~=~~ \resizebox{!}{1.2cm}{\input{PRES18B.tikz}}
&
\resizebox{!}{1.2cm}{\input{PRES18C.tikz}} & ~~=~~ \resizebox{!}{1.2cm}{\input{PRES18D.tikz}}
&
\resizebox{!}{1.2cm}{\input{PRES18E.tikz}} & ~~=~~ \resizebox{!}{1.2cm}{\input{PRES18F.tikz}} \\[18pt]  % <--- extra space here
\resizebox{!}{1.2cm}{\input{PRES21A.tikz}} & ~~=~~ \resizebox{!}{1.2cm}{\input{PRES21B.tikz}}
&
\resizebox{!}{1.2cm}{\input{PRES21C.tikz}} & ~~=~~ \resizebox{!}{1.2cm}{\input{PRES21D.tikz}}
&
\resizebox{!}{1.2cm}{\input{PRES21E.tikz}} & ~~=~~ \resizebox{!}{1.2cm}{\input{PRES21F.tikz}} \\[18pt]
\resizebox{!}{1.2cm}{\input{PRES22A.tikz}} & ~~=~~ \resizebox{!}{1.2cm}{\input{PRES22B.tikz}}
&
\resizebox{!}{1.2cm}{\input{PRES22C.tikz}} & ~~=~~ \resizebox{!}{1.2cm}{\input{PRES22D.tikz}}
&
\resizebox{!}{1.2cm}{\input{PRES22E.tikz}} & ~~=
\end{align*}
\begin{align*}
\resizebox{!}{1.2cm}{\input{nabladeltaA.tikz}}~~=~ ~
\scalebox{1}{$\underset{\substack{a,b,c,d \ge 0 \\ a+b=n \\ c+d=p \\ a+c=q \\ b+d=r}}{\sum}$} 
\resizebox{!}{1.2cm}{\input{nabladeltaB.tikz}}
\end{align*}
\begin{align*}
\resizebox{!}{1.2cm}{\input{PRES23.tikz}}~~=~~\resizebox{!}{1.2cm}{\input{PRES24.tikz}}~\qquad\qquad
\resizebox{!}{1.2cm}{\input{deltanabla.tikz}}~~=~~\binom{n+p}{n}\resizebox{!}{1.2cm}{\input{idn+p.tikz}}
\end{align*}
\caption{Axioms for a binomial bimonoid}
\label{fig:bialg-axiom}
\end{framed}
\end{figure}
\begin{figure}
\begin{framed}
\begin{align*}
\resizebox{!}{1.2cm}{\input{PRES18A.tikz}} & ~~=~~ \resizebox{!}{1.2cm}{\input{PRES18B.tikz}}
&
\resizebox{!}{1.2cm}{\input{PRES18C.tikz}} & ~~=~~ \resizebox{!}{1.2cm}{\input{PRES18D.tikz}}
&
\resizebox{!}{1.2cm}{\input{PRES18E.tikz}} & ~~=~~ \resizebox{!}{1.2cm}{\input{PRES18F.tikz}} \\[18pt]  % <--- extra space here
\resizebox{!}{1.2cm}{\input{PRES21A.tikz}} & ~~=~~ \resizebox{!}{1.2cm}{\input{PRES21B.tikz}}
&
\resizebox{!}{1.2cm}{\input{PRES21C.tikz}} & ~~=~~ \resizebox{!}{1.2cm}{\input{PRES21D.tikz}}
&
\resizebox{!}{1.2cm}{\input{PRES21E.tikz}} & ~~=~~ \resizebox{!}{1.2cm}{\input{PRES21F.tikz}} 
\end{align*}
\begin{align*}
\resizebox{!}{1.2cm}{\input{PRES22E.tikz}} & ~~=~~
&
\resizebox{!}{1.2cm}{\input{nabladeltaA.tikz}}& ~~=~~
\scalebox{1}{$\underset{\substack{a,b,c,d \ge 0 \\ a+b=n \\ c+d=p \\ a+c=q \\ b+d=r}}{\sum}$} 
\resizebox{!}{1.2cm}{\input{nabladeltaB.tikz}} \\[18pt]
\resizebox{!}{1.2cm}{\input{PRES23.tikz}}& ~~=~~\resizebox{!}{1.2cm}{\input{PRES24.tikz}}
&
\resizebox{!}{1.2cm}{\input{deltanabla.tikz}}& ~~=~~\binom{n+p}{n}\resizebox{!}{1.2cm}{\input{idn+p.tikz}}
\end{align*}
\caption{Equivalent axioms for a binomial bimonoid}
\label{fig:bialg-axiom-nonred}
\end{framed}
\end{figure}
\begin{figure}
\begin{framed}
\begin{equation*}
\begin{aligned}
\resizebox{!}{1.2cm}{\input{PRES18A.tikz}} & ~~=~~ \resizebox{!}{1.2cm}{\input{PRES18B.tikz}}
&
\resizebox{!}{1.2cm}{\input{PRES18A-.tikz}} & ~~=~~ \resizebox{!}{1.2cm}{\input{PRES18B.tikz}} \\[18pt]
\resizebox{!}{1.2cm}{\input{PRES21A.tikz}} & ~~=~~ \resizebox{!}{1.2cm}{\input{PRES21B.tikz}}
&
\resizebox{!}{1.2cm}{\input{PRES21A-.tikz}} & ~~=~~ \resizebox{!}{1.2cm}{\input{PRES21B.tikz}} \\[18pt]
\resizebox{!}{1.2cm}{\input{nabladeltaA.tikz}}
&~~=~~
\sum_{\substack{a,b,c,d \ge 0 \\ a+b=n \\ c+d=p \\ a+c=q \\ b+d=r}}
\resizebox{!}{1.2cm}{\input{nabladeltaB.tikz}}
& \hspace{4.5em} % slightly reduced gap
\resizebox{!}{1.2cm}{\input{PRES22E.tikz}}
&~~=~~ \\[18pt]
\resizebox{!}{1.2cm}{\input{deltanabla.tikz}}
&~~=~~
\binom{n+p}{n}\resizebox{!}{1.2cm}{\input{idn+p.tikz}}
& \hspace{4.5em} % same reduced gap on second row
\resizebox{!}{1.2cm}{\input{PRES23.tikz}}
&~~=~~\resizebox{!}{1.2cm}{\input{PRES24.tikz}}
\end{aligned}
\end{equation*}
\caption{Equivalent axioms for a binomial bimonoid in a $\mathbb{Q}_{\ge 0}$-linear symmetric monoidal category}
\label{fig:bialg-axiom-short}
\end{framed}
\end{figure}
Recall that \cref{SEC-TWO} was focused on the binomial bimonoid $k[x]$ for $k$ a field of characteristic $0$. Other examples of binomials bimonoids will be presented in \cref{SUBSEC:FIVE-THREE} and \cref{SEC:LAST}. Until there, the paper will discuss binomial bimonoids from an abstract point of view, our main concern being to prove \cref{main-theorem,main-theorem-2}.

Two axioms are redundant in \cref{fig:bialg-axiom} so that an equivalent set of axioms for binomial bimonoids is given by \cref{fig:bialg-axiom-nonred}. This follows from \cref{prop:red} and \cref{rem:red}.
\begin{proposition} \label{prop:red}
Let $(\mathsf{C},\otimes,I)$ be an additive symmetric monoidal category. The axiom
\begin{equation*}
\resizebox{!}{1.2cm}{\input{PRES22A.tikz}}~~=~~\resizebox{!}{1.2cm}{\input{PRES22B.tikz}}
\end{equation*}
can be omitted from the definition of a binomial bimonoid.
\end{proposition}
\begin{proof}
We compute:
\begin{equation*}
\resizebox{!}{1.2cm}{\input{PRES22A.tikz}}~=~\resizebox{!}{1.2cm}{\input{INTER1.tikz}}~=~\resizebox{!}{1.2cm}{\input{INTER2.tikz}}~=~\resizebox{!}{1.2cm}{\input{PRES22B.tikz}}~.
\end{equation*}
\end{proof}
\begin{remark} \label{rem:red}
The dual of \cref{prop:red} is that the axiom
\begin{equation*}
\resizebox{!}{1.2cm}{\input{PRES22C.tikz}}~~=~~\resizebox{!}{1.2cm}{\input{PRES22D.tikz}}
\end{equation*}
can be omitted from the definition of a binomial bimonoid.
\end{remark}
We can now state our two main theorems about permutation splittings and binomial bimonoids.
\begin{theorem} \label{main-theorem} Let $(\mathsf{C},\otimes,I)$ be a $\mathbb{Q}_{\ge 0}$-linear symmetric monoidal category. Let $(A_n)_{n \ge 0} \in \mathsf{C}^\mathbb{N}$ be a graded object. There is a bijection between:
\begin{enumerate}
\item maps $(r_n)_{n \ge 0}$, $(s_n)_{n \ge 0}$ making $(A_n)_{n \ge 0}$ into a permutation splitting;
\item maps $(\nabla_{n,p})_{n,p \ge 0}$, $(\Delta_{n,p})_{n,p \ge 0}$, $\eta$, $\epsilon$ making $(A_n)_{n \ge 0}$ into a binomial bimonoid.
\end{enumerate}
Starting from a permutation splitting, we define
\begin{equation} \label{perm-from-bimon-1}
\eta:=r_0,\quad\epsilon:=s_0,
\end{equation}
and
\begin{equation} \label{perm-from-bimon-2}
\nabla_{n,p}:=s_n \otimes s_p;r_{n+p},\quad \Delta_{n,p}:=\binom{n+p}{n}s_{n+p};r_n \otimes r_p
\end{equation}
for all $n,p \ge 0$.

Starting from a binomial bimonoid, we define
\begin{equation} \label{def-fun:-1}
r_0:= \eta,\quad s_0:=\epsilon,
\end{equation}
and
\begin{equation} \label{def-fun:2}
\resizebox{!}{1.2cm}{\input{r-n.tikz}}\quad:=\quad \resizebox{!}{1.2cm}{\input{rn-e.tikz}}~,\qquad
\qquad
\resizebox{!}{1.2cm}{\input{s-n.tikz}}\quad:=\quad \frac{1}{n!}\resizebox{!}{1.2cm}{\input{sn-e.tikz}}\quad
\end{equation}
for every $n \ge 1$. (If $n=1$, the RHS in the definitions of $r_n$ and $s_n$ must be understood as $\mathsf{id}_{A_1}$.)
\end{theorem}
The following theorem states that in any $\mathbb{Q}_{\ge 0}$-linear symmetric monoidal category, the biassociativity and bicommutativity axioms may be omitted from the definition of a binomial bimonoid. Since bicommutativity is not assumed, we must add the right unitality and counitality axioms to our left unitality and counitality axioms.
\begin{theorem} \label{main-theorem-2}
Let $(\mathsf{C},\otimes,I)$ be a $\mathbb{Q}_{\ge 0}$-linear symmetric monoidal category. Let $(A_n) \in \mathsf{C}^\mathbb{N}$. Let $(\nabla_{n,p}\colon A_n \otimes A_p \rightarrow A_{n+p})_{n,p \ge 0}$, $(\Delta_{n,p}\colon A_{n+p} \rightarrow A_n \otimes A_p)_{n,p \ge 0}$, $\eta\colon I \rightarrow A_0$ and $\epsilon\colon A_0 \rightarrow I$ be maps in $\mathsf{C}$. The graded object $(A_n)$ together with these maps is a binomial bimonoid iff the eight axioms in \cref{fig:bialg-axiom-short} hold.
\end{theorem}
\subsection{Permutation averages} \label{SIX}
Permutation averages are $\mathbb{Q}_{\ge 0}$-convex combinations of permutation natural transformations in a $\mathbb{Q}_{\ge 0}$-linear symmetric monoidal category. In this subsection, we define permutation averages and state some of their properties which will be useful later on.

A finite sum of natural transformations of the form $A^{\otimes n} \rightarrow A^{\otimes n}$ in a $\mathbb{Q}_{\ge 0}$-linear symmetric monoidal category is automatically a natural transformation. Hence, naturality is automatically satisfied in the next definition and does not have to be checked.
\begin{definition} \label{def:blender}
Let $\mathsf{C}$ be a $\mathbb{Q}_{\ge 0}$-linear symmetric monoidal category. An \emph{$n$-permutation average} is a natural transformation $\frac{1}{|M|}\underset{\sigma \in M}{\sum}\sigma\colon A^{\otimes n} \rightarrow A^{\otimes n}$ for some $n \ge 0$ and nonempty multiset $M \in \mathcal{M}(S_{n})$.
\end{definition}
In the next two propositions, we work in a $\mathbb{Q}_{\ge 0}$-linear symmetric monoidal category $\mathsf{C}$.
\begin{proposition} \label{prop:blender}
Permutations are permutation averages, and permutation averages are closed under tensor product and composition. More precisely: 
\begin{itemize}
\item Any permutation $A^{\otimes n} \rightarrow A^{\otimes n}$ is an $n$-permutation average.
\item The tensor product of an $n$-permutation average with a $p$-permutation average is an $(n+p)$-permutation average.
\item The composite of two $n$-permutation averages is again an $n$-permutation average.
\end{itemize}
\end{proposition}
\begin{proof}
Singleton multisets yield permutations.

For tensor products: 
\begin{align*}
&~\Big(\frac{1}{|M_1|}\underset{\sigma \in M_1}{\sum}\sigma \Big) \otimes \Big(\frac{1}{|M_2|}\underset{\rho \in M_2}{\sum}\rho \Big) \\
=&~\frac{1}{|M_1||M_2|} \underset{\nu \in M_{1} \otimes M_{2}}{\sum}\nu \\
=&~\frac{1}{|M_1 \otimes M_2|} \underset{\nu \in M_{1} \otimes M_{2}}{\sum}\nu
\end{align*}
where \[M_1 \otimes M_2 := [\sigma \otimes \rho~\mid~ \sigma \in M_1,~\rho \in M_2].\]

For composition: 
\begin{align*}
&~\Big(\frac{1}{|M_1|}\underset{\sigma \in M_1}{\sum}\sigma \Big);\Big(\frac{1}{|M_2|}\underset{\rho \in M_2}{\sum}\rho \Big) \\
=&~\frac{1}{|M_1||M_2|} \underset{\nu \in M_{1};M_{2}}{\sum}\nu \\
=&~\frac{1}{|M_1;M_2|} \underset{\nu \in M_{1};M_{2}}{\sum}\nu
\end{align*}
where 
\begin{equation*}
M_1;M_2 := [\sigma;\rho~\mid~\sigma \in M_1,~\rho \in M_2].\qedhere
\end{equation*}
\end{proof}
\begin{proposition}\label{INVARIANT}
Let $((A_n),(s_n),(r_n))$ be a permutation splitting. For every $ \ge 0$, $s_n$ is invariant under postcomposition with any $n$-permutation average, and $r_n$ is invariant under precomposition with any $n$-permutation average.
\end{proposition}
\begin{proof}
We already know from \cref{PROPSYM} that $s_{n}$ is invariant under postcomposition by permutations, i.e.~if $\sigma \in S_{n}$, then \[s_{n};\sigma = s_{n}.\] We thus have 
\begin{equation*}
s_{n};\Big(\frac{1}{|M|}\underset{\sigma \in M}{\sum}\sigma \Big)
=\frac{1}{|M|}\underset{\sigma \in M}{\sum}s_{n};\sigma
=\frac{1}{|M|}\underset{\sigma \in M}{\sum}s_{n}
=\frac{1}{|M|}|M|s_n
=s_n.
\end{equation*}
The argument for $r_{n}$ is analogous.
\end{proof}
\section{Proving the main theorems} \label{SEC:FOUR}
In all this section, we will work in a given $\mathbb{Q}_{\ge 0}$-linear symmetric monoidal category $(\mathsf{C},\otimes,I)$.
\subsection{From the short definition of a binomial bimonoid to permutation splitting}
In this subsection, we consider a graded object $(A_n)$ in $\mathsf{C}$ equipped with maps $\eta\colon I \rightarrow A_0$, $\epsilon\colon A_0 \rightarrow I$, $(\nabla_{n,p}\colon A_n \otimes A_p \rightarrow A_{n+p})_{n,p \ge 0}$, $(\Delta_{n,p}\colon A_{n+p} \rightarrow A_n \otimes A_p)_{n,p \ge 0}$ satisfying the eight axioms in \cref{fig:bialg-axiom-short}, and define the maps $(r_n\colon A_1^{\otimes n} \rightarrow A_n)_{n \ge 0}$, $(s_n\colon A_n \rightarrow A_1^{\otimes n})_{n \ge 0}$ as in \cref{main-theorem}. That is, we define the maps $r_n$ and $s_n$ by \cref{def-fun:-1,def-fun:2}. We will show that we obtain a permutation splitting $((A_n),(r_n),(s_n))$.

By definition, we already have
\begin{equation*}
r_1=s_1=\mathsf{id}_{A_1}.
\end{equation*}
Moreover,
\begin{equation*}
r_0;s_0=\eta;\epsilon=\mathsf{id}_I
\end{equation*}
and
\begin{equation*}
s_0;r_0=\epsilon;\eta=\mathsf{id}_{A_0}.
\end{equation*}
It follows that all we have to prove are \cref{perm-split-equation-1,perm-split-equation-2} for every $n \ge 2$.
\begin{definition} \label{extension}
We extend the family of morphisms $\nabla_{n_{1},\ldots,n_{q}}\colon A_{n_{1}} \otimes \cdots \otimes A_{n_{q}} \rightarrow A_{n_{1}+\cdots+n_{q}}$ where $n_1,\dots,n_q \in \mathbb{N}$ from $q=2$ to every $q \ge 2$ by induction:
\begin{equation*} 
\nabla_{n_{1},\ldots,n_{q+1}} = \nabla_{n_{1},\ldots,n_{q}} \otimes \mathsf{id}_{A_{n_{q+1}}}; \nabla_{n_{1}+\cdots+n_{q},n_{q+1}}.
\end{equation*}
In string diagram notation:
\begin{equation} \label{higher-order-mult-ind}
\resizebox{!}{1.2cm}{\input{highnabla.tikz}} \quad:=\quad \resizebox{!}{1.2cm}{\input{Nind.tikz}}\quad.
\end{equation}
In the same way, we extend the family of morphisms $\Delta_{n_{1},\ldots,n_{q}}\colon A_{n_{1}+\cdots+n_{q}} \rightarrow A_{n_{1}} \otimes \ldots \otimes A_{n_{q}}$ where $n_1,\dots,n_q \in \mathbb{N}$ from $q=2$ to every $q \ge 2$ by induction: 
\begin{equation*}
\Delta_{n_{1},\ldots,n_{q+1}} = \Delta_{n_{1}+\cdots+n_{q},n_{q+1}};\Delta_{n_{1},\ldots,n_{q}} \otimes \mathsf{id}_{A_{n_{q+1}}}.
\end{equation*}
In string diagram notation:
\begin{equation} \label{higher-order-comult-ind}
\resizebox{!}{1.2cm}{\input{highdelta.tikz}} \quad:=\quad \resizebox{!}{1.2cm}{\input{Dind.tikz}}\quad.
\end{equation}
\end{definition}
\begin{proposition} \label{HIGH1}
This identity holds for all $q,r \ge 2$ and nonnegative integers $n_i, p_j$:
\begin{equation} \label{prop:higher-order-mul-comul}
\resizebox{!}{1.2cm}{\input{highnabladeltaA.tikz}}\quad=\quad\underset{\substack{\underset{\alpha}{\sum}i_{\alpha}^{\beta} ~=~ p_{\beta} \\ \underset{\beta}{\sum} i_{\alpha}^{\beta} ~=~ n_{\alpha}}}{\sum} \resizebox{!}{1.2cm}{\input{highnabladeltaB.tikz}}
\end{equation}
where the sum ranges over all matrices $(i_{\alpha}^{\beta}) \in \mathbb{N}^{q \times r}$ such that for each $\alpha \in [1,q]$, we have $\underset{1 \le \beta \le r}{\sum}i_{\alpha}^{\beta} = n_{\alpha}$ and for each $\beta  \in [1,r]$,  we have $\underset{1 \le \alpha \le q}{\sum}i_{\alpha}^{\beta} = p_{\beta}$.
\end{proposition}
\begin{proof}
The proof is by induction on $(q,r)$. The case $(2,2)$ is an axiom in \cref{fig:bialg-axiom-short}. Suppose that \cref{prop:higher-order-mul-comul} holds for all couples $(q,r)$ where $2 \le q \le q_0$ and $2 \le r \le r_0$, for some fixed integers $q_0,r_0 \ge 2$. We will prove the cases $(q_0+1,r)$ for $2 \le r \le _0$. We start by using \cref{higher-order-mult-ind}:
\begin{align}
\resizebox{!}{1.2cm}{\input{proof1A.tikz}}\quad&=\quad\resizebox{!}{1.2cm}{\input{proof1B.tikz}}\quad. \notag
\intertext{We then use the hypothesis of induction for the couple $(2,r)$ to write:}
&=\quad \underset{\substack{I_{1}+J_{1}=p_{1} \\ \cdots \\ I_l+J_l = p_r \\ I_{1} + \cdots + I_l\quad=\quad n_{1} +\cdots+ n_{q} \\ J_{1} +\ldots+ J_l= n_{q+1}}}{\sum} \resizebox{!}{1.2cm}{\input{proof1C.tikz}}\quad. \notag
\intertext{We use the hypothesis of induction for the couple $(q_0,r)$ to write:}
&\quad=\quad \underset{\substack{I_{1}+J_{1}=p_{1} \\ \cdots \\ I_l+J_l = p_l \\ I_{1} + \cdots + I_l = n_{1} +\cdots+ n_{q} \\ J_{1} +\ldots+ J_l = n_{q+1}}}{\sum} \underset{\substack{\underset{1 \le \alpha \le q}{\sum}i_{\alpha}^{\beta} ~=~ I_{\beta}~(\forall 1 \le \beta \le l) \\ \underset{1 \le \beta \le l}{\sum} i_{\alpha}^{\beta} ~=~ n_{\alpha}~(\forall 1 \le \alpha \le q)}}{\sum} \resizebox{!}{1.2cm}{\input{proof2.tikz}} \quad. \notag
\intertext{We can simplify using \cref{higher-order-comult-ind} and write:\footnotemark
}
&\quad=\quad \underset{\substack{I_{1}+J_{1}=p_{1} \\ \cdots \\ I_l+J_l = p_l \\ I_{1} + \cdots + I_l = n_{1} +\cdots+ n_{q} \\ J_{1} +\ldots+ J_l = n_{q+1}}}{\sum} \underset{\substack{\underset{1 \le \alpha \le q}{\sum}i_{\alpha}^{\beta} ~=~ I_{\beta}~(\forall 1 \le \beta \le l) \\ \underset{1 \le \beta \le l}{\sum} i_{\alpha}^{\beta} ~=~ n_{\alpha}~(\forall 1 \le \alpha \le q)}}{\sum} \resizebox{!}{1.2cm}{\input{proof3.tikz}}\quad. \label{first-ind-step-A}
\intertext{We now want to write:}
&\quad=\quad \underset{\substack{\underset{1 \le \alpha \le q+1}{\sum}u_{\alpha}^{\beta} ~=~ p_{\beta}~(\forall 1 \le \beta \le l) \\ \underset{1 \le \beta \le l}{\sum} u_{\alpha}^{\beta} ~=~ n_{\alpha}~(\forall 1 \le \alpha \le q+1)}}{\sum} \resizebox{!}{1.2cm}{\input{proof4.tikz}}\quad. \label{first-ind-step-B}
\end{align}
But \cref{first-ind-step-B} must be justified carefully. We will prove that the same terms appear in the sums in \cref{first-ind-step-A,first-ind-step-B} so that the latter equation holds.
\footnotetext{Note that this step doesn't use associativity of the multiplication which is not assumed in \cref{fig:bialg-axiom-short}. We simply use \cref{higher-order-mult-ind}.}

We first prove that every term in the sum in \cref{first-ind-step-A} is a term in the sum in \cref{first-ind-step-B}: Consider a term in the sum in \cref{first-ind-step-A} given by indexes $I_\beta$, $J_\beta$, $i_\alpha^\beta$. Set 
\begin{equation*}
u_\alpha^\beta=i_\alpha^\beta
\end{equation*}
for every $(\alpha,\beta) \in [1,q_0] \times [1,r]$ and 
\begin{equation*}
u_{q_0+1}^\beta=J_\beta
\end{equation*}
for every $\beta \in [1,r]$. We then obtain that
\begin{equation*}
\underset{1 \le \alpha \le q_0+1}{\sum}u_\alpha^\beta=\bigg(\underset{1 \le \alpha \le q_0}{\sum}i_\alpha^\beta\bigg)+J_\beta=I_\beta+J_\beta=p_\beta
\end{equation*}
for every $1 \le \beta \le r$, that
\begin{equation*}
\underset{1 \le \beta \le r}{\sum}u_\alpha^\beta=\underset{1 \le \beta \le r}{\sum}i_\alpha^\beta=n_\alpha
\end{equation*}
for every $1 \le \alpha \le q_0$, and that
\begin{equation*}
\underset{1 \le \beta \le r}{\sum}u_{q_0+1}^\beta=J_1+\dots+J_r=n_{q_0+1}.
\end{equation*}
Our term is thus also a term in the sum in \cref{first-ind-step-B}.

Conversely, consider a term in the sum in \cref{first-ind-step-B} given by indexes $(u_\alpha^\beta)$. We will show that this is a term in the sum in \cref{first-ind-step-A}. We set 
\begin{equation*}
i_\alpha^\beta=u_\alpha^\beta
\end{equation*}
for every $(\alpha,\beta) \in [1,q_0] \times [1,r]$,
\begin{equation*}
I_\beta=\underset{1 \le \alpha \le q_0}{\sum}u_\alpha^\beta
\end{equation*}
and
\begin{equation*}
J_\beta=u_{q_0+1}^\beta
\end{equation*}
for every $1 \le \beta \le r$. We obtain that
\begin{equation*}
I_\beta+J_\beta=\underset{1 \le \alpha \le q_0+1}{\sum}u_\alpha^\beta=p_\beta
\end{equation*}
for every $1 \le \beta \le r$, that 
\begin{equation*}
I_1+\dots+I_r=\underset{1 \le \alpha \le q_0}{\sum}u_\alpha^1+\dots+\underset{1 \le \alpha \le q_0}{\sum}u_\alpha^l=\underset{1 \le \beta \le r}{\sum}u_1^\beta+\dots+\underset{1 \le \beta \le r}{\sum}u_{q_0}^\beta=n_1+\dots+n_{q_0},
\end{equation*}
that
\begin{equation*}
J_1+\dots+J_l=\underset{1 \le \beta \le r}{\sum}u_{q_0+1}^\beta=n_{q_0+1},
\end{equation*}
that
\begin{equation*}
\underset{1 \le \alpha \le q_0}{\sum}i_\alpha^\beta=\underset{1 \le \alpha \le q_0}{\sum}u_\alpha^\beta=I_\beta
\end{equation*}
for every $1 \le \beta \le r$ and that
\begin{equation*}
\underset{1 \le \beta \le r}{\sum}i_\alpha^\beta=\underset{1 \le \beta \le r}{\sum}u_\alpha^\beta=n_\alpha
\end{equation*}
for every $1 \le \alpha \le q_0$. Our term is thus also a term in the sum in \cref{first-ind-step-A}.

We have proved that the sums in \cref{first-ind-step-A} and in \cref{first-ind-step-B} contain exactly the same terms. These two sums are thus equal. It concludes the proof of the cases $(q_0+1,r)$ for $2 \le r \le r_0$. \cref{prop:higher-order-mul-comul} thus holds for all $(q,r)$ with $2 \le q \le q_0+1$ and $2 \le r \le r_0$. Applying an up-down symmetry to the above induction step, we can prove the cases $(q,r_0+1)$ for $2 \le q \le q_0+1$. We conclude that \cref{prop:higher-order-mul-comul} holds for all $(q,r)$ with $2 \le q \le q_0+1$ and $2 \le r \le r_0+1$.

By the principle of induction, \cref{prop:higher-order-mul-comul} holds for all couples $(q,r)$ with $q,r \ge 2$.
\end{proof}
\begin{corollary} \label{HIGH2}
This identity holds for every $q \ge 2$:
\begin{equation*}
\resizebox{!}{1.2cm}{\input{sumpermA-.tikz}}\quad=\quad \underset{\sigma \in S_{q}}{\sum} \resizebox{!}{1.2cm}{\input{sumpermB.tikz}}\quad.
\end{equation*}
\end{corollary}
\begin{proof}
We use \cref{HIGH1} with $q=r \ge 2$ and $n_{i}=p_{i}=1$ for every $1 \le i \le q$. We obtain:
\begin{equation*}
\resizebox{!}{1.2cm}{\input{sumpermA-.tikz}} \quad=\quad \underset{\substack{\underset{\alpha}{\sum}i_{\alpha}^{\beta} ~=~ 1 \\ \underset{\beta}{\sum} i_{\alpha}^{\beta} ~=~ 1}}{\sum} \resizebox{!}{1.2cm}{\input{highnabladeltaB1.tikz}}\quad.
\end{equation*}
The sum on the RHS is indexed by the set of all matrices $(i_{\alpha}^{\beta}) \in \mathbb{N}^{n \times n}$ which contain only $0$ and $1$ entries and such that there is a single $1$ in each column and in each row. These matrices are called permutation matrices. In string diagrams, it amounts to the fact that in every summand on the RHS, every input is connected to exactly one output via a string labelled $1$,  with no two inputs mapped to the same output.

We can replace the strings labelled $0$ on the RHS with $\epsilon;\eta$ using \cref{connected-intro} which is one of the axioms in \cref{fig:bialg-axiom-short}, and then use unitality and counitaly. We end up with every summand on the RHS being one of the permutations in $S_q$. Moreover, each of these permutations appears exactly once in the sum.
\end{proof}
\begin{proposition} \label{HIGH3}
This identity ($n$ times $1$ on the LHS) holds for every $n \ge 2$:
\begin{equation*}
\resizebox{!}{1.2cm}{\input{highdeltanablaA-.tikz}}\quad=\quad\scalebox{0.8}{$n!$} \resizebox{!}{1.2cm}{\begin{tikzpicture}
	\begin{pgfonlayer}{nodelayer}
		\node [style=none] (28) at (0, 3.5) {$n$};
		\node [style=none] (29) at (0, -3.5) {$n$};
		\node [style=none] (30) at (0, -3) {};
		\node [style=none] (31) at (0, 3) {};
	\end{pgfonlayer}
	\begin{pgfonlayer}{edgelayer}
		\draw (31.center) to (30.center);
	\end{pgfonlayer}
\end{tikzpicture}
}\quad.
\end{equation*}
\end{proposition}
\begin{proof}
By induction on $n$. For $n=2$, this is one of the axioms in \cref{fig:bialg-axiom-short}. Suppose that the identity holds for some $n \ge 2$. Then 
\begin{align*}
&~\Delta_{1,\ldots,1};\nabla_{1,\ldots,1}~(n+1\text{ times }1) \\
=&~ \Delta_{n,1};\Delta_{1,\ldots,1} \otimes \mathsf{id}_{A_1};\nabla_{1,\ldots,1} \otimes \mathsf{id}_{A_1}; \nabla_{n,1}~(n\text{ times }1) \\
=&~ n!\Delta_{n,1};\nabla_{n,1} \\
=&~n! \binom{n+1}{1}\cdot \mathsf{id}_{A_n} \\
=&~ n!(n+1)\cdot \mathsf{id}_{A_n} \\
=&~ (n+1)!\cdot \mathsf{id}_{A_n}.\qedhere
\end{align*}
\end{proof}
We conclude from \cref{HIGH2,HIGH3} that $((A_n),(r_n),(s_n))$ is a permutation splitting. We sum up in the next proposition what has been proved in this subsection.
\begin{proposition} \label{prop-short-to-perm}
Let $(\mathsf{C},\otimes,I)$ be a $\mathbb{Q}_{\ge 0}$-linear symmetric monoidal category. Let $(A_n) \in \mathsf{C}^\mathbb{N}$ and $(\nabla_{n,p}\colon A_n \otimes A_p \rightarrow A_{n+p})_{n,p \ge 0}$, $(\Delta_{n,p}\colon A_{n+p} \rightarrow A_n \otimes A_p)_{n,p \ge 0}$, $\eta\colon I \rightarrow A_0$, $\epsilon\colon A_0 \rightarrow I$ be maps such that the eight axioms in \cref{fig:bialg-axiom-short} hold. If we define the maps $(r_n)_{n \ge 0}$ and $(s_n)_{n \ge 0}$ as in \cref{main-theorem}, then we obtain a permutation splitting.
\end{proposition}
\subsection{From permutation splitting to the full definition of a binomial bimonoid} \label{SEVEN}
In this subsection, we verify that from any permutation splitting, we obtain a binomial bimonoid defined as in \cref{main-theorem}.
We write $\nabla_{n,p}^*$ and $\Delta_{n,p}^*$ and add stars in the same way in the string diagrams to denote the putative multiplications and comultiplications which are defined starting from a permutation splitting $(A_n,s_{n},r_{n})_{n \ge 0}$, i.e.~
\begin{equation*} 
\nabla_{n,p}^* :=s_n \otimes s_p;r_{n+p}
\end{equation*}
and
\begin{equation*}
\Delta_{n,p}^* := \binom{n+p}{n}s_{n+p};r_n \otimes r_p
\end{equation*}
for all $n,p \ge 0$. We also write $\eta^*$ and $\epsilon^*$ for the putative unit and counit, which are defined as
\begin{equation*}
\eta^*:=r_0 \text{ and }\epsilon^*:=s_0.
\end{equation*}
The proof of the following proposition is the most difficult part in this subsection.
\begin{proposition}\label{PROP1}
We have for all $n,p,q,r \in \mathbb{N}$ such that $n+p=q+r$:
\begin{equation} \label{AAAB}
\resizebox{!}{1.2cm}{\input{nabladeltaA-.tikz}}\quad=\quad 
\scalebox{1}{$\underset{\substack{a,b,c,d \ge 0 \\ a+b=n \\ c+d=p \\ a+c=q \\ b+d=r}}{\sum}$} 
\resizebox{!}{1.2cm}{\input{nabladeltaB-.tikz}}\quad.
\end{equation}
\end{proposition}
\begin{proof}
Let $n,p,q,r \in \mathbb{N}$ such that $n+p=q+r$.

\textbf{First step: we reduce \cref{AAAB} to \cref{eq-to-prove-in-perm-to-bialg}.} We compute\footnote{In this first step, we use the notation $e_k:=r_k;s_k=\frac{1}{k!}\underset{\sigma \in S_k}{\sum}\sigma\colon A_1^{\otimes k} \rightarrow A_1^{\otimes k}$.}
\begin{align*}
\nabla_{n,p}^*;\Delta_{q,r}^*=&~\binom{n+p}{q}s_n \otimes s_p;r_{n+p};s_{n+p};r_q \otimes r_r \\
=&~\binom{n+p}{q}s_n \otimes s_p;e_{n+p};r_q \otimes r_r.
\end{align*}
It follows that
\begin{align} \label{proof-for-nabla-delta-1}
r_n \otimes r_p;\nabla_{n,p}^*;\Delta_{q,r}^*;s_q \otimes s_r=&~\binom{n+p}{q}e_n \otimes e_p;e_{n+p};e_q \otimes e_r \notag \\
=&~\binom{n+p}{q}e_{n+p} \notag \\
=&~\frac{1}{q!r!}\underset{\sigma \in S_{n+p}}{\sum}\sigma.
\end{align}
Let $(a,b,c,d) \in \mathbb{N}^4$ such that $a+b=n$, $c+d=p$, $a+c=q$ and $b+d=r$. We have
\begin{align*}
&~\Delta_{a,b}^* \otimes \Delta_{c,d}^*;\mathsf{id}_{A_a} \otimes \gamma_{A_b,A_c} \otimes \mathsf{id}_{A_d};\nabla_{a,c}^* \otimes \nabla_{b,d}^* \\
=&~\binom{n}{a}\binom{p}{c}s_n \otimes s_p;r_a \otimes r_b \otimes r_c \otimes r_d;\mathsf{id}_{A_a} \otimes \gamma_{A_b,A_c} \otimes \mathsf{id}_{A_d};s_a \otimes s_c \otimes s_b \otimes s_d;r_q \otimes r_r \\
=&~\binom{n}{a}\binom{p}{c}s_n \otimes s_p;\mathsf{id}_{A_1^{\otimes a}} \otimes \gamma_{A_1^{\otimes b},A_1^{\otimes c}} \otimes \mathsf{id}_{A_1^{\otimes d}};e_a \otimes e_c \otimes e_b \otimes e_d;r_q \otimes r_r.
\end{align*}
We warn the reader that in the above computations, we used the symbol $r$ to denote both a nonnegative integer and the morphisms $r_n\colon A^{\otimes n} \rightarrow A_n$. It should not cause confusion since when $r$ denotes a morphism, this morphism is always written with a subscript as in $r_n$.

From the above computations, it follows that
\begin{align} \label{proof-for-nabla-delta-2}
&~r_n \otimes r_p;\Delta_{a,b}^* \otimes \Delta_{c,d}^*;\mathsf{id}_{A_a} \otimes \gamma_{A_b,A_c} \otimes \mathsf{id}_{A_d};\nabla_{a,c}^* \otimes \nabla_{b,d}^*;s_q \otimes s_r \notag \\
=&~\binom{n}{a}\binom{p}{c}e_n \otimes e_p;\mathsf{id}_{A_1^{\otimes a}} \otimes \gamma_{A_1^{\otimes b},A_1^{\otimes c}} \otimes \mathsf{id}_{A_1^{\otimes d}};e_a \otimes e_c \otimes e_b \otimes e_d;e_q \otimes e_r \notag \\
=&~\binom{n}{a}\binom{p}{c}e_n \otimes e_p;\mathsf{id}_{A_1^{\otimes a}} \otimes \gamma_{A_1^{\otimes b},A_1^{\otimes c}} \otimes \mathsf{id}_{A_1^{\otimes d}};e_q \otimes e_r \notag \\
=&~\binom{n}{a}\binom{p}{c}\frac{1}{n!p!q!r!}\underset{(\alpha,\beta,\gamma,\delta) \in S_n \times S_p \times S_q \times S_r}{\sum}\alpha \otimes \beta;\mathsf{id}_{A_1^{\otimes a}} \otimes \gamma_{A_1^{\otimes b},A_1^{\otimes c}} \otimes \mathsf{id}_{A_1^{\otimes d}};\gamma \otimes \delta \notag \\
=&~\frac{1}{a!b!c!d!q!r!}\underset{(\alpha,\beta,\gamma,\delta) \in S_n \times S_p \times S_q \times S_r}{\sum}\alpha \otimes \beta;\mathsf{id}_{A_1^{\otimes a}} \otimes \gamma_{A_1^{\otimes b},A_1^{\otimes c}} \otimes \mathsf{id}_{A_1^{\otimes d}};\gamma \otimes \delta.
\end{align}
Since $r_n \otimes r_p$ is a split epimorphisms and $s_q \otimes s_r$ is a split monomorphism, we deduce from \cref{proof-for-nabla-delta-1,proof-for-nabla-delta-2} that to prove \cref{AAAB}, it suffices to prove that
\begin{equation} \label{eq-to-prove-in-perm-to-bialg}
\underset{\substack{a,b,c,d \ge 0 \\ a+b=n \\ c+d=p \\ a+c=q \\ b+d=r}}{\sum}\frac{1}{a!b!c!d!}\underset{(\alpha,\beta,\gamma,\delta) \in S_n \times S_p \times S_q \times S_r}{\sum}\alpha \otimes \beta;\mathsf{id}_{A_1^{\otimes a}} \otimes \gamma_{A_1^{\otimes b},A_1^{\otimes c}} \otimes \mathsf{id}_{A_1^{\otimes d}};\gamma \otimes \delta=\underset{\sigma \in S_{n+p}}{\sum}\sigma.
\end{equation}

\textbf{Second step: reducing \cref{eq-to-prove-in-perm-to-bialg} to \cref{eq-to-prove-in-perm-to-bialg-rewritten}.}
For every tuple $(a,b,c,d) \in \mathbb{N}^4$ such that $a+b=n$, $c+d=p$, $a+c=q$ and $b+d=r$, we define the function
\begin{equation*}
\phi_{a,b,c,d}\colon S_n \times S_p \times S_q \times S_r \rightarrow S_{n+p}(A_1)
\end{equation*}
by the formula
\begin{equation} \label{formula-for-phi}
\phi_{a,b,c,d}(\alpha,\beta,\gamma,\delta)=\alpha \otimes \beta;\mathsf{id}_{A_1^{\otimes a}} \otimes \gamma_{A_1^{\otimes b},A_1^{\otimes c}} \otimes \mathsf{id}_{A_1^{\otimes d}};\gamma \otimes \delta.
\end{equation}
We warn the reader that \cref{formula-for-phi} uses the symbol $\gamma$ to denote both a permutation in $S_q$ and the exchange. It should not cause confusion since in formulas the exchange is always written with subscripts as in $\gamma_{A,B}$.

\cref{eq-to-prove-in-perm-to-bialg} can then be written
\begin{equation} \label{eq-given-by-F}
\underset{\substack{a,b,c,d \ge 0 \\ a+b=n \\ c+d=p \\ a+c=q \\ b+d=r}}{\sum}\frac{1}{a!b!c!d!}\underset{(\alpha,\beta,\gamma,\delta) \in S_n \times S_p \times S_q \times S_r}{\sum}\phi_{a,b,c,d}(\alpha,\beta,\gamma,\delta)=\underset{\sigma \in S_{n+p}}{\sum}\sigma.
\end{equation}
We can rewrite the above equation as
\begin{equation} \label{the-equation-to-be-proved}
\underset{\sigma \in S_{n+p}(A_1)}{\sum}\Big(
\underset{\substack{a,b,c,d \ge 0 \\ a+b=n \\ c+d=p \\ a+c=q \\ b+d=r}}{\sum}\frac{|\phi_{a,b,c,d}^{-1}(\sigma)|}{a!b!c!d!}\Big)\sigma=\underset{\sigma \in S_{n+p}(A_1)}{\sum}\sigma
\end{equation}
where $S_{n+p}(A_1)$ denotes the set of all the morphisms of the form
\[
\sigma:A_1^{\otimes (n+p)} \rightarrow A_1^{\otimes (n+p)}
\]
for any $\sigma \in S_{n+p}$. A caveat is that we cannot identify $S_{n+p}(A_1)$ with $S_{n+p}$ so that it would not be easy to directly prove \cref{the-equation-to-be-proved}. We have a function
\begin{equation*}
S_{n+p} \rightarrow S_{n+p}(A_1)
\end{equation*}
but no guarantee that this is a bijection.

That's why we will rather work in the initial pointed $\mathbb{Q}_{\ge 0}$-linear symmetric strict monoidal category $((\mathcal{I},\otimes,0),1)$ defined in \cref{app:initial} in which we can reason more easily.

We will thus consider the function
\begin{equation*}
\phi_{a,b,c,d}'\colon S_n \times S_p \times S_q \times S_r \rightarrow S_{n+p}(1) \simeq S_{n+p},
\end{equation*}
where $S_{n+p}(1) \simeq S_{n+p}$ is a bijection,\footnote{In the same way as above, $S_{n+p}(1)$ denotes the set of all the morphisms of the form
\[
\sigma\colon 1^{\otimes (n+p)}=n+p \rightarrow 1^{\otimes (n+p)}=n+p
\]
in $(\mathcal{I},\otimes,0)$ for any $\sigma \in S_{n+p}$.} defined by the formula
\begin{equation} \label{formula-for-phi'}
\phi_{a,b,c,d}'(\alpha,\beta,\gamma,\delta)=\alpha \otimes \beta;\mathsf{id}_{a} \otimes \gamma_{b,c} \otimes \mathsf{id}_{d};\gamma \otimes \delta
\end{equation}
and prove that
\begin{equation} \label{eq-to-prove-in-perm-to-bialg-rewritten}
\underset{\sigma \in S_{n+p}(1)}{\sum}\Big(
\underset{\substack{a,b,c,d \ge 0 \\ a+b=n \\ c+d=p \\ a+c=q \\ b+d=r}}{\sum}\frac{|(\phi_{a,b,c,d}')^{-1}(\sigma)|}{a!b!c!d!}\Big)\sigma=\underset{\sigma \in S_{n+p}(1)}{\sum}\sigma.
\end{equation}
It will imply that
\begin{equation*}
\underset{\substack{a,b,c,d \ge 0 \\ a+b=n \\ c+d=p \\ a+c=q \\ b+d=r}}{\sum}\frac{1}{a!b!c!d!}\underset{(\alpha,\beta,\gamma,\delta) \in S_n \times S_p \times S_q \times S_r}{\sum}\phi_{a,b,c,d}'(\alpha,\beta,\gamma,\delta)=\underset{\sigma \in S_{n+p}(1)}{\sum}\sigma.
\end{equation*}
Applying to the above equation the unique $\mathbb{Q}_{\ge 0}$-linear symmetric strong monoidal functor
\begin{equation*}
F\colon (\mathcal{I},\otimes,0) \rightarrow (\mathsf{C},\otimes,I)
\end{equation*}
such that $F(1)=A_1$, from \cref{cons-mac}, we will obtain that \cref{eq-given-by-F} holds. Note that
\begin{equation*}
F(\phi_{a,b,c,d}'(\alpha,\beta,\gamma,\delta))=\phi_{a,b,c,d}(\alpha,\beta,\gamma,\delta).
\end{equation*}

\textbf{Third step: given $\sigma \in S_{n+p}$, we determine the unique tuple $(a,b,c,d) \in \mathbb{N}^4$ such that $(\phi_{a,b,c,d}')^{-1}(\sigma)$ can be nonempty.} Starting from this step, we will identify $S_{n+p}(1)$ with $S_{n+p}$. Let $\sigma \in S_{n+p}$ and $(a,b,c,d) \in S_n \times S_p \times S_q \times S_r$ such that $a+b=n$, $c+d=p$, $a+c=q$ and $b+d=r$.
Suppose that $(\alpha,\beta,\gamma,\delta) \in S_n \times S_p \times S_q \times S_r$ is such that
\begin{equation*}
\sigma=\phi'_{a,b,c,d}(\alpha,\beta,\gamma,\delta).
\end{equation*}
For the computations of $\sigma(i)$ below, it could be helpful to look at \cref{fig:proof-sum-perm}.

Let $i \in [1,n]$. If $\alpha(i) \in [1,a]$, then 
\begin{align*}
\sigma(i)=&~(\gamma \otimes \delta \circ \mathsf{id}_a \otimes \gamma_{b,c} \otimes \mathsf{id}_d)((\alpha \otimes \beta)(i)) \\
=&~(\gamma \otimes \delta \circ \mathsf{id}_a \otimes \gamma_{b,c} \otimes \mathsf{id}_d)(\alpha(i)) \\
=&~(\gamma \otimes \delta)(\alpha(i)) \\
=&~\gamma(\alpha(i)) \in [1,q].
\end{align*}
If $\alpha(i) \in [a+1,n]$, then
\begin{align*}
\sigma(i)=&~(\gamma \otimes \delta \circ \mathsf{id}_{a} \otimes \gamma_{b,c} \otimes \mathsf{id}_{d})((\alpha \otimes \beta)(i)) \\
=&~(\gamma \otimes \delta \circ \mathsf{id}_a \otimes \gamma_{b,c} \otimes \mathsf{id}_d)(\alpha(i)) \\
=&~(\gamma \otimes \delta)(c+\alpha(i)) \\
=&~\delta(\alpha(i)-a)+q \in [q+1,n+p].
\end{align*}
We conclude that for every $i \in [1,n]$, we have $\sigma(i) \in [1,q]$ if{f} $\alpha(i) \in [1,a]$. That is, we have
\begin{equation*}
\alpha^{-1}([1,a])=\sigma^{-1}([1,q]) \cap [1,n].
\end{equation*}
But since $\alpha$ is a bijection, we obtain $a=|[1,a]|=|\alpha^{-1}([1,a])|=|\sigma^{-1}([1,q]) \cap [1,n]|$. Thus
\begin{equation} \label{eq-for-a}
a=|\sigma^{-1}([1,q]) \cap [1,n]|.
\end{equation}
Recall that 
\begin{equation} \label{eq-b-c-d-from-a-in-perm-to-bialg}
b=n-a,~c=q-a \text{ and }d=p-q+a.
\end{equation}
It follows that $(a,b,c,d)$ is the unique tuple in $\mathbb{N}^4$ determined by \cref{eq-for-a,eq-b-c-d-from-a-in-perm-to-bialg}.

Given $\sigma \in S_{n+p}$, we will write $(a_\sigma,b_\sigma,c_\sigma,d_\sigma)$ for the tuple in $\mathbb{N}^4$ given by
\begin{equation*}
a_\sigma=|\sigma^{-1}([1,q]) \cap [1,n]|,
\end{equation*}
\begin{equation*}
b_\sigma=n-a_\sigma,~c_\sigma=q-a_\sigma \text{ and }d_\sigma=p-q+a_\sigma.
\end{equation*}
We have proved that for all $\sigma \in S_{n+p}$ and $(a,b,c,d) \in \mathbb{N}^4$ such that $a+b=n$, $c+d=p$, $a+c=q$ and $b+d=r$, we have 
\begin{equation} \label{the-implication}
\big(\phi'_{a,b,c,d}(\sigma) \neq \emptyset\big) \Rightarrow \big((a,b,c,d)=(a_\sigma,b_\sigma,c_\sigma,d_\sigma)\big).
\end{equation}

\textbf{Fourth step: simplification of \cref{eq-to-prove-in-perm-to-bialg-rewritten} and equivalent condition.}
\cref{the-implication} implies that we can rewrite \cref{eq-to-prove-in-perm-to-bialg-rewritten} as
\begin{equation*}
\underset{\sigma \in S_{n+p}}{\sum}\frac{|(\phi'_{a_\sigma,b_\sigma,c_\sigma,d_\sigma})^{-1}(\sigma)|}{a_\sigma!b_\sigma!c_\sigma!d_\sigma!}\sigma=\underset{\sigma \in S_{n+p}}{\sum}\sigma
\end{equation*}
which is equivalent to
\begin{equation*}
\frac{|(\phi'_{a_\sigma,b_\sigma,c_\sigma,d_\sigma})^{-1}(\sigma)|}{a_{\sigma}!b_{\sigma}!c_{\sigma}!d_{\sigma}!}=1
\end{equation*}
for every $\sigma \in S_{n+p}$, which is equivalent to
\begin{equation*}
|(\phi'_{a_\sigma,b_\sigma,c_\sigma,d_\sigma})^{-1}(\sigma)|=a_\sigma!b_\sigma!c_\sigma!d_\sigma!
\end{equation*}
for every $\sigma \in S_{n+p}$, which is equivalent to the existence of a bijection
\begin{equation} \label{the-bijec}
(\phi'_{a_\sigma,b_\sigma,c_\sigma,d_\sigma})^{-1}(\sigma) \simeq S_{a_\sigma} \times S_{b_\sigma} \times S_{c_\sigma} \times S_{d_\sigma}.
\end{equation}
\begin{figure}
\begin{framed}
\begin{equation*}
\resizebox{!}{4cm}{\input{big-drawing.tikz}}
\end{equation*}
\caption{String diagram representation of $\psi_\sigma(u,v,w,x)=(\alpha,\beta,\gamma,\delta) \in (\phi'_{a_\sigma,b_\sigma,c_\sigma,d_\sigma})^{-1}(\sigma)$.~\\~\\The permutations $\alpha$ and $\beta$ are obtained from $u,v$ and $w,x$ as depicted. Then, there exist unique $\gamma$ and $\delta$ such that $\phi'_{a_\sigma,b_\sigma,c_\sigma,d_\sigma}(\alpha,\beta,\gamma,\delta)$, which is the permutation represented by the whole diagram, is equal to $\sigma$.} 
\label{fig:proof-sum-perm}
\end{framed}
\end{figure}

\textbf{Fifth step: for every $\sigma \in S_{n+p}$, we define a map $\psi_\sigma\colon S_{a_\sigma} \times S_{b_\sigma} \times S_{c_\sigma} \times S_{d_\sigma} \rightarrow S_n \times S_p \times S_q \times S_r$.} We will later prove that $\psi_\sigma$ is injective (in the sixth step) and that $\mathrm{im}\,\psi_\sigma=(\phi'_{a_\sigma,b_\sigma,c_\sigma,d_\sigma})^{-1}(\sigma)$. These two facts will imply the desired bijection \labelcref{the-bijec}.

Let $\sigma \in S_{n+p}$. We will explain how the function $\psi_\sigma$ is defined. It will be helpful to look at \cref{fig:proof-sum-perm}. We write
\begin{equation*}
\sigma^{-1}([1,q]) \cap [1,n]=\{i_1^1<i_2^1<\dots<i_{a_\sigma}^1\},
\end{equation*}
\begin{equation*}
\sigma^{-1}([q+1,n+p]) \cap [1,n]=\{i_1^2<i_2^2<\dots<i_{b_\sigma}^2\},
\end{equation*}
\begin{equation*}
\sigma^{-1}([1,q]) \cap [n+1,n+p]=\{i_1^3<i_2^3<\dots<i_{c_\sigma}^3\},
\end{equation*}
\begin{equation*}
\sigma^{-1}([q+1,n+p]) \cap [n+1,n+p]=\{i_1^4<i_2^4<\dots<i_{d_\sigma}^4\}.
\end{equation*}
Let $(u,v,w,x) \in S_{a_\sigma} \times S_{b_\sigma} \times S_{c_\sigma} \times S_{d_\sigma}$. We define 
\begin{equation*}
\psi_\sigma(u,v,w,x)=(\alpha,\beta,\gamma,\delta)
\end{equation*}
where
$\alpha \in S_n$, $\beta \in S_p$, $\gamma \in S_q$ and $\delta \in S_r$ are given by the following formulas:\footnote{We must check that $\alpha,\beta,\gamma,\delta$ are permutations in the appropriate symmetric groups. It suffices to first verify that the formulas for $\alpha,\beta,\gamma,\delta$ define endofunctions respectively on $[1,n]$, $[1,p]$, $[1,q]$ and $[1,r]$, and then to verify that the following formulas provide the inverses of $\alpha,\beta,\gamma,\delta$:
\begin{equation*}
\left\{
\begin{aligned}
\alpha^{-1}(l)&=i^1_{u^{-1}(l)}~&\text{for every }l \in [1,a_\sigma], \\
\alpha^{-1}(l)&=i^2_{v^{-1}(l-a_\sigma)}~&\text{for every }l \in [a_\sigma+1,n],
\end{aligned}
\right.
\qquad\qquad
\left\{
\begin{aligned}
\beta^{-1}(l)&=i^3_{w^{-1}(l)}-n ~&\text{for every }l \in [1,c_\sigma], \\
\beta^{-1}(l)&=i^4_{x^{-1}(l-c_\sigma)}-n~&\text{for every }l \in [c_\sigma+1,p],
\end{aligned}
\right.
\end{equation*}
\begin{equation*}
\left\{
\begin{aligned}
\gamma^{-1}(l)&=u(m)\text{ where }m\text{ is given by }\sigma^{-1}(l)=i^1_m~&\text{for every }l \in [1,q] \cap \sigma([1,n]), \\
\gamma^{-1}(l)&=w(m)+a_\sigma\text{ where }m\text{ is given by }\sigma^{-1}(l)=i^3_m~&\text{for every }l \in [1,q] \cap \sigma([n+1,n+p]),
\end{aligned}
\right.
\end{equation*}
\begin{equation*}
\left\{
\begin{aligned}
\delta^{-1}(l)&=v(m)\text{ where }m\text{ is given by }\sigma^{-1}(l+q)=i^2_m ~&\text{for every }l \in [1,r] \text{ such that }l+q \in \sigma([1,n]), \\
\delta^{-1}(l)&=x(m)+b_\sigma\text{ where }m\text{ is given by }\sigma^{-1}(l+q)=i^4_m~&\text{for every }l \in [1,r] \text{ such that } l+q \in \sigma([n+1,n+p]).
\end{aligned}
\right.
\end{equation*}
}
\begin{equation*}
\left\{ 
\begin{aligned}
\alpha(i^1_k)&=u(k)~&\text{for every }k \in [1,a_\sigma], \\
\alpha(i^2_k)&=v(k)+a_\sigma~&\text{for every }k \in [1,b_\sigma],
\end{aligned}
\right.
\qquad\qquad
\left\{ 
\begin{aligned}
\beta(i^3_k-n)&=w(k)~&\text{for every }k \in [1,c_\sigma], \\
\beta(i^4_k-n)&=x(k)+c_\sigma~&\text{for every }k \in [1,d_\sigma],
\end{aligned}
\right.
\end{equation*}
\begin{equation*}
\left\{ 
\begin{aligned}
\gamma(k)&=\sigma(i^1_{u^{-1}(k)})~&\text{for every }k \in [1,a_\sigma], \\
\gamma(k)&=\sigma(i^3_{w^{-1}(k-a_\sigma)})~&\text{for every }k \in [a_\sigma+1,q],
\end{aligned}
\right.
\end{equation*}
\begin{equation*}
\left\{ 
\begin{aligned}
\delta(k)&=\sigma(i^2_{v^{-1}(k)})-q~&\text{for every }k \in [1,b_\sigma], \\
\delta(k)&=\sigma(i^4_{x^{-1}(k-b_\sigma)})-q~&\text{for every }k \in [b_\sigma+1,r].
\end{aligned}
\right.
\end{equation*}

\textbf{Sixth step: we prove that $\psi_\sigma$ is injective.}
Let $\sigma \in S_{n+p}$ and let $i_k^j$ for $j \in [1,4]$ be defined as in the third step. Let $(u,v,w,x)$ and $(u',v',w',x')$ in $S_{a_\sigma} \times S_{b_\sigma} \times S_{c_\sigma} \times S_{d_\sigma}$. We write
\begin{equation*}
\psi_\sigma(u,v,w,x)=(\alpha,\beta,\gamma,\delta)
\end{equation*}
and
\begin{equation*}
\psi_\sigma(u',v',w',x')=(\alpha',\beta',\gamma',\delta').
\end{equation*}
Suppose that
\begin{equation*}
(\alpha,\beta,\gamma,\delta)=(\alpha',\beta',\gamma',\delta').
\end{equation*}
We obtain from the fifth step that
\begin{equation*}
\left\{ 
\begin{aligned}
u(k)&=\alpha(i^1_k)=\alpha'(i^1_k)=u'(k)~&\text{for every }k \in [1,a_\sigma], \\
v(k)+a_\sigma&=\alpha(i^2_k)=\alpha'(i^2_k)=v'(k)+a_\sigma~&\text{for every }k \in [1,b_\sigma],
\end{aligned}
\right.
\end{equation*}
\begin{equation*}
\left\{ 
\begin{aligned}
w(k)&=\beta(i^3_k-n)=\beta'(i^3_k-n)=w'(k)~&\text{for every }k \in [1,c_\sigma], \\
x(k)+c_\sigma&=\beta(i^4_k-n)=\beta'(i^4_k-n)=x'(k)+c_\sigma~&\text{for every }k \in [1,d_\sigma].
\end{aligned}
\right.
\end{equation*}
It follows that
\begin{equation*}
(u,v,w,x)=(u',v',w',x').
\end{equation*}

\textbf{Seventh step: some formulas for $\phi'$.}
Let $(a,b,c,d) \in \mathbb{N}^4$ such that $a+b=n$, $c+d=p$, $a+c=q$ and $b+d=r$. Let also $(\alpha,\beta,\gamma,\delta) \in S_n \times S_p \times S_q \times S_r$. Define the nonnegative integers $I^j_k$ as follows:
\begin{equation*}
\alpha^{-1}([1,a])=\{I^1_1<I^1_2<\dots<I^1_a\},
\end{equation*}
\begin{equation*}
\alpha^{-1}([a+1,n])=\{I_1^2<I_2^1<\dots<I_b^2\},
\end{equation*}
\begin{equation*}
\beta^{-1}([1,c])=\{I_1^3-n<I_2^3-n<\dots<I_c^3-n\},
\end{equation*}
\begin{equation*}
\beta^{-1}([c+1,p])=\{I_1^4-n<I_2^4-n<\dots<I_d^4-n\}.
\end{equation*}
By using \cref{formula-for-phi'}, we obtain that for every $j \in [1,4]$, we have
\begin{equation} \label{the-formula-for-phi'}
\phi'_{a,b,c,d}(\alpha,\beta,\gamma,\delta)(I_k^j)=
\left\{
\begin{aligned}
&\gamma(\alpha(I_k^j))~&\text{for every }k \in [1,a] \text{ if }j=1, \\
&\delta(\alpha(I_k^j)-a)+q~&\text{for every }k \in [1,b] \text{ if }j=2, \\
&\gamma(\beta(I_k^j-n)+a)~&\text{for every }k \in [1,c] \text{ if }j=3, \\
&\delta(\beta(I_k^j-n)+b-c)+q~&\text{for every }k \in [1,d] \text{ if }j=4.
\end{aligned}
\right.
\end{equation}

\textbf{Eighth step: we prove that $\mathrm{im}\,\psi_\sigma=(\phi'_{a_\sigma,b_\sigma,c_\sigma,d_\sigma})^{-1}(\sigma)$.} Let $\sigma \in S_{n+p}$.  Let $(i_k^j)$ for $j \in [1,4]$ be defined as in the fifth step.

\textbf{(i)} We first check that
$\mathrm{im}\,\psi_\sigma \subseteq (\phi'_{a_\sigma,b_\sigma,c_\sigma,d_\sigma})^{-1}(\sigma)$. Let $(u,v,w,x) \in S_{a_\sigma} \times S_{b_\sigma} \times S_{c_\sigma} \times S_{d_\sigma}$ and let 
\begin{align} \label{alpha-in-8}
(\alpha,\beta,\gamma,\delta):=\psi_\sigma(u,v,w,x) \in S_n \times S_p \times S_q \times S_r.
\end{align}
We must prove that
\begin{equation*}
\phi'_{a_\sigma,b_\sigma,c_\sigma,d_\sigma}(\alpha,\beta,\gamma,\delta)=\sigma.
\end{equation*}
Let $(I_k^j)$ for $j \in [1,4]$ be defined as in the seventh step with $(a,b,c,d)=(a_\sigma,b_\sigma,c_\sigma,d_\sigma)$ and $(\alpha,\beta,\gamma,\delta)$ given by \cref{alpha-in-8}. We want to prove that $(I_k^j)=(i_k^j)$. It suffices to prove that
\begin{equation*}
\alpha^{-1}([1,a_\sigma])=\{i_1^1,\dots,i_{a_{\sigma}}^1\},
\end{equation*}
\begin{equation*}
\alpha^{-1}([a_\sigma+1,n])=\{i_1^2,\dots,i_{b_\sigma}^2\},
\end{equation*}
\begin{equation*}
\beta^{-1}([1,c_\sigma])=\{i_1^3-n,\dots,i_{c_\sigma}^3-n\},
\end{equation*}
\begin{equation*}
\beta^{-1}([c_\sigma+1,p])=\{i_1^4-n,\dots,i_{d_\sigma}^4-n\}.
\end{equation*}
But these are implied by the following identities:
\begin{equation*}
\left\{ 
\begin{aligned}
\alpha(i^1_k)&=u(k) \in [1,a_\sigma]~&\text{for every }k \in [1,a_\sigma], \\
\alpha(i^2_k)&=v(k)+a_\sigma \in [a_\sigma+1,n]~&\text{for every }k \in [1,b_\sigma],
\end{aligned}
\right.
\end{equation*}
\begin{equation*}
\left\{ 
\begin{aligned}
\beta(i^3_k-n)&=w(k) \in [1,c_\sigma]~&\text{for every }k \in [1,c_\sigma], \\
\beta(i^4_k-n)&=x(k)+c_\sigma \in [c_\sigma+1,p]~&\text{for every }k \in [1,d_\sigma].
\end{aligned}
\right.
\end{equation*}
By using \cref{the-formula-for-phi'} and the definition of $\psi_\sigma$ from the fifth step, we obtain that for every $k \in [1,a_\sigma]$, we have
\begin{equation*}
\phi'_{a_\sigma,b_\sigma,c_\sigma,d_\sigma}(\alpha,\beta,\gamma,\delta)(i_k^1)=\gamma(\alpha(i_k^1))
=\gamma(u(k))
=\sigma(i^1_k),
\end{equation*}
for every $k \in [1,b_\sigma]$, we have
\begin{equation*}
\phi'_{a_\sigma,b_\sigma,c_\sigma,d_\sigma}(\alpha,\beta,\gamma,\delta)(i_k^2)=\delta(\alpha(i_k^2)-a_\sigma)+q
=\delta(v(k))+q
=\sigma(i^2_k),
\end{equation*}
for every $k \in [1,c_\sigma]$, we have
\begin{equation*}
\phi'_{a_\sigma,b_\sigma,c_\sigma,d_\sigma}(\alpha,\beta,\gamma,\delta)(i_k^3)=\gamma(\beta(i_k^3-n)+a_\sigma)
=\gamma(w(k)+a_\sigma)
=\sigma(i^3_k),
\end{equation*}
and for every $k \in [1,d_\sigma]$, we have
\begin{equation*}
\phi'_{a_\sigma,b_\sigma,c_\sigma,d_\sigma}(\alpha,\beta,\gamma,\delta)(i_k^4)=\delta(\beta(i_k^4-n)b_\sigma-c_\sigma)+q
=\delta(x(k)+b_\sigma)+q
=\sigma(i_k^4).
\end{equation*}
Since 
\begin{equation*}
[1,n+p]=\{i_k^1,~1 \le k \le a_\sigma\} \cup \{i_k^2,~1 \le k \le b_\sigma\} \cup \{i_k^3,~1 \le k \le c_\sigma\} \cup \{i_k^4,~1 \le k \le d_\sigma\},
\end{equation*}
we conclude that $\phi'_{a_\sigma,b_\sigma,c_\sigma,d_\sigma}(\alpha,\beta,\gamma,\delta)=\sigma$.

\textbf{(ii)} We will now prove that $(\phi'_{a_\sigma,b_\sigma,c_\sigma,d_\sigma})^{-1}(\sigma) \subseteq \mathrm{im}\,\psi_\sigma$. Let $(\alpha,\beta,\gamma,\delta) \in S_n \times S_p \times S_q \times S_r$ such that 
\begin{equation*}
\phi'_{a_\sigma,b_\sigma,c_\sigma,d_\sigma}(\alpha,\beta,\gamma,\delta)=\sigma.
\end{equation*}
Let $(I^j_k)$ for $j \in [1,4]$ be defined as in the seventh step with $(a,b,c,d)=(a_\sigma,b_\sigma,c_\sigma,d_\sigma)$ and the tuple $(\alpha,\beta,\gamma,\delta)$ we have just introduced. We define $(u,v,w,x) \in S_{a_\sigma} \times S_{b_\sigma} \times S_{c_\sigma} \times S_{d_\sigma}$ by the following formulas:\footnote{We must check that $u,v,w,x$ are permutations in the appropriate symmetric groups. It suffices to first verify that the formulas for $u,v,w,x$ define endofunctions respectively on $[1,a_\sigma]$, $[1,b_\sigma]$, $[1,c_\sigma]$ and $[1,d_\sigma]$, and then to verify that the following formulas provide the inverses of $u,v,w,x$:
\begin{equation*}
u^{-1}(l)=k~\text{such that }\alpha^{-1}(l)=I_k^1, \text{ for every }l \in [1,a_\sigma],
\end{equation*}
\begin{equation*}
v^{-1}(l)=k~\text{such that }\alpha^{-1}(a_\sigma+l)=I_k^2, \text{ for every }l \in [1,b_\sigma],
\end{equation*}
\begin{equation*}
w^{-1}(l)=k~\text{such that }\beta^{-1}(l)=I_k^3-n, \text{ for every }l \in [1,c_\sigma],
\end{equation*}
\begin{equation*}
x^{-1}(l)=k~\text{such that }\beta^{-1}(c_\sigma+l)=I_k^4-n, \text{ for every }l \in [1,d_\sigma].
\end{equation*}
}
\begin{equation*}
u(k)=\alpha(I_k^1)~\text{for every }k \in [1,a_\sigma],
\end{equation*}
\begin{equation*}
v(k)=\alpha(I_k^2)-a_\sigma~\text{for every }k \in [1,b_\sigma],
\end{equation*}
\begin{equation*}
w(k)=\beta(I_k^3-n)~\text{for every }k \in [1,c_\sigma],
\end{equation*}
\begin{equation*}
x(k)=\beta(I_k^4-n)-c_\sigma~\text{for every }k \in [1,c_\sigma].
\end{equation*}
It could be helpful to look at \cref{fig:proof-sum-perm} to understand these formulas.

We want to prove that $(I_k^j)=(i_k^j)$. For every $k \in [1,a_\sigma]$, we have $I_k^1 \in [1,n]$ (from the definitions of the $I_k^j$ in the seventh step) and we obtain from \cref{the-formula-for-phi'} that $\sigma(I_k^1) \in [1,q]$, it follows that $I_k^1 \in \sigma^{-1}([1,q]) \cap [1,n]=\{i_1^1<\dots<i^1_{a_\sigma}\}$.  In the same way, for every $k \in [1,b_\sigma]$, we have $I_k^2 \in [1,n]$ and $\sigma(I_k^2) \in [q+1,n+p]$. Thus $I_k^2 \in \sigma^{-1}([q+1,n+p]) \cap [1,n]=\{i_1^2<\dots<i^2_{b_\sigma}\}$. We also have for every $k \in [1,c_\sigma]$, that $I_k^3 \in [n+1,n+p]$ and $\sigma(I_k^3) \in [1,q]$, thus $I_k^3 \in \{i_1^3<\dots<i_{c_\sigma}^3\}$. Finally, for every $k \in [1,d_\sigma]$, we have $I_k^4 \in [n+1,n+p]$ and $\sigma(I_k^4) \in [q+1,n+p]$, thus $I_k^4 \in \{i_1^4<\dots<i_{d_\sigma}^4\}$. Since $|\{I_1^1<\dots<I_{a_\sigma}^1\}|=a_\sigma=|\{i_1^1<\dots<i_{a_\sigma}^1\}|$, we obtain that $\{I_1^1<\dots<I_{a_\sigma}^1\}=\{i_1^1<\dots<i_{a_\sigma}^1\}$ and thus $I_k^1=i_k^1$ for every $k \in [1,a_\sigma]$. In the same way, we obtain that $I_k^j=i_k^j$ for any $j \in [2,4]$ and any appropriate $k$.

Let 
\begin{equation*}
(\alpha',\beta',\gamma',\delta'):=\psi_\sigma(u,v,w,x).
\end{equation*}
We obtain from the fifth step and \cref{the-formula-for-phi'} that
\begin{equation*}
\left\{ 
\begin{aligned}
\alpha'(i^1_k)&=u(k)=\alpha(I^1_k)=\alpha(i^1_k)~&\text{for every }k \in [1,a_\sigma], \\
\alpha'(i^2_k)&=v(k)+a_\sigma=\alpha(I_k^2)=\alpha(i_k^2)~&\text{for every }k \in [1,b_\sigma], \\
\end{aligned}
\right.
\end{equation*}
\begin{equation*}
\left\{ 
\begin{aligned}
\beta'(i^3_k-n)&=w(k)=\beta(I^3_k-n)=\beta(i^3_k-n)~&\text{for every }k \in [1,c_\sigma], \\
\beta'(i^4_k-n)&=x(k)+c_\sigma=\beta(I_4^k-n)=\beta(i_4^k-n)~&\text{for every }k \in [1,d_\sigma],
\end{aligned}
\right.
\end{equation*}
\begin{equation*}
\left\{ 
\begin{aligned}
\gamma'(l)&=\sigma(i^1_{u^{-1}(l)})=\sigma(I^1_{u^{-1}(l)})
=\gamma(\alpha(I^1_{u^{-1}(l)}))
=\gamma(\alpha(\alpha^{-1}(l)))
=\gamma(l) \\
&\text{for every }l \in [1,a_\sigma], \\
\gamma'(l)&=\sigma(i^3_{w^{-1}(l-a_\sigma)})=\sigma(I^3_{w^{-1}(l-a_\sigma)})=\gamma(\beta(I^3_{w^{-1}(l-a_\sigma)}-n)+a_\sigma)=\gamma(\beta(\beta^{-1}(l-a_\sigma))+a_\sigma)=\gamma(l) \\
&\text{for every }l \in [a_\sigma+1,q],
\end{aligned}
\right.
\end{equation*}
\begin{equation*}
\left\{ 
\begin{aligned}
\delta'(l)&=\sigma(i^2_{v^{-1}(l)})-q=\sigma(I^2_{v^{-1}(l)})-q=\delta(\alpha(I^2_{v^{-1}(l)})-a_\sigma)=
\delta(\alpha(\alpha^{-1}(a_\sigma+l))-a_\sigma)=\delta(l) \\
&\text{for every }l \in [1,b_\sigma], \\
\delta'(l)&=\sigma(i^4_{x^{-1}(l-b_\sigma)})-q=\sigma(I^4_{x^{-1}(l-b_\sigma)})-q=\delta(\beta(I^4_{x^{-1}(l-b_\sigma)}-n)+b_\sigma-c_\sigma) \\
&\quad=\delta(\beta(\beta^{-1}(c_\sigma+l-b_\sigma))+b_\sigma-c_\sigma)=\delta(l) \\
&\quad\text{for every }l \in [b_\sigma+1,r].
\end{aligned}
\right.
\end{equation*}
It follows that
\begin{equation*}
(\alpha',\beta',\gamma',\delta')=(\alpha,\beta,\gamma,\delta),
\end{equation*}
thus
\begin{equation*}
(\alpha,\beta,\gamma,\delta) \in \mathrm{im}\,\psi_\sigma.
\end{equation*}

\textbf{Ninth step: conclusion.} Let $\sigma \in S_{n+p}$. From the sixth step, the map
\begin{equation*}
\psi_\sigma\colon S_{a_\sigma} \times S_{b_\sigma} \times S_{c_\sigma} \times S_{d_\sigma} \rightarrow S_n \times S_p \times S_q \times S_r
\end{equation*}
is injective and from the eighth step, we have $\mathrm{im}\,\psi_\sigma=(\phi'_{a_\sigma,b_\sigma,c_\sigma,d_\sigma})^{-1}(\sigma)$. We obtain a bijection
\begin{equation*}
(\phi'_{a_\sigma,b_\sigma,c_\sigma,d_\sigma})^{-1}(\sigma) \simeq S_{a_\sigma} \times S_{b_\sigma} \times S_{c_\sigma} \times S_{d_\sigma}.
\end{equation*}
The fourth step implies that \cref{eq-to-prove-in-perm-to-bialg-rewritten} is satisfied. The second step then implies that \cref{eq-to-prove-in-perm-to-bialg} is satisfied. Finally, we obtain from the first step that \cref{AAAB} is satisfied which is what we wanted to prove.
\end{proof}
The rest of the subsection will be much easier than the proof of \cref{PROP1}.
\begin{proposition} \label{1ASSOC}
For all $n,p,q \ge 0$, we have:
\begin{equation} \label{assoc-star}
\resizebox{!}{1.2cm}{\input{assocA-.tikz}}\quad=\quad\resizebox{!}{1.2cm}{\input{assocB-.tikz}}\quad.
\end{equation}
\end{proposition}
\begin{proof}
We prove \cref{assoc-star} for all $n,p,q \ge 0$ using in-line notations. We start with:
\begin{align*}
\nabla_{n,p}^* \otimes  \mathsf{id}_{A_q};\nabla_{n+p,q}^*=&~ s_{n} \otimes s_{p} \otimes  \mathsf{id}_{A_q};r_{n+p} \otimes  \mathsf{id}_{A_q};s_{n+p} \otimes s_{q};r_{n+p+q} \\
=&~s_{n} \otimes s_{p} \otimes s_{q};\left(\frac{1}{(n+p)!}\underset{\sigma \in S_{n+p}}{\sum}\sigma \right) \otimes  \mathsf{id}_{A_q};r_{n+p+q}.
\intertext{Using \cref{INVARIANT}, we can write:}
&= s_{n} \otimes s_{p} \otimes s_{q};r_{n+p+q}.
\end{align*}
But we also have:
\begin{align*}
 \mathsf{id}_{A_n} \otimes \nabla_{p,q}^*;\nabla_{n,p+q}^* &=  \mathsf{id}_{A_n} \otimes s_{p} \otimes s_{q}; \mathsf{id}_{A_n} \otimes r_{p+q};s_{n} \otimes s_{p+q};r_{n+p+q}\\
&= s_{n} \otimes s_{p} \otimes s_{q}; \mathsf{id}_{A_n} \otimes \left(\frac{1}{(p+q)!}\underset{\sigma \in S_{p+q}}{\sum}\sigma \right);r_{n+p+q}. \\
\intertext{Using \cref{INVARIANT}, we can write:}
&= s_{n} \otimes s_{p} \otimes s_{q};r_{n+p+q}.
\end{align*}
Hence the equality.
\end{proof}
\begin{remark} \label{2COASSOC}
The dual of \cref{1ASSOC} gives coassociativity.
\end{remark}
\begin{proposition} \label{4COMM}
For all $n,p \ge 0$, we have:
\begin{equation*}
\resizebox{!}{1.2cm}{\input{comm.tikz}}\quad.
\end{equation*}
\end{proposition}
\begin{proof}
By naturality of the exchange, we can write:
\begin{align*}
\gamma_{A_n,A_p};s_p \otimes s_n;r_{n+p}=&~s_n \otimes s_p;\gamma_{A_1^{\otimes n},A_1^{\otimes p}};r_{n+p}.
\intertext{Using \cref{INVARIANT}, we can then write:}
=&~s_n \otimes s_p;r_{n+p}.\qedhere
\end{align*}
\end{proof}
\begin{remark} \label{5COCOMM}
The dual of \cref{4COMM} gives cocommutativity.
\end{remark}
Note that \cref{INVARIANT} has been used in the proofs of both \cref{1ASSOC} and \cref{4COMM}. It will be used repeatedly throughout the remainder of the paper, in most cases without explicit reference.
\begin{proposition} \label{PROP2}
For all $n,p \ge 0$, we have:
\begin{equation*}
\hspace{-2.5cm}
\resizebox{!}{1.2cm}{\input{specialtyA-.tikz}}\quad=\quad\scalebox{1}{$\binom{n+p}{n}$} \resizebox{!}{1.2cm}{\begin{tikzpicture}
	\begin{pgfonlayer}{nodelayer}
		\node [style=none] (13) at (0, 3) {};
		\node [style=none] (14) at (0, -3) {};
		\node [style=none] (15) at (0, 3.5) {$n+p$};
		\node [style=none] (16) at (0, -3.5) {$n+p$};
	\end{pgfonlayer}
	\begin{pgfonlayer}{edgelayer}
		\draw (13.center) to (14.center);
	\end{pgfonlayer}
\end{tikzpicture}
}\quad.
\end{equation*}
\end{proposition}
\begin{proof}
We compute:
\begin{align*}
\Delta_{n,p}^*;\nabla_{n,p}^*
=&~\binom{n+p}{n}s_{n+p};r_{n} \otimes r_{p};s_{n} \otimes s_{p};r_{n+p} \\
=&~\binom{n+p}{n}s_{n+p};\Big(\frac{1}{n!}\underset{\sigma \in S_n}{\sum}\sigma\Big) \otimes \Big(\frac{1}{p!}\underset{\rho \in S_p}{\sum}\rho\Big);r_{n+p} \\
=&~\binom{n+p}{n}s_{n+p};\frac{1}{n!p!}\underset{\substack{\sigma \in S_n \\ \rho \in S_p}}{\sum} \sigma \otimes \rho;r_{n+p} \\
=&~\binom{n+p}{n}s_{n+p};r_{n+p} \\
=&~\binom{n+p}{n}  \mathsf{id}_{A_{n+p}}.\qedhere
\end{align*}
\end{proof}
\begin{proposition} \label{6UNIT}
The following identities hold:
\begin{align*}
\resizebox{!}{1.2cm}{\input{PRES18A-S.tikz}} &~~=~~ \resizebox{!}{1.2cm}{\begin{tikzpicture}
	\begin{pgfonlayer}{nodelayer}
		\node [style=none] (8) at (0, -3.5) {$n$};
		\node [style=none] (13) at (0, 3.5) {$n$};
		\node [style=none] (22) at (0, -3) {};
		\node [style=none] (25) at (0, 3) {};
		\node [style=none] (29) at (0, -3) {};
	\end{pgfonlayer}
	\begin{pgfonlayer}{edgelayer}
		\draw (22.center) to (25.center);
	\end{pgfonlayer}
\end{tikzpicture}
}~~,
& \resizebox{!}{1.2cm}{\input{PRES21A-S.tikz}} &~~=~~ \resizebox{!}{1.2cm}{\begin{tikzpicture}
	\begin{pgfonlayer}{nodelayer}
		\node [style=none] (8) at (0, -3.5) {$n$};
		\node [style=none] (13) at (0, 3.5) {$n$};
		\node [style=none] (22) at (0, -3) {};
		\node [style=none] (25) at (0, 3) {};
		\node [style=none] (29) at (0, -3) {};
	\end{pgfonlayer}
	\begin{pgfonlayer}{edgelayer}
		\draw (22.center) to (25.center);
	\end{pgfonlayer}
\end{tikzpicture}
}~~,
&\resizebox{!}{1.2cm}{\input{PRES22E-S.tikz}} &~~=\qquad~~,
&\resizebox{!}{1.2cm}{\input{PRES23-S.tikz}}&~~=~~\resizebox{!}{1.2cm}{\begin{tikzpicture}
	\begin{pgfonlayer}{nodelayer}
		\node [style=none] (23) at (0, 3) {};
		\node [style=none] (32) at (0, 3.5) {$0$};
		\node [style=none] (34) at (0, -3) {};
		\node [style=none] (37) at (0, -3.5) {$0$};
	\end{pgfonlayer}
	\begin{pgfonlayer}{edgelayer}
		\draw (23.center) to (34.center);
	\end{pgfonlayer}
\end{tikzpicture}
}~~.
\end{align*}
\end{proposition}
\begin{proof}
We compute:
\begin{align*}
\eta^* \otimes \mathsf{id}_{A_n};\nabla_{0,n}^*=&~r_0 \otimes \mathsf{id}_{A_n};s_0 \otimes s_n;r_n \\
=&~(r_0;s_0) \otimes s_n;r_n \\
=&~\mathsf{id}_{I} \otimes s_n;r_n \\
=&~s_n;r_n \\
=&~\mathsf{id}_{A_n},
\end{align*}
\begin{align*}
\Delta_{0,n}^*;\epsilon^* \otimes \mathsf{id}_{A_n}=&~s_n;r_0 \otimes r_n;s_0 \otimes \mathsf{id}_{A_n} \\
=&~s_n;(r_0;s_0) \otimes r_n \\
=&~s_n;\mathsf{id}_I \otimes r_n \\
=&~s_n;r_n \\
=&~\mathsf{id}_{A_n},
\end{align*}
\begin{align*}
\eta^*;\epsilon^*=&~r_0;s_0 \\
=&~\mathsf{id}_I,
\end{align*}
\begin{align*}
\epsilon^*;\eta^*=&~s_0;r_0 \\
=&~\mathsf{id}_{A_0}.\qedhere
\end{align*}
\end{proof}
We can conclude the following from Propositions and Remarks \labelcref{PROP1} to \labelcref{6UNIT}.
\begin{proposition} \label{perm-to-long}
Let $(\mathsf{C},\otimes,I)$ be a $\mathbb{Q}_{\ge 0}$-linear symmetric monoidal category. Let $(A_n) \in \mathsf{C}^{\mathbb{N}}$ and let $(r_n)_{n \ge 0},(s_n)_{n \ge 0}$ be families of maps which make the graded object $(A_n)$ into a permutation splitting. Define the maps $(\nabla_{n,p})_{n,p \ge 0}$, $(\Delta_{n,p})_{n,p \ge 0}$, $\eta$, $\epsilon$ as in \cref{main-theorem}. The graded object $(A_n)$ together with these maps is a binomial bimonoid.
\end{proposition}
\subsection{Biassociativity from the short definition of a binomial bimonoid} \label{FIVE}
Recall that in this section, we work in a fixed $\mathbb{Q}_{\ge 0}$-linear symmetric monoidal category $(\mathsf{C},\otimes,I)$. In this subsection, we prove that given any graded object $(A_n) \in \mathsf{C}^\mathbb{N}$ together with maps $(\nabla_{n,p}\colon A_n \otimes A_p \rightarrow A_{n+p})_{n,p \ge 0}$, $(\Delta_{n,p}\colon A_{n+p} \rightarrow A_n \otimes A_p)_{n,p \ge 0}$, $\eta\colon I \rightarrow A_0$ and $\epsilon\colon A_0 \rightarrow I$ such that all the axioms in \cref{fig:bialg-axiom-short} hold, the multiplication is associative and the comultiplication is coassociative.
We will use this result in the proof of the bijectivity in \cref{main-theorem}.
\begin{lemma} \label{lemma:assoc-bimagma}
Suppose given a graded object $(A_n) \in \mathsf{C}^\mathbb{N}$ together with maps $(\nabla_{n,p}\colon A_n \otimes A_p \rightarrow A_{n+p})_{n,p \ge 0}$, $(\Delta_{n,p}\colon A_{n+p} \rightarrow A_n \otimes A_p)_{n,p \ge 0}$, $\eta\colon I \rightarrow A_0$ and $\epsilon\colon A_0 \rightarrow I$ such that all the axioms in \cref{fig:bialg-axiom-short} hold except the axiom $\Delta_{n,p};\nabla_{n,p}=\binom{n+p}{n}\mathsf{id}_{A_{n+p}}$. Suppose moreover that $\Delta_{n,1}\colon A_{n+1} \rightarrow A_n \otimes A_1$ is a monomorphism for every $n \ge 0$. The multiplication is associative, that is, we have:
\begin{equation} \label{prop:assoc}
\resizebox{!}{1.2cm}{\input{assocA.tikz}}\quad=\quad\resizebox{!}{1.2cm}{\input{assocB.tikz}}\quad
\end{equation}
for all $n,p,q \ge 0$.
\end{lemma}
\begin{proof}
The proof is by induction on $k=n+p+q \ge 0$. We first prove the base case, that is, when $k=0$:
\begin{equation*}
\resizebox{!}{1cm}{\input{ZEROA.tikz}}~=~\resizebox{!}{1cm}{\input{ZEROB.tikz}}~=~\resizebox{!}{1cm}{\input{ZEROC.tikz}}~=~\resizebox{!}{1cm}{\input{ZEROD.tikz}}
\end{equation*}
\begin{equation*}
~~=~\resizebox{!}{1cm}{\input{ZEROE.tikz}}~~=~\resizebox{!}{1cm}{\input{ZEROF.tikz}}~~=~\resizebox{!}{1cm}{\input{ZEROG.tikz}}~.
\end{equation*}
Suppose now that \cref{prop:assoc} holds for some $k \ge 0$. We will prove that it holds for $k+1$. Let $n,p,q \ge 0$ such that $n+p+q=k+1$. We will prove that
\begin{equation} \label{assoc-then-mono}
\resizebox{!}{1.2cm}{\input{autoassoproof5A.tikz}}~=~\resizebox{!}{1.2cm}{\input{autoassoproof7A.tikz}}
\end{equation}
from which \cref{prop:assoc} follows since $\Delta_{n+p+q,1}$ is a monomorphism.
We first use \cref{compatibility-intro-graded-2} twice:
\begin{align} \label{PATHASSOC1}
\resizebox{!}{1.2cm}{\input{autoassoproof5A.tikz}} &\quad=\quad \underset{\substack{a,b,c,d \ge 0 \\ a+b = n+p \\ c+d = q \\ a+c = n+p+q-1 \\ b+d=1}}{\sum} \resizebox{!}{1.2cm}{\input{autoassoproof5B.tikz}} \notag \\
&\quad=\quad\underset{\substack{a,b,c,d \ge 0 \\ a+b = n+p \\ c+d = q \\ a+c = n+p+q-1 \\ b+d=1}}{\sum} 
\underset{\substack{e,f,g,h \ge 0 \\ e+f = n \\ g+h = p \\ e+g = a \\ f+h=b}}{\sum}
\resizebox{!}{1.2cm}{\input{autoassoproof5C.tikz}}\notag
\intertext{and then rewrite the string diagrams in a different form:}
&\quad=\quad\underset{\substack{a,b,c,d \ge 0 \\ a+b = n+p \\ c+d = q \\ a+c = n+p+q-1 \\ b+d=1}}{\sum} 
\underset{\substack{e,f,g,h \ge 0 \\ e+f = n \\ g+h = p \\ e+g = a \\ f+h=b}}{\sum}
\resizebox{!}{1.2cm}{\input{autoassoproof6.tikz}} \quad.
\end{align}
We again use \cref{compatibility-intro-graded-2} twice, as follows:
\begin{align} \label{PATHASSOC2}
\resizebox{!}{1.2cm}{\input{autoassoproof7A.tikz}} &\quad=\quad \underset{\substack{a',b',c',d' \ge 0 \\ a'+b' = n \\ c'+d' = p+q \\ a'+c' = n+p+q-1 \\ b'+d'=1}}{\sum} \resizebox{!}{1.2cm}{\input{autoassoproof7B.tikz}} \notag \\
&\quad=\quad \underset{\substack{a',b',c',d' \ge 0 \\ a'+b' = n \\ c'+d' = p+q \\ a'+c' = n+p+q-1 \\ b'+d'=1}}{\sum} 
\underset{\substack{e',f',g',h' \ge 0 \\ e'+f' = p \\ g'+h' = q \\ e'+g' = c' \\ f'+h'=d'}}{\sum}
\resizebox{!}{1.2cm}{\input{autoassoproof8.tikz}}\quad, \notag \\
\intertext{we rewrite the string diagrams in a different form:}
&\quad=\quad\underset{\substack{a',b',c',d' \ge 0 \\ a'+b' = n \\ c'+d' = p+q \\ a'+c' = n+p+q-1 \\ b'+d'=1}}{\sum} 
\underset{\substack{e',f',g',h' \ge 0 \\ e'+f' = p \\ g'+h' = q \\ e'+g' = c' \\ f'+h'=d'}}{\sum}
\resizebox{!}{1.2cm}{\input{autoassoproof8-.tikz}}\notag \\
\intertext{and use the hypothesis of induction for $k=n+p+q-1$ and $k=1$:}
&\quad=\quad\underset{\substack{a',b',c',d' \ge 0 \\ a'+b' = n \\ c'+d' = p+q \\ a'+c' = n+p+q-1 \\ b'+d'=1}}{\sum} 
\underset{\substack{e',f',g',h' \ge 0 \\ e'+f' = p \\ g'+h' = q \\ e'+g' = c' \\ f'+h'=d'}}{\sum}
\resizebox{!}{1.2cm}{\input{autoassoproof9-.tikz}} \quad.
\end{align}
We now check that the terms in the double sums in \cref{PATHASSOC1} and \cref{PATHASSOC2} are exactly the same. If we choose a term in the double sum in \cref{PATHASSOC1}, it is also a term in the double sum in \cref{PATHASSOC2} by setting:
\begin{equation*}
(a',b',c',d',e',f',g',h') := (e,f,g+c,d+h,g,h,c,d).
\end{equation*}
If we choose a term in the double sum in \cref{PATHASSOC2}, it is also a term in the double sum in \cref{PATHASSOC1} by setting:
\begin{equation*}
(a,b,c,d,e,f,g,h) := (a'+e',b'+f',g',h',a',b',e',f').
\end{equation*} 
\cref{assoc-then-mono} is thus satisfied for $n+p+q=k+1$.
\end{proof}
\begin{proposition}\label{ASSOC}
Suppose given a graded object $(A_n) \in \mathsf{C}^\mathbb{N}$ together with maps $(\nabla_{n,p}\colon A_n \otimes A_p \rightarrow A_{n+p})_{n,p \ge 0}$, $(\Delta_{n,p}\colon A_{n+p} \rightarrow A_n \otimes A_p)_{n,p \ge 0}$, $\eta\colon I \rightarrow A_0$ and $\epsilon\colon A_0 \rightarrow I$ such that all the axioms in \cref{fig:bialg-axiom-short} hold. The multiplication is associative, that is, we have:
\begin{equation*}
\resizebox{!}{1.2cm}{\input{assocA.tikz}}\quad=\quad\resizebox{!}{1.2cm}{\input{assocB.tikz}}\quad
\end{equation*}
for all $n,p,q \ge 0$.
\end{proposition}
\begin{proof}
\cref{special-intro-diag} implies that 
\begin{equation*}
\Delta_{n,1};\bigg(\frac{1}{n+1}\bigg)\nabla_{n,1}=\frac{1}{n+1}\Delta_{n,1};\nabla_{n,1}=\frac{1}{n+1}(n+1)\mathsf{id}_{A_{n+1}}=\mathrm{id}_{A_{n+1}}
\end{equation*}
for every $n \ge 0$. Thus $\Delta_{n,1}$ is a split monomorphism for any $n \ge 0$ and \cref{lemma:assoc-bimagma} applies.
\end{proof}
\begin{remark} \label{ASSOC-REM}
The duals of \cref{lemma:assoc-bimagma} and \cref{ASSOC} are obtained by replacing associativity with coassociativity.
\end{remark}
\begin{remark}
The idea of using monomorphic multiplication and comultiplication maps comes from \cite{Ard}, where a graded monoid with epimorphic multiplication maps and a graded comonoid with monomorphic multiplication maps are called respectively strongly graded algebras and strongly graded coalgebras. This idea is not used further in our paper. Strongly graded algebras and coalgebras are used for other purposes in \cite{Ard}. In particular this reference only treats monoids and comonoids which are supposed respectively associative and coassociative from the start.
\end{remark}
\subsection{Bijectivity of the correspondence} \label{EIGHT}
In this subsection, we verify that the functions from maps making $(A_n)_{n \ge 0}$ into a permutation splitting to maps making $(A_n)_{n \ge 0}$ into a binomial bimonoid and vice versa in \cref{main-theorem} are inverses of each other.

Recall that given $(A_n)_{n \ge 0} \in \mathsf{C}^\mathbb{N}$, \cref{prop-short-to-perm} explains how to go from maps satisfying the axioms in \cref{fig:bialg-axiom-short} to maps making $(A_n)_{n \ge 0}$ into a permutation splitting. Recall also that \cref{perm-to-long} explains how to go from a permutation splitting with underlying graded object $(A_n)_{n \ge 0}$ to a binomial bimonoid with underlying graded object $(A_n)_{n \ge 0}$.
\begin{proposition}\label{BIJEC1}
If we start with operations $(\nabla_{n,p})_{n,p \ge 0}$, $(\Delta_{n,p})_{n,p \ge 0}$, $\eta$, $\epsilon$ on a graded object $(A_n)_{n \ge 0}$ satisfying all the axioms in \cref{fig:bialg-axiom-short}, obtain a permutation splitting by defining the maps $(r_n)_{n \ge 0}$, $(s_n)_{n \ge 0}$ as in \cref{main-theorem}, and finally define the maps $(\nabla_{n,p}^*)_{n,p \ge 0}$, $(\Delta_{n,p}^*)_{n,p \ge 0}$, $\eta^*$, $\epsilon^*$ making $(A_n)_{n \ge 0}$ into a binomial bimonoid as in \cref{main-theorem}, then we obtain that
\begin{equation*}
\nabla_{n,p}^*=\nabla_{n,p},\quad\Delta_{n,p}^*=\Delta_{n,p}
\end{equation*}
for all $n,p \ge 0$ and
\begin{equation*}
\eta^*=\eta,\quad \epsilon^*=\epsilon.
\end{equation*}
\end{proposition} 
\begin{proof}
Let $(A_n)_{n \ge 0}$ with appropriate maps satisfying all the axioms in \cref{fig:bialg-axiom-short}. Recall from \cref{prop-short-to-perm} that a permutation splitting is given by
\begin{equation} \label{BBB1}
\resizebox{!}{1.2cm}{\input{r-n.tikz}}\quad:=\quad \resizebox{!}{1.2cm}{\input{rn-e.tikz}}
~,\qquad
\resizebox{!}{1.2cm}{\input{s-n.tikz}}\quad:=\quad \frac{1}{n!}\resizebox{!}{1.2cm}{\input{sn-e.tikz}}
\end{equation}
for every $n \ge 1$, and
\begin{equation*}
r_0:=\eta,\quad s_0:=\epsilon.
\end{equation*}
Then, by \cref{perm-to-long}, a binomial bimonoid is given by
\begin{equation} \label{BBB2}
\resizebox{!}{1cm}{\input{nablanp-.tikz}} \quad:=\quad \resizebox{!}{1.2cm}{\input{s-np-r-n+p.tikz}}~,
\qquad
\resizebox{!}{1cm}{\input{deltanp-.tikz}} \quad:=\quad 
\scalebox{1.2}{$\binom{n+p}{n}$} 
\resizebox{!}{1.2cm}{\input{r-n+p-s-np.tikz}}
\end{equation}
for all $n,p \ge 0$, and
\begin{equation*}
\eta^*=r_0,\quad \epsilon^*=s_0.
\end{equation*}
We must prove that $\nabla_{n,p}^*=\nabla_{n,p}$ and $\Delta_{n,p}^*=\Delta_{n,p}$ for all $n,p \ge 0$, and that $\eta^*=\eta$, $\epsilon^*=\epsilon$. The last two identities follow directly from what's above, so we will concentrate on the first two identities. We will separate the proof of these two identities in several distinct cases.
\begin{itemize}
\item We first prove that $\nabla_{n,p}^* = \nabla_{n,p}$ for all $n,p \ge 1$, i.e.~by \cref{BBB1,BBB2} that:
\begin{equation*}
\quad \frac{1}{n!p!} \resizebox{!}{2cm}{\input{bijec1.tikz}}\quad= \quad \resizebox{!}{1cm}{\input{nablanp.tikz}}\quad.
\end{equation*}
\textbf{First step:} we obtain that
\begin{equation*}
\frac{1}{n!p!}\resizebox{!}{2cm}{\input{bijec1.tikz}}\quad=\quad\frac{1}{p!} \resizebox{!}{2cm}{\input{bijecA.tikz}}
\end{equation*}
by using $n-1$ times the identity $\Delta_{k,1};\nabla_{k,1}=(k+1)\mathsf{id}_{A_{k+1}}$.

\textbf{Second step:} we obtain from associativity (proven in \cref{ASSOC}) that
\begin{equation*}
\quad=\quad\frac{1}{p!} \resizebox{!}{2cm}{\input{bijecB.tikz}}\quad.
\end{equation*}
\textbf{Third step:} we obtain the final result by using $p-1$ times the identity $\Delta_{k,1};\nabla_{k,1}=(k+1)\mathsf{id}_{A_{k+1}}$. 
\item We now prove that $\nabla_{n,0}^*=\nabla_{n,0}$ for every $n \ge 1$, i.e.~that
\begin{equation*}
\frac{1}{n!}\resizebox{!}{2cm}{\input{bijecB-.tikz}}\quad=\quad\resizebox{!}{1cm}{\input{nablan0-.tikz}}~~.
\end{equation*}
The proof is as follows:
\begin{equation*}
\frac{1}{n!}\resizebox{!}{2.5cm}{\input{bijecB-.tikz}}\quad=\quad\frac{n!}{n!}\resizebox{!}{1cm}{\input{nablanp-INTER1.tikz}}\quad=\quad\resizebox{!}{1cm}{\input{nablanp-INTER2.tikz}}\quad=\quad\resizebox{!}{1cm}{\input{nablanp-INTER3.tikz}}\quad=\quad\resizebox{!}{1cm}{\input{nablan0-.tikz}}~.
\end{equation*}
\item To prove that $\nabla_{0,p}^*=\nabla_{0,p}$ for every $p \ge 1$, it suffices to apply a left-right symmetry to the proof of $\nabla_{n,0}^*=\nabla_{n,0}$ for every $n \ge 0$.
\item We prove that $\nabla_{0,0}^*=\nabla_{0,0}$ i.e.~that
\begin{equation*}
\resizebox{!}{1cm}{\input{nabla000.tikz}}\quad=\quad\resizebox{!}{1cm}{\input{nabla00.tikz}}~.
\end{equation*}
The proof is as follows:
\begin{equation*}
\resizebox{!}{1cm}{\input{nabla000.tikz}}\quad=\quad\resizebox{!}{1cm}{\input{nabla000A.tikz}}\quad=\quad\resizebox{!}{1cm}{\input{nabla000B.tikz}}\quad=\quad\resizebox{!}{1cm}{\input{nabla00.tikz}}~.
\end{equation*}
\item We now prove $\Delta_{n,p}^* = \Delta_{n,p}$ for all $n,p \ge 1$ i.e.~by \cref{BBB1,BBB2} that:
\begin{equation*}
\frac{1}{(n+p)!} \binom{n+p}{n} \resizebox{!}{2cm}{\input{bijec2.tikz}} \quad = \quad \resizebox{!}{1cm}{\input{deltanp.tikz}}\quad.
\end{equation*}
It suffices to observe that $\frac{1}{(n+p!)}\binom{n+p}{n} = \frac{1}{n!p!}$ and then to  apply an up-down symmetry to the proof of $\nabla_{n,p}^* = \nabla_{n,p}$ for all $n,p \ge 1$.
\item We want to prove that $\Delta_{n,0}^* = \Delta_{n,0}$ for every $n \ge 1$ i.e.~that
\begin{equation*}
\frac{1}{n!}\resizebox{!}{2cm}{\input{bijecB-0.tikz}}\quad=\quad\resizebox{!}{1cm}{\input{deltan0-.tikz}}~.
\end{equation*}
It suffices to apply an up-down symmetry to the proof of $\nabla_{n,0}^*=\nabla_{n,0}$ for every $n \ge 0$.
\item We want to prove that $\Delta_{0,p}^* = \Delta_{0,p}$ for every $p \ge 1$ i.e.~that
\begin{equation*}
\frac{1}{p!}\resizebox{!}{2cm}{\input{bijecB-0p.tikz}}\quad=\quad\resizebox{!}{1cm}{\input{deltap0-.tikz}}~.
\end{equation*}
It suffices to apply an up-down and a left-right symmetry to the proof of $\nabla_{n,0}^*=\nabla_{n,0}$ for every $n \ge 1$.
\item We want to prove that $\Delta_{0,0}^* = \Delta_{0,0}$ i.e.~that
\begin{equation*}
\resizebox{!}{1cm}{\input{nabla000-.tikz}}\quad=\quad\resizebox{!}{1cm}{\input{nabla00-.tikz}}~.
\end{equation*}
It suffices to apply an up-down symmetry to the proof of $\nabla_{0,0}^*=\nabla_{0,0}$.\qedhere
\end{itemize}
\end{proof}
We obtain \cref{main-theorem-2} as a consequence of \cref{BIJEC1}: starting from a graded object $(A_n)$ in a $\mathbb{Q}_{\ge 0}$-linear symmetric monoidal category, with operations $(\nabla_{n,p})_{n,p \ge 0}$, $(\Delta_{n,p})_{n,p \ge 0}$, $\eta$, $\epsilon$ satisfying the axioms in \cref{fig:bialg-axiom-short}, we obtain 
a binomial bimonoids with underlying graded object $(A_n)$ and operations $(\nabla_{n,p}^*)_{n,p \ge 0}$, $(\Delta_{n,p}^*)_{n,p \ge 0}$, $\eta^*$, $\epsilon^*$. But \cref{BIJEC1} implies that the operations with a superscipt $*$ are in fact equal to the operations without a superscript. It follows that we already had a binomial bimonoid from the start. Conversely, if we start with a binomial bimonoid, then all the axioms in \cref{fig:bialg-axiom-short} are either axioms in the definition of a binomial bimonoid, or the right unitality or counitality axiom, which follows from either left unitality and commutativity, or left counitality and cocommutativity. All the axioms in \cref{fig:bialg-axiom-short} are thus satisfied.
\begin{proposition} \label{BIJEC2}
If we start with maps $(r_n)_{n \ge 0}$, $(s_n)_{n \ge 0}$ making an $\mathbb{N}$-graded object $(A_n)_{n \ge 0}$ into a permutation splitting, then obtain a binomial bimonoid by defining the maps $(\nabla_{n,p})_{n,p \ge 0}$, $(\Delta_{n,p})_{n,p \ge 0}$, $\eta$, $\epsilon$ as in \cref{main-theorem}, and finally define the maps $(r_n^*)_{n \ge 0}$, $(s_n^*)_{n \ge 0}$ to obtain a permutation splitting from this binomial bimonoid as in \cref{main-theorem}, we obtain that
\begin{equation*}
r_n^*=r_n,\quad s_n^*=s_n
\end{equation*}
for every $n \ge 0$.
\end{proposition} 
\begin{proof}
Let $(A_n,r_n,s_n)_{n \ge 0}$ be a permutation splitting. Recall from \cref{perm-to-long} that a binomial bimonoid is given by
\begin{equation} \label{AAA1}
\resizebox{!}{1cm}{\input{nablanp.tikz}} \quad:=\quad \resizebox{!}{1.2cm}{\input{s-np-r-n+p.tikz}}~,\qquad
\resizebox{!}{1cm}{\input{deltanp.tikz}} \quad:=\quad 
\scalebox{1.2}{$\binom{n+p}{n}$} 
\resizebox{!}{1.2cm}{\input{r-n+p-s-np.tikz}}
\end{equation}
for all $n,p \ge 0$, and
\begin{equation*}
\eta:=r_0,\quad \epsilon:=s_0.
\end{equation*}
Then, by \cref{prop-short-to-perm}, a permutation splitting is given by
\begin{equation} \label{AAA2}
\resizebox{!}{1.2cm}{\input{r-nB.tikz}}\quad:=\quad \resizebox{!}{1.2cm}{\input{rn-e.tikz}}~,\qquad
\resizebox{!}{1.2cm}{\input{s-nB.tikz}}\quad:=\quad \frac{1}{n!}\resizebox{!}{1.2cm}{\input{sn-e.tikz}}
\end{equation}
for every $n \ge 1$ and
\begin{equation*}
r_0:=\eta,\quad s_0:=\epsilon.
\end{equation*}
We will prove that $r_n^*=r_n$ and $s_n^*=s_n$ for every $n \ge 0$. It is immediate from what's above that $r_0^*=r_0$ and $s_0^*=s_0$. We will thus focus on the cases where $n \ge 1$. We proceed with in-line equations.
\begin{itemize}
\item We start by proving that
\begin{equation*}
r_{n}^* = r_{n},
\end{equation*}
for every $n \ge 1$ i.e.~by \cref{AAA1,AAA2} that 
\begin{equation*}
s_{1}^{\otimes n};((r_2;s_2) \otimes \mathsf{id}_{A_1^{\otimes(n-2)}});\cdots;((r_{n-1};s_{n-1}) \otimes \mathsf{id}_{A_1});r_{n}=r_n.
\end{equation*}
First, $s_{1} = \mathsf{id}_{A_1}$ and thus
\begin{equation*}
r_{n}^* = ((r_2;s_2) \otimes \mathsf{id}_{A_1^{\otimes (n-2)}});\cdots;((r_{n-1};s_{n-1}) \otimes \mathsf{id}_{A_1});r_{n}.
\end{equation*}
Then observe that 
\begin{equation*}
((r_2;s_2) \otimes \mathsf{id}_{A_1^{\otimes (n-2)}});\cdots;((r_{n-1};s_{n-1}) \otimes \mathsf{id}_{A_1})
\end{equation*}
is a permutation average which is annihilated by the following $r_{n}$, thus 
\begin{equation*}
r_n^* = r_n.
\end{equation*}
\item We must also show that 
\begin{equation*}
s_{n}^* = s_{n},
\end{equation*}
for every $n \ge 1$ i.e.~by \cref{AAA1,AAA2} that 
\begin{align*}
\frac{1}{n!}\binom{n}{n-1}\binom{n-1}{n-2}\cdots\binom{2}{1} s_{n};((r_{n-1};s_{n-1}) \otimes \mathsf{id}_{A_1})&; \\
((r_{n-2};s_{n-2}) \otimes \mathsf{id}_{A_1^{\otimes 2}});\cdots;((r_2;s_2) \otimes \mathsf{id}_{A_1^{\otimes (n-2)}});r_{1}^{\otimes n}&=s_{n}.
\end{align*}
First, we have
\begin{equation*}
\binom{n}{n-1}\binom{n-1}{n-2}\cdots\binom{2}{1} = n(n-1)\cdots2=n!
\end{equation*}
and thus
\begin{equation*}
s_n^*=s_{n};((r_{n-1};s_{n-1}) \otimes \mathsf{id}_{A_1});
((r_{n-2};s_{n-2}) \otimes \mathsf{id}_{A_1^{\otimes 2}});\cdots;((r_2;s_2) \otimes \mathsf{id}_{A_1^{\otimes (n-2)}});r_{1}^{\otimes n}.
\end{equation*}
Then observe that 
\begin{equation*}
r_1 = \mathsf{id}_{A_1}
\end{equation*}
and thus that 
\begin{equation*}
s_n^*=s_{n};((r_{n-1};s_{n-1}) \otimes \mathsf{id}_{A_1});((r_{n-2};s_{n-2}) \otimes \mathsf{id}_{A_1^{\otimes 2}});\cdots;((r_2;s_2) \otimes \mathsf{id}_{A_1^{\otimes (n-2)}}).
\end{equation*}
Finally we use that 
\begin{equation*}
((r_{n-1};s_{n-1}) \otimes \mathsf{id}_{A_1});((r_{n-2};s_{n-2}) \otimes \mathsf{id}_{A_1^{\otimes 2}});\cdots;((r_2;s_2) \otimes \mathsf{id}_{A_1^{\otimes (n-2)}})
\end{equation*} 
is a permutation average and thus 
\begin{equation*}
s_n^*=s_n.\qedhere
\end{equation*}
\end{itemize}
\end{proof}
We obtain \cref{main-theorem} as a consequence of \cref{prop-short-to-perm,perm-to-long,BIJEC1,BIJEC2}.
\section{Being a binomial bimonoid: property or structure?} \label{SEC:FIVE}
In this section, we will be interested in the following question: is being a binomial bimonoid in an additive symmetric monoidal category a property or a structure? To make sense of this question, we first need to define morphisms between binomial bimonoids. 
\subsection{Categories of binomial bimonoids and of permutation splittings} \label{SUBSEC:FIVE-ONE}
In this subsection, we define several categories of binomial bimonoids and permutation splittings, and functors between them.
\begin{definition}
Let $(\mathsf{C},\otimes,I)$ be an additive symmetric monoidal category. Let $(A_n)$ and $(B_n)$ be binomial bimonoids. A \emph{morphism of binomial bimonoids} from $(A_n)$ to $(B_n)$ is any graded morphism\footnote{If $\mathsf{C}$ is a category, we call \emph{graded morphism} any morphism in $\mathsf{C}^\mathbb{N}$.} $(f_n)$ from $(A_n)$ to $(B_n)$ such that the following identities are satisfied:
\begin{enumerate}
\item $\eta;f_0=\eta\colon I \rightarrow B_0$,
\item $f_0;\epsilon=\epsilon\colon A_0 \rightarrow I$,
\item $f_n \otimes f_p;\nabla_{n,p}=\nabla_{n,p};f_{n+p}\colon A_n \otimes A_p \rightarrow B_{n+p}$ for all $n,p \ge 0$,
\item $\Delta_{n,p};f_n \otimes f_p=f_{n+p};\Delta_{n,p}\colon A_{n+p} \rightarrow B_n \otimes B_p$ for all $n,p \ge 0$.
\end{enumerate}
\end{definition}
\begin{definition}
Let $(\mathsf{C},\otimes,I)$ be a $\mathbb{Q}_{\ge 0}$-linear symmetric monoidal category. Let $(A_n)$ and $(B_n)$ be two permutation splittings. A \emph{morphism of permutation splittings} from $(A_n)$ to $(B_n)$ is any graded morphism $(f_n)$ from $(A_n)$ to $(B_n)$ such that the following identities are satisfied:
\begin{enumerate}
\item $f_1^{\otimes n};r_n=r_n;f_n\colon A_1^{\otimes n} \rightarrow B_n$ for every $n \ge 0$,
\item $s_n;f_1^{\otimes n}=f_n;s_n\colon A_n \rightarrow B_1^{\otimes n}$ for every $n \ge 0$.
\end{enumerate}
\end{definition}
We introduce four categories that we will prove are isomorphic.
\begin{definition}
Let $(\mathsf{C},\otimes,I)$ be a $\mathbb{Q}_{\ge 0}$-linear symmetric monoidal category.
\begin{enumerate}
\item The category $\mathsf{BinBimon}(\mathsf{C})$ has for objects the binomial bimonoids in $(\mathsf{C},\otimes,I)$ and for morphisms the morphisms of binomial bimonoids between them. The composition and the identities are the same as in $\mathsf{C}^\mathbb{N}$.
\item The category $\mathsf{BinBimon}_1(\mathsf{C})$ has for objects the binomial bimonoids in $(\mathsf{C},\otimes,I)$ and a morphism from $(A_n)$ to $(B_n)$ is any morphism from $A_1$ to $B_1$. The identity on $(A_n)$ is $\mathsf{id}_{A_1}$.
\item The category $\mathsf{PermSplit}(\mathsf{C})$ has for objects the permutation splittings in $(\mathsf{C},\otimes,I)$ and for morphisms the morphisms of permutation splittings between them. The composition and the identities are the same as in $\mathsf{C}^\mathbb{N}$.
\item The category $\mathsf{PermSplit}_1(\mathsf{C})$ has for objects the permutation splittings in $(\mathsf{C},\otimes,I)$ and a morphism from $(A_n)$ to $(B_n)$ is any morphism from $A_1$ to $B_1$. The identity on $(A_n)$ is $\mathsf{id}_{A_1}$.
\end{enumerate}
\end{definition}
\begin{remark}
A morphism $(f_n)$ in $\mathsf{BinBimon}(\mathsf{C})$ or $\mathsf{PermSplit}(\mathsf{C})$ is an isomorphism iff $f_n$ is an isomorphism for every $n \ge 0$.
\end{remark}
In the rest of the section, we suppose given a $\mathbb{Q}_{\ge 0}$-linear symmetric monoidal category $(\mathsf{C},\otimes,I)$. We will introduce some functors:
\begin{equation*}
% https://tikzcd.yichuanshen.de/#N4Igdg9gJgpgziAXAbVABwnAlgFyxMJZAJgBoAGAXVJADcBDAGwFcYkQAFGAJwFsBlNI1wAKADpje9HAAsAxk2ABhAL4BKECtLpMufIRQAWCtTpNW7Ln0HCcAfWABGFeMnT5i1Rq07seAkTkJjQMLGyIIABCWGAA4tzRvASuUrIKjMrqmtogGH76RABswWZh7NFxCVhJYA7OKe7pmd6mMFAA5vBEoABm3BC8SGQgOBBIQSBwMlg9OEgAtI4h5uEgAGJ1Ktm9-YOIE6NIS5PTswvHoRYRG8DEWz4gfQNHNIeIAMw0UzNziIvLZWuDne9xyTz2nxGY0Qx2+Zz+FxW7FiwNBO2eMNe0ImcN+-1KVxAKNuaMeu3GWKGX1OeMRgKJm00lBUQA
\begin{tikzcd}
\mathsf{BinBimon}(\mathsf{C}) \arrow[rr, "F_{1}", shift left] &  & \mathsf{PermSplit}(\mathsf{C}) \arrow[rr, "F_{2}", shift left] \arrow[ll, "G_{1}", shift left] &  & \mathsf{PermSplit}_{1}(\mathsf{C}) \arrow[rr, "F_{3}", shift left] \arrow[ll, "G_{2}", shift left] &  & \mathsf{BinBimon}_{1}(\mathsf{C}) \arrow[ll, "G_{3}", shift left]
\end{tikzcd}~.
\end{equation*}
We will show that these are isomorphisms with $F_i^{-1}=G_i$ for any $i \in [1,3]$. 

On objects, these functors are defined as expected from \cref{main-theorem}. Let $(A_n)$ be a binomial bimonoid. We define $F_1((A_n))=G_3((A_n))$ as the permutation splitting obtained by keeping the same underlying graded object and equipping it with the operations defined in \cref{def-fun:-1,def-fun:2}. Let $(A_n)_{n \ge 0}$ be a permutation splitting. We define $G_1((A_n))=F_3((A_n))$ as the binomial bimonoid obtained by keeping the same underlying graded object and equipping it with the operations defined in \cref{perm-from-bimon-1,perm-from-bimon-2}. Finally, $F_2$ and $G_2$ act as the identity on objects. The fact that $G_i=F_i^{-1}$ on objects is either trivial or a consequence of \cref{main-theorem}.

From \cref{extension}, we have morphisms $\nabla_{1,\dots,1}~(n \text{ times }1)\colon A_1^{\otimes n} \rightarrow A_n$ defined for every $n \ge 2$. We extend this definition to any $n \ge 0$ by defining $\nabla_1=\mathsf{id}_{A_1}\colon A_1 \rightarrow A_1$ and $\nabla=\eta\colon I \rightarrow A_0$. In the same way, we extend the definition of $\Delta_{1,\dots,1}~(n \text{ times }1)\colon A_n  \rightarrow A_1^{\otimes n}$ from $n \ge 2$ to any $n \ge 0$ by defining $\Delta_1=\mathsf{id}_{A_1}\colon A_1 \rightarrow A_1$ and $\Delta=\epsilon\colon A_0 \rightarrow I$.
\begin{lemma}
Let $(A_n)$ and $(B_n)$ be binomial bimonoids and let $(f_n)$ be a morphism of $\mathbb{N}$-graded bimonoids from $(A_n)$ to $(B_n)$. We have $f_1^{\otimes n};\nabla_{1,\dots,1}=\nabla_{1,\dots,1};f_n$ ($n$ times $1$ in $\nabla_{1,\dots,1}$).
\end{lemma}
\begin{proof}
If $n=0$, we must prove that $\mathsf{id}_I;\eta=\eta;f_1$, that is, $\eta;f_1=\eta$, but this is exactly the unitality equation for an $\mathbb{N}$-graded monoid. If $n=1$, we have $f_1;\nabla_1=f_1;\mathsf{id}_{B_1}=f_1=\mathsf{id}_{A_1};f_1=\nabla_1;f_1$. Suppose the result is true for some $n \ge 1$. By \cref{higher-order-mult-ind}, we then have
\begin{align*}
f_1^{\otimes (n+1)};\nabla_{1,\dots,1}~(n+1 \text{ times }1)=&~f_1^{\otimes (n+1)};\nabla_{1,\dots,1}~(n \text{ times }1)  \otimes \mathsf{id}_{A_1};\nabla_{n,1} \\
=&~(f_1^{\otimes n};\nabla_{1,\dots,1}~(n \text{ times }1)) \otimes f_1;\nabla_{n,1} \\
=&~(\nabla_{1,\dots,1}~(n \text{ times }1);f_n) \otimes f_1;\nabla_{n,1} \\
=&~\nabla_{1,\dots,1}~(n \text{ times }1) \otimes \mathsf{id}_{A_1};f_n \otimes f_1;\nabla_{n,1} \\
=&~\nabla_{1,\dots,1}~(n \text{ times }1) \otimes \mathsf{id}_{A_1};\nabla_{n,1};f_{n+1} \\
=&~\nabla_{1,\dots,1}~(n+1 \text{ times }1);f_{n+1}.
\end{align*}
We conclude that the result is true for every $n \ge 0$.
\end{proof}
\begin{remark}
The dual result gives $\Delta_{1,\dots,1};f_1^{\otimes n}=f_n;\Delta_{1,\dots,1}$ ($n$ times $1$ in $\Delta_{1,\dots,1}$) for every $n \ge 0$.
\end{remark}
\begin{corollary} \label{lem-fun-mor-1}
Let $(A_n)$ and $(B_n)$ be binomial bimonoids and let $(f_n)$ be a morphism of binomial bimonoids from $(A_n)$ to $(B_n)$. Then $(f_n)$ is a morphism of permutation splittings from $F_1((A_n))$ to $F_1((B_n))$.
\end{corollary}
\begin{proof}
We compute:
\begin{equation*}
f_1^{\otimes n};r_n=f_1^{\otimes n};\nabla_{1,\dots,1}
=\nabla_{1,\dots,1};f_n
=r_n;f_n
\end{equation*}
and
\begin{equation*}
s_n;f_1^{\otimes n}=\Delta_{1,\dots,1};f_1^{\otimes n}
=f_n;\Delta_{1,\dots,1}
=f_n;s_n.\qedhere
\end{equation*}
\end{proof}
\begin{lemma} \label{lem-fun-mor-2}
Let $(A_n)$ and $(B_n)$ be permutation splittings and let $(f_n)$ be a morphism of permutation splittings from $(A_n)$ to $(B_n)$. Then $(f_n)$ is a morphism of binomial bimonoids from $G_1((A_n))$ to $G_1((B_n))$.
\end{lemma}
\begin{proof}
We compute:
\begin{equation*}
\eta;f_0=r_0;f_0
=f_1^{\otimes 0};r_0
=\mathsf{id}_I;r_0
=r_0
=\eta,
\end{equation*}
\begin{equation*}
f_0;\epsilon=f_0;s_0
=s_0;f_1^{\otimes 0}
=s_0;\mathsf{id}_I
=s_0
=\epsilon,
\end{equation*}
\begin{align*}
f_n \otimes f_p;\nabla_{n,p}=&~f_n \otimes f_p;s_n \otimes s_p;r_{n+p} \\
=&~(f_n;s_n) \otimes (f_p;s_p);r_{n+p} \\
=&~(s_n;f_1^{\otimes n}) \otimes (s_p;f_1^{\otimes p});r_{n+p} \\
=&~s_n \otimes s_p;f_1^{\otimes (n+p)};r_{n+p} \\
=&~s_n \otimes s_p;r_{n+p};f_{n+p} \\
=&~\nabla_{n,p};f_{n+p},
\end{align*}
\begin{align*}
\Delta_{n,p};f_n \otimes f_p=&~\binom{n+p}{n}s_{n+p};r_n \otimes r_p;f_n \otimes f_p \\
=&~\binom{n+p}{n}s_{n+p};(r_n;f_n) \otimes (r_p;f_p) \\
=&~\binom{n+p}{n}s_{n+p};(f_1^{\otimes n};r_n) \otimes (f_1^{\otimes p};r_p) \\
=&~\binom{n+p}{n}s_{n+p};f_1^{\otimes (n+p)};r_n \otimes r_p \\
=&~\binom{n+p}{n};f_{n+p};s_{n+p};r_n \otimes r_p \\
=&~f_{n+p};\Delta_{n,p}.\qedhere
\end{align*}
\end{proof}
From \cref{lem-fun-mor-1,lem-fun-mor-2}, we can define the functors $F_1$ and $G_1$ to act as the identity on morphisms. If $(A_n)$ and $(B_n)$ are permutation splittings, and $(f_n)$ is a morphism of permutation splittings from $(A_n)$ to $(B_n)$, then we define $F_2((f_n))=f_1$. The functors $F_3$ and $G_3$ act as the identity on morphisms. We still have to define the functor $G_2$ on morphisms.
\begin{lemma} \label{lem-fun-mor-3}
Let $(A_n)$ and $(B_n)$ be permutation splittings and let $f_1\colon A_1 \rightarrow B_1$ be any morphism. Define $f_n\colon A_n \rightarrow B_n$ for every $n \in \mathbb{N} \backslash \{1\}$ by the formula
\begin{equation} \label{ext-f1}
f_n=s_n;f_1^{\otimes n};r_n\colon A_n \rightarrow B_n.
\end{equation}
Then, \cref{ext-f1} holds for any $n \in \mathbb{N}$. Moreover, $(f_n)$ is a morphism of permutation splittings from $(A_n)$ to $(B_n)$. 
\end{lemma}
\begin{proof}
Suppose that $f_n$ is defined for every $n \in \mathbb{N} \backslash \{1\}$ by \cref{ext-f1}. Then \cref{ext-f1} also holds for $n=1$ since $s_1=r_1=\mathsf{id}_{A_1}$. We now prove that $(f_n)$ is a morphism of permutation splittings. For every $n \in \mathbb{N}$, we have
\begin{align*}
f_1^{\otimes n};r_n=&~\frac{1}{n!}n!(f_1^{\otimes n};r_n) \\
=&~\frac{1}{n!}\underset{\sigma \in S_n}{\sum}(f_1^{\otimes n};r_n) \\
=&~\frac{1}{n!}\underset{\sigma \in S_n}{\sum}(f_1^{\otimes n};\sigma;r_n) \\
=&~\frac{1}{n!}\underset{\sigma \in S_n}{\sum}(\sigma;f_1^{\otimes n};r_n) \\
=&~(\frac{1}{n!}\underset{\sigma \in S_n}{\sum}\sigma);f_1^{\otimes n};r_n \\
=&~(r_n;s_n);f_1^{\otimes n};r_n \\
=&~r_n;(s_n;f_1^{\otimes n};r_n) \\
=&~r_n;f_n
\end{align*}
and
\begin{align*}
s_n;f_1^{\otimes n}=&~\frac{1}{n!}n!(s_n;f_1^{\otimes n}) \\
=&~\frac{1}{n!}\underset{\sigma \in S_n}{\sum}(s_n;f_1^{\otimes n}) \\
=&~\frac{1}{n!}\underset{\sigma \in S_n}{\sum}(s_n;\sigma;f_1^{\otimes n}) \\
=&~\frac{1}{n!}\underset{\sigma \in S_n}{\sum}(s_n;f_1^{\otimes n};\sigma) \\
=&~s_n;f_1^{\otimes n};(\frac{1}{n!}\underset{\sigma \in S_n}{\sum}\sigma) \\
=&~s_n;f_1^{\otimes n};(r_n;s_n) \\
=&~(s_n;f_1^{\otimes n};r_n);s_n \\
=&~f_n;s_n.\qedhere
\end{align*}
\end{proof}
From \cref{lem-fun-mor-3}, if $(A_n)$ and $(B_n)$ are permutation splittings and $f_1\colon A_1 \rightarrow B_1$ is any morphism, we can define $G_2(f_1)=(f_n)$ where $f_n$ is given by \cref{ext-f1} for any $n \in \mathbb{N}$.

At this point, we have completely specified the actions of the functors $F_i$ and $G_i$ both on objects and on morphisms. It is fairly immediate from the definitions that the $F_i$ and $G_i$ preserve identities and composition. The only nontrivial case is the preservation of composition by $G_2$ which follows from a straightforward calculation.\footnote{Let $(A_n)$, $(B_n)$, $(C_n)$ be two permutation splittings and let $f_1\colon A_1 \rightarrow B_1$, $g_1\colon B_1 \rightarrow C_1$ be any morphisms. We compute
\begin{align*}
G_2(f_1);G_2(g_1)=&~s_n;f_1^{\otimes n};r_n;s_n;g_1^{\otimes n};r_n \\
=&~s_n;f_1^{\otimes n};\frac{1}{n!}\underset{\sigma \in S_n}{\sum}\sigma;g_1^{\otimes n};r_n \\
=&~\frac{1}{n!}\underset{\sigma \in S_n}{\sum}s_n;f_1^{\otimes n};\sigma;g_1^{\otimes n};r_n \\
=&~\frac{1}{n!}\underset{\sigma \in S_n}{\sum}s_n;\sigma;f_1^{\otimes n};g_1^{\otimes n};r_n \\
=&~\frac{1}{n!}\underset{\sigma \in S_n}{\sum}s_n;f_1^{\otimes n};g_1^{\otimes n};r_n \\
=&~\frac{1}{n!}(n!)s_n;f_1^{\otimes n};g_1^{\otimes n};r_n \\
=&~s_n;(f_1;g_1)^{\otimes n};r_n \\
=&~G_2(f_1;g_1).
\end{align*}}
 Hence, for each $i \in [1,3]$, the assignments $F_i$ and $G_i$ indeed define functors. We also know that the functors $F_i$ and $G_i$ realize bijections on objects. To prove that these functors are isomorphisms with $F_i^{-1}=G_i$, it only remains to prove that they realize bijections on morphisms. The functors $F_i$ and $G_i$ act as the identity on morphisms for $i \in \{1,3\}$. Moreover, if we start with a morphism $f_1$ in $\mathsf{PermSplit}_1(\mathsf{C})$, then it is immediate from \cref{lem-fun-mor-3}, that $F_2(G_2(f_1))=f_1$.
\begin{lemma}
Let $(A_n)$ and $(B_n)$ be permutation splittings and let $(f_n)$ be a morphism of permutation splittings from $(A_n)$ to $(B_n)$. Then \cref{ext-f1} holds for every $n \in \mathbb{N}$.
\end{lemma}
\begin{proof}
For every $n \in \mathbb{N}$, we have
\begin{align*}
s_n;f_1^{\otimes n};r_n=&~f_n;s_n;r_n \\
=&~f_n;\mathsf{id}_{A_n} \\
=&~f_n.\qedhere
\end{align*}
\end{proof}
We thus have $G_2(F_2((f_n)))=(f_n)$. We have completely proved the following.
\begin{theorem}
The functors $F_i$ and $G_i$ for $i \in [1,3]$ are isomorphisms with $F_i^{-1}=G_i$ for any $i \in [1,3]$.
\end{theorem}
\subsection{Property in a $\mathbb{Q}_{\ge 0}$-linear symmetric monoidal category} \label{SUBSEC:FIVE-TWO}
As promised in the introduction, this subsection shows that being a binomial bimonoid in a $\mathbb{Q}_{\ge 0}$-linear symmetric monoidal category is a property.
\begin{proposition} \label{prop-which-gives-the-iso}
Let $(A_n)$ and $(B_n)$ be binomial bimonoids in a $\mathbb{Q}_{\ge 0}$-linear symmetric monoidal category $(\mathsf{C},\otimes,I)$. Suppose that $A_1$ is isomorphic to $B_1$ in $\mathsf{C}$. Then $(A_n)$ and $(B_n)$ are isomorphic as binomial bimonoids. 
\end{proposition}
\begin{proof}
Let $f_1\colon A_1 \simeq B_1$ be an isomorphism in $\mathsf{C}$. We obtain an isomorphism $f_1\colon (A_n) \simeq (B_n)$ in $\mathsf{BinBimon}_1(\mathsf{C})$. We finally obtain an isomorphism $G_2(f_1)\colon(A_n) \simeq (B_n)$ in $\mathsf{BinBimon}(\mathsf{C})$.
\end{proof}
\begin{corollary} \label{cor-property}
Let $(A_n)$ and $(B_n)$ be binomial bimonoids in a $\mathbb{Q}_{\ge 0}$-linear symmetric monoidal category $(\mathsf{C},\otimes,I)$. Suppose the graded object $(A_n)$ is isomorphic to $(B_n)$ in $\mathsf{C}^\mathbb{N}$. Then $(A_n)$ and $(B_n)$ are isomorphic as binomial bimonoids. 
\end{corollary}
\begin{proof}
If $(f_n)\colon (A_n) \simeq (B_n)$ is an isomorphism in $\mathsf{C}^\mathbb{N}$, we obtain an isomorphism $f_1\colon B_1 \simeq B_1$ in $\mathsf{C}$. We then apply \cref{prop-which-gives-the-iso}.
\end{proof}
\subsection{Structure in an additive symmetric monoidal category} \label{SUBSEC:FIVE-THREE}
We will show in this subsection that \cref{cor-property} does not extend to additive symmetric monoidal categories. Being a binomial bimonoid in an arbitrary additive symmetric monoidal category is thus a structure rather than a property. Let $k$ be a field of characteristic $p>0$.

We will exhibit a binomial bimonoid $(k_n\langle x \rangle)$ in $\mathsf{Vec}_k$ such that $(k_n\langle x \rangle)$ is isomorphic to $(k_n[x])$ in $\mathsf{Vec}_k^\mathbb{N}$ but $(k_n\langle x \rangle)$ is not isomorphic to $(k_n[x])$ as a binomial bimonoid. We define $x^{[0]}=1$ and for every $n \in \mathbb{N} \backslash \{1\}$, we suppose given a symbol $x^{[n]} \neq 0$. We define $k_n\langle x\rangle$ as the $k$-vector space of all the expressions of the form $a_nx^{[n]}$ where $a_n \in k$. An element $a_nx^{[n]} \in k_n\langle x^n\rangle$ is called an \emph{homogeneous divided power polynomial of degree} $n$.

We will make $(k_n\langle x\rangle)$ into a binomial bimonoid by equipping it with the following $k$-linear maps:
\begin{equation*}
\nabla_{n,p}\colon
\begin{aligned}
k_n\langle x\rangle \otimes k_p\langle x \rangle \longrightarrow &~k_{n+p}\langle x\rangle \\
x^{[n]} \otimes x^{[p]} \longmapsto &~\binom{n+p}{n}x^{[n+p]}, \\
\end{aligned}
\qquad\qquad
\eta\colon
\begin{aligned}
k \longrightarrow &~k_0\langle x\rangle \\
1 \longmapsto &~x^{[0]}, \\
\end{aligned}
\end{equation*}
\begin{equation*}
\Delta_{n,p}\colon
\begin{aligned}
k_{n+p}\langle x\rangle \longrightarrow &~k_n\langle x\rangle \otimes k_p\langle x\rangle  \\
x^{[n+p]} \longmapsto &~x^{[n]} \otimes x^{[p]}, \\
\end{aligned}
\qquad\qquad
\epsilon\colon
\begin{aligned}
k_0\langle x\rangle \longrightarrow &~k \\
x^{[0]} \longmapsto &~1. \\
\end{aligned}
\end{equation*}
All the axioms in \cref{fig:bialg-axiom-nonred} must be satisfied. There are only two axioms which are not immediate. The first one is associativity, which follows easily from the identity
\begin{equation*}
\binom{n+p+q}{n}\binom{p+q}{p}=\frac{(n+p+q)!}{n!p!q!}=\binom{n+p}{n}\binom{n+p+q}{n+p}.
\end{equation*}
The second and most difficult axiom is \cref{compatibility-intro-graded-2}. The LHS in \cref{compatibility-intro-graded-2} applied to $x^{[n]} \otimes x^{[p]}$ gives 
\begin{equation} \label{eq-divided-1}
\binom{n+p}{n}x^{[q]} \otimes x^{[r]}
\end{equation}
whereas the RHS applied to $x^{[n]} \otimes x^{[p]}$ gives
\begin{equation} \label{eq-divided-2}
\underset{\substack{a,b,c,d \ge 0 \\ a+b=n \\ c+d=p \\ a+c=q \\ b+d=r}}{\sum}\binom{q}{a}\binom{r}{b}x^{[q]} \otimes x^{[r]}.
\end{equation}
If $(a,b) \in \mathbb{N}^2$ is a couple such that $a+b=n$ and for some $c,d \ge 0$, we have $c+d=p$, $a+c=q$ and $b+d=r$, then there exists a unique such couple $(c,d)$, which is equal to $(q-a,r-b)$. It follows that the sum (\ref{eq-divided-2}) is equal to
\begin{equation} \label{eq-divided-3}
\underset{(a,b) \in S}{\sum}\binom{q}{a}\binom{r}{b}x^{[q]} \otimes x^{[r]}
\end{equation}
where
\begin{equation*}
S=\{(a,b)~|~\exists c,d \ge 0,~a+b=n,~c+d=p,~a+c=q,~b+d=r\}.
\end{equation*}
But it is easy to see by double inclusion that
\begin{equation*}
S=\{(a,b)~|~0 \le a \le q,~0 \le b \le r,~a+b=n\}.
\end{equation*}
We can then conclude by Vandermonde's identity (\ref{intro-before-binomial-3}) that the sum (\ref{eq-divided-3}) is equal to the sum (\ref{eq-divided-1}) which ends the proof of \cref{compatibility-intro-graded-2} for $(k_n\langle x \rangle)$. We have thus proved that the above operations make $(k_n\langle x \rangle)$ into a binomial bimonoid.

The graded objects $(k_n[x])$ and $(k_n\langle x \rangle)$ are isomorphic in $\mathsf{Vec}_k^\mathbb{N}$ since for every $n \in \mathbb{N}$, we have $k_n[x] \simeq k \simeq k_n\langle x \rangle$. Now, note that in $(k_n\langle x \rangle)$, we have 
\begin{equation*}
\nabla_{1,p-1}(x^{[1]} \otimes x^{[p-1]})=\binom{p}{1}x^{[p]}=px^{[p]}=0
\end{equation*}
so that $\nabla_{1,p-1}=0$.

But in $(k_n[x])$, the map $\nabla_{1,p-1}$ is not equal to $0$ since
\begin{equation*}
\nabla_{1,p-1}(x \otimes x^{p-1})=x^p \neq 0.
\end{equation*}

These two identites imply that the binomial $k$-bialgebras $(k_n[x])$ and $(k_n\langle x \rangle)$ are not isomorphic. Indeed, suppose that $(f_n)$ was an isomorphism of binomial $k$-bialgebras from $(k_n\langle x\rangle)$ to $k_n[x]$. Then we would have
\begin{equation*}
0=\nabla_{1,p-1};f_1 \otimes f_{p-1}=f_p;\nabla_{1,p-1}
\end{equation*}
so that
\begin{equation*}
\nabla_{1,p-1}=f_p^{-1};0=0
\end{equation*}
in $(k_n[x])$ which is false. We conclude with the following proposition.
\begin{proposition}
Let $k$ be a field of positive characteristic. There exist two non-isomorphic binomial bimonoids in $(\mathsf{Vec}_k,\otimes,k)$ whose underlying graded $k$-vector spaces are isomorphic.
\end{proposition}
\section{Examples of binomial bimonoids from permutations splittings} \label{SEC:LAST}
In this section, we use \cref{main-theorem} to compute binomial graded bimonoids from permutation splittings in three $\mathbb{Q}_{\ge 0}$-linear symmetric monoidal categories: the category of modules over any commutative $\mathbb{Q}_{\ge 0}$-algebra, the category of sets and relations and the category of suplattices.

\paragraph{In the category of modules over a commutative $\mathbb{Q}_{\ge 0}$-algebra.}
Let $R$ be a commutative rig which is a $\mathbb{Q}_{\ge 0}$-algebra. Let $M$ be any $R$-module and consider the permutation splitting from \cref{ex:1b} whose underlying graded object is $(S^nM)_{n \ge 0}$. Using \cref{main-theorem}, we obtain a binomial bimonoid with same underlying graded object. The multiplication and comultiplication maps are given by
\begin{equation*}
\nabla_{n,p}=s_{n} \otimes s_{p}; r_{n+p}
\end{equation*}
and 
\begin{equation*}
\Delta_{n,p}=\binom{n+p}{n}s_{n+p};r_{n} \otimes r_{p}.
\end{equation*}
We compute
\begin{align*}
\nabla_{n,p}( (x_1 \otimes_s \cdots \otimes_s x_n) \otimes (y_1 \otimes_s \cdots \otimes_s y_p) )
=&~ r_{n+p}\bigg(\frac{1}{n!p!}\underset{\substack{\sigma \in S_{n} \\ \tau \in S_{p}}}{\sum} x_{\sigma(1)} \otimes \cdots \otimes x_{\sigma(n)} \otimes y_{\tau(1)} \otimes \cdots \otimes y_{\tau(p)}\bigg) \\
=&~ \frac{1}{n!p!}\underset{\substack{\sigma \in S_{n} \\ \tau \in S_{p}}}{\sum} x_{\sigma(1)} \otimes_s \cdots \otimes_s x_{\sigma(n)} \otimes_s y_{\tau(1)} \otimes_s \cdots \otimes y_{\tau(p)} \\ 
=&~\frac{1}{n!p!} (n!p!) x_1 \otimes_s \cdots \otimes_s x_n \otimes_s y_1 \otimes_s \cdots \otimes y_p \\ 
=&~x_1 \otimes_s \cdots \otimes_s x_n \otimes_s y_1 \otimes_s \cdots \otimes y_p
\end{align*}
and (see the footnote for the details of third step)\footnote{Note that given $x_{1},\ldots,x_{n+p} \in M$ the value of $x_{\tau(1)} \otimes_s \cdots \otimes_s x_{\tau(n)}$ does not depend on $\tau \in S_n$. It thus makes sense to define $x_{A} := x_{a_1} \otimes_{s} \cdots \otimes_{s} x_{a_n}$ for any subset $A=\{a_1,\ldots,a_n\} \subseteq [1,n+p]$ with $n$ elements. Given such a subset $A$, there are exactly $n!p!$ permutations $\sigma \in S_{n+p}$ such that $A=\{\sigma(1),\ldots,\sigma(n)\}$ and $[1,n+p] \backslash A = \{\sigma(n+1),\ldots,\sigma(n+p)\}$. Indeed, we know that each element among $1,\ldots,n$ must be sent to some distinct value in $A$. There are $n!$ ways to realize this. Each element in $n+1,\ldots,n+p$ must be sent to some distinct value in $[1,n+p] \backslash A$. There are $p!$ ways to realize this. 

Moreover, given two distinct subsets $A,B \subseteq [1,n+p]$ each with $n$ elements, the two sets of $n!p!$ permutations in $S_{n+p}$ associated respectively to $A$ and to $B$ are disjoint because these permutations are respectively such that $\{\sigma(1),\ldots,\sigma(n)\} = A$ or such that $\{\sigma(1),\ldots,\sigma(n)\} = B$. We thus have shown that $S_{n+p}$ is partitioned into $\binom{n+p}{n}$ sets of $n!p!$ permutations, each of these sets being the set $S_{n+p,A}$ of permutations $\sigma \in S_{n+p}$ such that $\{\sigma(1),\ldots,\sigma(n)\} = A$ for some subset $A \subseteq [1,n+p]$ with $n$ elements.}
\begin{align*}
\Delta_{n,p}(x_1 \otimes_s \cdots \otimes_s x_{n+p})
=&~\binom{n+p}{n} (r_n \otimes r_p)\bigg(\frac{1}{(n+p)!} \underset{\sigma \in S_{n+p}}{\sum} x_{\sigma(1)} \otimes \cdots \otimes x_{\sigma(n+p)}\bigg) \\
=&~\frac{1}{n!p!}\underset{\sigma \in S_{n+p}}{\sum}(x_{\sigma(1)} \otimes_s \cdots \otimes_s x_{\sigma(n)}) \otimes (x_{\sigma(n+1)} \otimes_s \cdots \otimes_s x_{\sigma(n+p)}) \\ 
=&~\frac{1}{n!p!} \underset{A \in \mathcal{P}_n([1,n+p])}{\sum}~\underset{\sigma \in S_{n+p,A}}{\sum} x_{A} \otimes x_{[1,\ldots,n+p] \backslash A}\\
=&~\frac{1}{n!p!} \underset{A \in \mathcal{P}_n([1,n+p])}{\sum}(n!p!) x_{A} \otimes x_{[1,\ldots,n+p] \backslash A}\\
=&~\underset{A \in \mathcal{P}_n([1,n+p])}{\sum} x_{A} \otimes x_{[1,\ldots,n+p] \backslash A}.
\end{align*}
Moreover, the unit and counit maps are given by
\begin{equation*}
\eta=r_0=\mathsf{id}_R\colon R \rightarrow S^0M=R
\end{equation*}
and
\begin{equation*}
\epsilon=s_0=\mathsf{id}_R\colon S^0M=R \rightarrow R.
\end{equation*}
\paragraph{In the category of sets and relations.}
Let $X$ be any set and consider the permutation splitting in $(\mathsf{Rel},\times,*)$ from \cref{ex:2b} whose underlying graded object is $(\mathcal{M}_nX)_{n \ge 0}$.
We compute the binomial bimonoid obtained from \cref{main-theorem}:
\begin{align*}
&~\nabla_{n,p} \\
=&~s_{n} \times s_{p} ; r_{n+p} \\
=&~\Big(\{([x_1,\ldots,x_n], (x_{\sigma(1)},\ldots,x_{\sigma(n)})),~x_1,\ldots,x_n \in X, \sigma \in S_n \}  \\
&\times \{([y_1,\ldots,y_p], (y_{\tau(1)},\ldots,y_{\tau(p)})),~y_1,\ldots,y_p \in X, \tau \in S_p \}\Big); \\
&\{ ((w_1,\ldots,w_{n+p}),[w_1,\ldots,w_{n+p}]),~w_1,\ldots,w_{n+p} \in X \} \\
=&~\{ (([x_1,\ldots,x_n],[y_1,\ldots,y_p]), (x_{\sigma(1)},\ldots,x_{\sigma(n)},y_{\tau(1)},\ldots,y_{\tau(p)})) ,~ \\
&x_1,\ldots,x_n,y_1,\ldots,y_p \in X, \sigma \in S_n, \tau \in S_p\};\{ ((w_1,\ldots,w_{n+p}),[w_1,\ldots,w_{n+p}]),~w_1,\ldots,w_{n+p} \in X \} \\
=&~\{ (([x_1,\ldots,x_n],[y_1,\ldots,y_p]), [x_1,\ldots,x_n,y_1,\ldots,y_p]),~x_1,\ldots,x_n,y_1,\ldots,y_{n+p} \in X \}.
\end{align*}

In words: $\nabla_{n,p}$ relates a bag of $n$ elements and a bag of $p$ elements to the bag of all these $n+p$ elements.  We also define

\begin{align*}
\Delta_{n,p}=&~ \binom{n+p}{n} s_{n+p};r_n \times r_p \\
=&~s_{n+p};r_n \times r_p \\
=&~ \{ ([x_1,\ldots,x_{n+p}],(x_{\sigma(1)},\ldots,x_{\sigma(n+p)})),~x_1,\ldots,x_{n+p} \in X,
 \sigma \in S_{n+p} \}; \\
&\{ ((y_1,\ldots,y_n),[y_1,\ldots,y_n]),~y_1,\ldots,y_n \in X \} \times \{ ((w_1,\ldots,w_p),[w_1,\ldots,w_p]),~w_1,\ldots,w_p \in X \} \\
=&~\{ ([x_1,\ldots,x_{n+p}], (x_{\sigma(1)},\ldots,x_{\sigma(n+p)})),~x_1,\ldots,x_{n+p} \in X, \sigma \in S_{n+p} \}; \\
&\{ ((y_1,\ldots,y_n,w_1,\ldots,w_p),([y_1,\ldots,y_n],[w_1,\ldots,w_p])),~y_1,\ldots,y_n,w_1,\ldots,w_p \in X \} \\
=&~\{ ([x_1,\ldots,x_{n+p}], ([x_{\sigma(1)},\ldots,  x_{\sigma(n)}],[x_{\sigma(n+1)},\ldots,x_{\sigma(n+p)}])),~x_1,\ldots,x_{n+p} \in X, \sigma \in S_{n+p}\}.
\end{align*}

In words: $\Delta_{n,p}$ relates a bag of $n+p$ elements to all the pairs of a bag of $n$ elements and a bag of $p$ elements obtained by dividing the bag of $n+p$ elements into two bags.

The unit and counit maps are identities:
\begin{equation*}
\eta=r_0=\mathsf{id}_*\colon *\rightarrow \mathcal{M}_0(X)=*
\end{equation*}
and
\begin{equation*}
\epsilon=s_0=\mathsf{id}_*\colon\mathcal{M}_0(X)=* \rightarrow *.
\end{equation*}
\paragraph{In the category of suplattices.} Let $X$ be a suplattice and consider the permutation splitting in $(\mathsf{Sup},\otimes,\mathcal{P}(*))$ from \cref{ex:3b} whose underlying graded object is $(S^nX)_{n \ge 0}$. We can compute the binomial bimonoid obtained from \cref{main-theorem}. The computations are essentially the same as in the case of $(\mathsf{Mod}_R,\otimes,R)$ for $R$ a commutative $\mathbb{Q}_{\ge 0}$-algebra.

We obtain that
\begin{equation*}
\nabla_{n,p}( (x_1 \otimes_s \cdots \otimes_s x_n) \otimes (y_1 \otimes_s \cdots \otimes_s y_p) )=x_1 \otimes_s \cdots \otimes_s x_n \otimes_s y_1 \otimes_s \cdots \otimes_s y_p
\end{equation*}
and
\begin{equation*}
\Delta_{n,p}(x_1 \otimes_s \cdots \otimes_s x_{n+p})=\underset{A \in \mathcal{P}_n([1,n+p])}{\bigvee} x_{A} \otimes x_{[1,\ldots,n+p] \backslash A}.
\end{equation*}
The unit and counit maps are identities:
\begin{equation*}
\eta=\mathsf{id}_{\mathcal{P}(*)}\colon \mathcal{P}(*) \rightarrow S^0M=\mathcal{P}(*)
\end{equation*}
and
\begin{equation*}
\epsilon=\mathsf{id}_\mathcal{P}(*)\colon S^0M=\mathcal{P}(*) \rightarrow \mathcal{P}(*).
\end{equation*}
\bibliographystyle{plainurl}
\bibliography{biblioyet-another-sym}
\appendix
\section{The initial pointed $\mathbb{Q}_{\ge 0}$-linear symmetric strict monoidal category} \label{app:initial}
We will define in this appendix the initial object $((\mathcal{I},\otimes,0),1)$ in the category of pointed $\mathbb{Q}_{\ge 0}$-linear symmetric strict monoidal categories.

The objects of the category $\mathcal{I}$ are the nonnegative integers. For all $n,p \in \mathbb{N}$, we set 
\begin{equation*}
\mathcal{I}(n,p)=
\left\{
\begin{aligned}
\{0\} &~\text{if }n\neq p, \\
\mathbb{Q}_{\ge 0}[S_n] &~\text{if }n=p
\end{aligned}
\right.
\end{equation*}
where $\mathbb{Q}_{\ge 0}[S_n]$ is the group $\mathbb{Q}_{\ge 0}$-algebra of $S_n$. The elements of $\mathbb{Q}_{\ge 0}[S_n]$ are formal linear combinations
\begin{equation*}
\underset{\sigma \in S_n}{\sum}\lambda_\sigma\sigma
\end{equation*}
where $\lambda_\sigma \in \mathbb{Q}_{\ge 0}$ for every $\sigma \in S_n$. It follows that $\mathbb{Q}_{\ge 0}[S_n]$ is a $\mathbb{Q}_{\ge 0}$-module. We will make $\mathbb{Q}_{\ge 0}[S_n]$ into a $\mathbb{Q}_{\ge 0}$-algebra by defining a multiplication on $\mathbb{Q}_{\ge 0}[S_n]$. This multiplication is defined by the formula
\begin{equation*}
\Big(\underset{\sigma \in S_n}{\sum}\lambda_\sigma\sigma\Big)\Big(\underset{\rho \in S_n}{\sum}\mu_\rho\rho\Big)=\underset{(\lambda,\sigma) \in {S_n}^2}{\sum}\lambda_\sigma\mu_\rho(\sigma;\rho).
\end{equation*}
The multiplicative unit is $1 \in S_n$. The composition map
\begin{equation*}
\begin{aligned}
\mathcal{I}(n,p) \times \mathcal{I}(p,q) &\longrightarrow \mathcal{I}(n,q) \\
(f,g) &\longmapsto f;g
\end{aligned}
\end{equation*}
is then defined as follows:
\begin{equation*}
f;g=
\left\{
\begin{aligned}
0 &~\text{if }n\neq p \text{ or }p \neq q, \\
fg &~\text{if }n=p=q.
\end{aligned}
\right.
\end{equation*}
The identity on $n \in \mathbb{N}$ is the multiplicative unit $1 \in S_n$.

We still have to define the symmetric monoidal structure on $\mathcal{I}$. The monoidal unit is $0 \in \mathbb{N}$ and the tensor product of objects is given by $n \otimes p:=n+p$. If $\sigma \in S_n$ and $\rho \in S_p$, we define a new permutation $\sigma \otimes \rho \in S_{n+p}$ by the formula
\begin{equation*}
(\sigma \otimes \rho)(k)=
\left\{
\begin{aligned}
\sigma(k) &~\text{if } 1 \le k \le n, \\
\rho(k-n)+n &~\text{if } n+1 \le k \le n+p. 
\end{aligned}
\right.
\end{equation*}
We extend this operation to a map
\begin{equation*}
\begin{aligned}
\mathbb{Q}_{\ge 0}[S_n] \times \mathbb{Q}_{\ge 0}[S_p] &\longrightarrow \mathbb{Q}_{\ge 0}[S_{n+p}] \\
\Big(\underset{\sigma \in S_n}{\sum}\lambda_\sigma\sigma,\underset{\rho \in S_p}{\sum}\mu_\rho\rho\Big) &\longmapsto f \otimes g:=\underset{(\lambda,\sigma) \in S_n \times S_p}{\sum}\lambda_\sigma\mu_\rho(\sigma \otimes \rho).
\end{aligned}
\end{equation*}
We also define
\begin{equation*}
f \otimes 0=0 \otimes f:=0 \in \mathcal{I}(n+q,p+r)
\end{equation*}
for all $n,p \ge 0$, $f \in \mathcal{I}(n,p)$ and $0 \in \mathcal{I}(q,r)$ where $q \neq r$. Finally, the exchange $\gamma_{n,p} \in \mathcal{I}(n+p,n+p)$ is given by
\begin{equation*}
\gamma_{n,p}(k)=
\left\{
\begin{aligned}
k+p &~\text{if }1 \le k \le n, \\
k-n &~\text{if }n+1 \le k \le n+p.
\end{aligned}
\right.
\end{equation*}
Checking that $(\mathcal{I},\otimes,0)$ satisfies the appropriate equations to be a $\mathbb{Q}_{\ge 0}$-linear symmetric strict monoidal category is then straightforward.
\begin{definition}
Let $(\mathsf{C},\otimes,I)$ be a $\mathbb{Q}_{\ge 0}$-linear symmetric strict monoidal category and let $(\mathsf{D},\otimes,I)$ be a $\mathbb{Q}_{\ge 0}$-linear symmetric monoidal category. A \emph{$\mathbb{Q}_{\ge 0}$-linear symmetric strong monoidal functor} is a functor $F\colon\mathsf{C} \rightarrow \mathsf{D}$ together with a natural isomorphism
\begin{equation*}
\mu_{A,B}\colon F(A) \otimes F(B) \simeq F(A \otimes B)
\end{equation*}
and an isomorphism
\begin{equation*}
u\colon I \simeq F(I)
\end{equation*}
such that the function
\begin{align*}
\mathsf{C}(A,B) &\longrightarrow \mathsf{D}(F(A),F(B)) \\
f &\longmapsto F(f)
\end{align*}
is $\mathbb{Q}_{\ge 0}$-linear for all $A,B \in \mathsf{C}$, and the following diagrams commute:
\begin{equation*}
% https://tikzcd.yichuanshen.de/#N4Igdg9gJgpgziAXAbVABwnAlgFyxMJZABgBoAmAXVJADcBDAGwFcYkQAxACgEEACADoCIeALbw+AIQCUg4WIncAwtJABfUuky58hFOQrU6TVu278hIrOLhS5Vm3xXrNIDNjwEiZYkYYs2RBAuc1lLBVtuGTD5a0UuZw0tD10iA18af1Mg0PsIvhCuGTy4yITpVSS3bU89ZAMARj8TQM5eGId4yRLHRKMYKABzeCJQADMAJwhRJABmGhwIJAAWTJb2IVF6HAALODHgLCg1HolN5hdxqZnEVZBFpAa1gI2BUQuqyemkA3ulxDIxheQXOp1sm22ewORzUIBojHoACMYIwAAo1VJBCZYQY7HCXEBfG6Ah6IJ5A7Igc5wkAI5FojFeLE4vEEok-Bb-eYU1pCJhoHb0dSUNRAA
\begin{tikzcd}
(F(A) \otimes F(B)) \otimes F(C) \arrow[dd, "\mu \otimes \mathsf{id}"'] \arrow[rr, "\alpha"] &  & F(A) \otimes (F(B) \otimes F(C)) \arrow[d, "\mathsf{id} \otimes \mu"] \\
                                                                                             &  & F(A) \otimes F(B \otimes C) \arrow[d, "\mu"]                          \\
F(A \otimes B) \otimes F(C) \arrow[rr, "\mu"']                                               &  & F(A \otimes B \otimes C)~,                                             
\end{tikzcd}
\end{equation*}
\begin{equation*}
% https://tikzcd.yichuanshen.de/#N4Igdg9gJgpgziAXAbVABwnAlgFyxMJZABgBpiBdUkANwEMAbAVxiRAEkACAHW4jwC28TgDEAFAEEAlCAC+pdJlz5CKAIzkqtRizbj2Unn0HDx0uQpAZseAkQ1qt9Zq0QgzM+YpsqiAZk1qZ103DyN+LCE4TnYLb2U7FAAWQO0XPUlDXgio0TEDOKslW1VkFMcgnVd3TLktGCgAc3giUAAzACcIASQyEBwIJAAmSvS3XgY6AQAjKDoQaknpmAYABWLfNw6sRoALHELO7t7qAaQNNJCQJnCTaN4BOhxduDbgLChZQ66exAuzxAjS7VB5Mb7HRABfqDRAAVlGV14HV2g0WdGWaw2iRA2z2By8ICOvyhAJSwLYDyeLzeH1kt0iwjBBKJSDJAPh5PG3AETIosiAA
\begin{tikzcd}
I \otimes F(A) \arrow[rd, "\lambda"'] \arrow[r, "u \otimes \mathsf{id}"] & F(I) \otimes F(A) \arrow[d, "\mu"] &  & F(A) \otimes I \arrow[rd, "\rho"'] \arrow[r, "\mathsf{id} \otimes u"] & F(A) \otimes F(I) \arrow[d, "\mu"] \\
                                                                         & F(A)~,                              &  &                                                                       & F(A)~,                            
\end{tikzcd}
\end{equation*}
\begin{equation*}
% https://tikzcd.yichuanshen.de/#N4Igdg9gJgpgziAXAbVABwnAlgFyxMJZABgBpiBdUkANwEMAbAVxiRADEAKAQQEoACADqCIeALbx+XAEK8QAX1LpMufIRQBGclVqMWbLtyEjxk2QqUgM2PASJaNO+s1aIOnacdFYJcfnwtlGzUiMkdqZ303GQFhb18pHjl5HRgoAHN4IlAAMwAnCDEkMhAcCCQtXRc2YTEmQJB8worqMqQAJgi9V3dhdLoxMTpkyyaixBK2xABmLuq3PoGhkGoGOgAjGAYABRVbdRA8rHSACxwGsaRZ0vLETqqokFr61Y2t3eC7NyPT85T5IA
\begin{tikzcd}
F(A) \otimes F(B) \arrow[r, "\mu"] \arrow[d, "\gamma"'] & F(A \otimes B) \arrow[d, "F(\gamma)"] \\
F(B) \otimes F(A) \arrow[r, "\mu"']                     & F(B \otimes A)~.                      
\end{tikzcd}
\end{equation*}
\end{definition}
We define a \emph{pointed $\mathbb{Q}_{\ge 0}$-linear symmetric strict monoidal category} as a couple $((\mathsf{C},\otimes,I),A)$ where $(\mathsf{C},\otimes,I)$ is a $\mathbb{Q}_{\ge 0}$-linear symmetric strict monoidal category and $A \in \mathsf{C}$. A \emph{morphism of pointed $\mathbb{Q}_{\ge 0}$-linear symmetric strict monoidal categories} from $((\mathsf{C},\otimes,I),A)$ to $((\mathsf{D},\otimes,I),B)$ is a strong monoidal functor $F\colon (\mathsf{C},\otimes,I) \rightarrow (\mathsf{D},\otimes,I)$ such that $F(A)=B$.
\begin{proposition} \label{cons-mac}
$((\mathcal{I},\otimes,0),1)$ is the initial object in the category of pointed $\mathbb{Q}_{\ge 0}$-linear symmetric strict monoidal categories.
\end{proposition} 
If we want to work with a non-strict monoidal category in the domain of our unique $\mathbb{Q}_{\ge 0}$-linear symmetric strong monoidal functor, we have the following proposition which is a consequence of MacLane's coherence theorem on symmetric monoidal categories.
\begin{proposition}
Let $(\mathsf{C},\otimes,I)$ be a $\mathbb{Q}_{\ge 0}$-linear symmetric monoidal category and let $A \in \mathsf{C}$. There exists a unique $\mathbb{Q}_{\ge 0}$-linear symmetric strong monoidal functor
\begin{equation*}
F\colon(\mathcal{I},\otimes,0) \rightarrow (\mathsf{C},\otimes,I)
\end{equation*}
such that $F(n)=A^{\otimes n}$ for every $n \ge 0$ where $A^{\otimes n}$ is defined by $A^{\otimes 0}:=I$, $A^{\otimes 1}:=A$ and $A^{\otimes (n+1)}:=A^{\otimes n} \otimes A$ for every $n \ge 1$.
\end{proposition}
\end{document}